\documentclass{article}

\usepackage{arxiv}

\usepackage[hidelinks]{hyperref}

\usepackage{babel}
\usepackage[]{geometry}
\DeclareMathAlphabet\mathbfcal{OMS}{cmsy}{b}{n}

\usepackage{amsmath,amsthm,amssymb,babel,amsfonts,graphics, mathtools}

\usepackage{booktabs}

\usepackage{algorithm}
\usepackage{algpseudocode}

\usepackage{caption}
\usepackage{subcaption}

\usepackage{amsmath,amssymb,amsfonts,latexsym,cancel}
\usepackage{rawfonts}
\usepackage{pictexwd}
\usepackage{tikz}

\usepackage{tikz-cd}

\usepackage{pgfplots}

\pgfplotsset{compat=newest}
\pgfplotsset{plot coordinates/math parser=false}
\newlength\figureheight
\newlength\figurewidth

\usepackage{color}
\usepackage{epsfig}

\newcommand{\R}{\mathbb{R}}
\newcommand{\N}{\mathbb{N}}
\newcommand{\Z}{\mathbb{Z}}

\def\1{\raisebox{2pt}{\rm{$\chi$}}}

\theoremstyle{plain}
\newtheorem{definition}{Definition}

\newtheorem{theorem}{Theorem}
\newtheorem{corollary}{Corollary}
\newtheorem{lemma}{Lemma}

\newtheorem{assumption}{Assumption}

\theoremstyle{definition}

\theoremstyle{remark}

\numberwithin{equation}{section}

\begin{document}

\title{Finite-difference least square methods for solving Hamilton-Jacobi equations using neural networks}

\author{ 
Carlos Esteve-Yag\"ue* \\
	Departamento de Matem\'aticas \\
    Universidad de Alicante \\
    03690 San Vicente del Raspeig, Alicante, Spain \\
	\texttt{c.esteve@ua.es} \\
 *Corresponding author
\And
Richard Tsai \\
	Department of Mathematics and Oden Institute \\
	for Computational Engineering and Sciences\\
	The University of Texas at Austin\\
	Austin, USA \\
	\texttt{ytsai@math.utexas.edu} \\
\And
Alex Massucco\\
	Department of Applied Mathematics \\
	and Theoretical Physics,
	University of Cambridge\\
 	Cambridge, UK \\
	\texttt{am3270@cam.ac.uk} \\
}

\date{\today}

\maketitle

\begin{abstract}
We present a simple algorithm to approximate the viscosity solution of Hamilton-Jacobi~(HJ) equations using an artificial deep neural network. The algorithm uses a stochastic gradient descent-based method to minimize the least square principle defined by a monotone, consistent numerical scheme. We analyze critical points of the least square principle in function space and derive a mesh-independent sufficient condition for their residuals to vanish within a prescribed Lipschitz class. Under the usual convergence assumptions, the resulting finite-difference solutions approximate the viscosity solution. Using a deep artificial neural network lifts the restriction of conventional finite difference methods that rely on computing functions on a fixed grid. This feature makes it possible to solve HJ equations posed in higher dimensions where conventional methods are infeasible. We demonstrate the efficacy of our algorithm through numerical studies on various canonical HJ equations across different dimensions, showcasing its potential and versatility. 
\end{abstract}



\section{Introduction}
In recent years, significant research has focused on computing solutions for partial differential equations in high dimensions. These endeavours are driven by advancements in deep learning and new numerical optimisation algorithms. 
In particular, there is great interest in numerical algorithms for Hamilton-Jacobi equations arising in optimal control problems and differential games.

We consider the boundary value problem associated with a Hamilton-Jacobi equation of the form
\begin{equation}
    \label{BVP intro}
    \begin{cases}
    H(x , \nabla u(x) ) = 0 & x\in  \Omega \\
    u (x) = g(x) & x\in \partial\Omega,
    \end{cases}
\end{equation}
where $\Omega\subset \R^d$, with $d\geq 1$, is an open bounded domain with Lipschitz boundary, and
$$
H:\Omega \times \R^d \to \R
\quad \text{and} \quad
g: \partial\Omega \to \R
$$
are the given functions, which we shall refer to as the Hamiltonian and the boundary condition, respectively.
Throughout the paper, we shall assume that 
\begin{equation}
    \label{cond g}
    x\mapsto g(x) \ \text{is Lipschitz on $\partial \Omega$}
\end{equation}
and
\begin{equation}
\label{cond H}
    (x,p) \mapsto H (x,p) \ \text{is locally Lipschitz in $\Omega\times \R^d$ and differentiable with respect to $p$.}
\end{equation}

Certain types of Hamilton-Jacobi equations describe the optimality condition for an optimal control problem. By applying Bellman's dynamic programming principle, one can show that the value function
of the associated optimal control problem formally satisfies the so-called Hamilton-Jacobi-Bellman equation (see \cite{Bardi-Capuzzo-Dolcetta:1997}).
However, it is well-known that even if the boundary condition $g$ is very smooth, the boundary value problem \eqref{BVP intro} is not guaranteed a classical solution. This happens when the characteristics of \eqref{BVP intro}, representing the optimal trajectories of the optimal control problem, collide with each other.
On the other hand, continuous functions that satisfy the PDE for almost every $x\in \Omega$ along with the boundary condition are not necessarily unique.
In \cite{CL-viscosity:1983}, Crandall and Lions introduced the notion of viscosity solution to single out a unique weak solution for the problem \eqref{BVP intro}. Furthermore, in many situations, one can prove that the viscosity solution coincides with the value function, and one may use it to synthesize the optimal feedback for the control problem.

Under hypotheses ensuring solvability, stability and comparison, the \textit{vanishing viscosity method} approximates the viscosity solution by solutions to the problem
\begin{equation}
\label{BVP-visc intro}
\begin{cases}
    H(x , \nabla u^\varepsilon (x) ) - \varepsilon \Delta u^\varepsilon (x) = 0, & x\in  \Omega, \\
    u^\varepsilon (x) = g(x), & x\in \partial\Omega,
\end{cases}
\end{equation}
for $\varepsilon>0$. Under these additional hypotheses, the viscosity solution to \eqref{BVP intro} is obtained as the locally uniform limit of $u^\varepsilon$ as $\varepsilon\to0^+$.
A different characterisation of the viscosity solution exists, using smooth test functions based on the comparison principle (see \cite{crandall1992user}).
Depending on the structure of the Hamiltonian and the boundary data, the viscosity solution may be discontinuous (see \cite{Bardi-Capuzzo-Dolcetta:1997,ishii1987perron}). This is typically true for Hamilton--Jacobi equations arising in differential game theory or when the boundary condition is prescribed along a characteristic curve. However, in this paper, we make the following assumption:

\begin{assumption}
\label{assump: Lipschitz}
   The boundary value problem~\eqref{BVP intro} has a Lipschitz continuous viscosity solution. 
\end{assumption}

The design and implementation of numerical algorithms to approximate the viscosity solution of \eqref{BVP intro} is a classical problem widely considered throughout the years.
In two or three dimensions, Hamilton-Jacobi equations are typically solved by special finite-difference methods operating on uniform Cartesian grids. 
The Hamiltonian $H$ is approximated by a so-called numerical Hamiltonian $\widehat H$ defined on the grid and the grid function. 
For schemes that are \emph{consistent} on smooth test functions and \emph{monotone} (see Definition~\ref{def: consistent and monotone scheme}), stability, a consistent treatment of boundary data, and a comparison principle for the limiting problem yield uniform convergence of the grid solutions, suitably interpolated on $\overline{\Omega}$, to the viscosity solution as the grid spacing tends to zero (see e.g. \cite{crandall1984two, BS:1991}).
Similarly to the development of numerical schemes for hyperbolic conservation laws, higher-order approximations to the partial derivatives are proposed to be used with a monotone and consistent numerical Hamiltonian. In parallel, fast algorithms such as the fast marching methods \cite{Tsitsiklis:1995,Sethian:1996, sethian-book:1999,Helmsen:1996} and fast sweeping methods \cite{zhao2005fast,Tsai-Cheng-Osher-Zhao:2001} are developed for a certain class of Hamilton-Jacobi equations to solve the resulting nonlinear system with computational complexities that are linear in $N_d$, the total number of unknowns (for the fast marching methods, the formal complexity is $N_d\log N_d$). 

The main drawback of finite-difference methods implemented in the classical fashion is that they involve solving a system of nonlinear equations with as many unknowns as the number of grid points. 
This means that the complexity of these methods scales at least exponentially with the dimension of the domain.
Therefore, the classical method of formulating and solving for finite difference approximations on grid functions cannot be carried out in the high-dimensional setup.
However, many Hamilton-Jacobi-Bellman equations for optimal control applications are posed in high dimensions. 
Recently, there has been substantial development 
in algorithms for solving high dimensional problems utilising ideas from the Hopf-Lax formula and its generalisation; see, e.g. 
\cite{darbon2016algorithms},\cite{chow2019algorithm}.
{Other algorithms \cite{dolgov2021tensor, kalise2018polynomial, dolgov2023data} use tensor decomposition to address special classes of high-dimensional Hamilton-Jacobi equations arising in optimal control theory.}

In recent years, Deep Learning approaches have proven very effective at solving some non-linear problems in high dimensions.
These methods consist of approximating a target function (in our case, the viscosity solution) by an artificial neural network, whose parameters are optimised by minimising a loss function (typically) through a gradient-based method.
This can be formulated as a variational problem of the form
\begin{equation}
\label{general DL problem}
{u}^* \in \arg \min_{u\in \mathcal{F}} \mathcal{J} (u), 
\end{equation}
where 
$$
\mathcal{F} := \{ u(\cdot, \theta) \, : \ \theta \in \R^P \} \subset C(\overline{\Omega})
$$ 
is a parameterized class of functions given by a Neural Network architecture, and $\mathcal{J}:C(\overline{\Omega}) \to \R^+$ is the loss functional. 
Of course, the success of this method heavily relies on the choice of $\mathcal{F}$, $\mathcal{J}(\cdot)$ and the optimisation algorithm.

Ideally, $\mathcal{J}(u)$ should yield smaller values for functions $u$ close to the target function.
A common practice in Deep Learning is to seek a
minimizer of $\mathcal{J} (\cdot)$ in $\mathcal{F}$
using information derived from $\nabla_\theta \mathcal{J}(\Phi(\cdot, \theta)).$
Let us consider the following continuous version of the gradient descent method to find a minimizer of $\mathcal{J}$ in $\mathcal{F}$:
\begin{equation}
\label{cont grad descent}
    \frac{d}{dt}\theta_t = -\nabla _\theta \mathcal{J}(u(\cdot, \theta_t)) = -\mathcal{J}^\prime(u(\cdot,\theta_t)) \,  \nabla_\theta u(\cdot,\theta_t),
\end{equation}
where, for any continuous function $u\in C(\overline{\Omega})$, $\mathcal{J}^\prime(u)\in C (\overline{\Omega})^\ast$ denotes the Fréchet derivative of $\mathcal{J} (\cdot)$ at $u$, and $\nabla_\theta u(\cdot, \theta)$ denotes the vector--valued function with the partial derivatives of $u(x,\theta)$ with respect to the parameter $\theta$. 
We see that a parameter $\theta\in \R^P$ is a stationary point of \eqref{cont grad descent} if and only if all the elements of the vector $\nabla_\theta u (\cdot, \theta) \in [C(\overline{\Omega})]^P$ lie in the kernel of the linear functional $\mathcal{J}^\prime(u (\cdot, \theta))$.
The precise form of the vector-valued function $\nabla_\theta u(\cdot, \theta)$ depends on the architecture of the Neural Network and will not be discussed in this paper.
On the other hand, any parameter $\theta\in \R^P$ such that the function $u = u (\cdot, \theta)$ is a solution to the Euler-Lagrange equation $\mathcal{J}^\prime(u)=0$ is a critical point of $\theta \mapsto \mathcal{J}(u(\cdot, \theta))$, and therefore, a stationary point for the gradient descent method.

In this paper, we investigate the choice of the functional $\mathcal{J}(\cdot)$.
Namely, we are interested in constructing a functional with the two following properties:
\begin{enumerate}
    \item Any global minimiser of $\mathcal{J}(u)$ approximates the viscosity solution of \eqref{BVP intro}.
    \item Any critical point of $\mathcal{J} (u)$ is a global minimizer.
\end{enumerate}
Note that the second property above is crucial in Deep Learning since the optimisation algorithm used to minimise $\mathcal{J}(u)$ is typically a variant of gradient descent, similar to \eqref{cont grad descent}.
In particular, we shall look at
functionals with the structure 
\begin{equation}
    \label{functional def intro}
    \mathcal{J} (u) = \mathcal{L} (u; g) + \mathcal{R} (u),
\end{equation}
where $\mathcal{L}(u; g)$ and $\mathcal{R}(u)$ are, respectively, fitting terms related to the boundary data $g$ and the PDE in \eqref{BVP intro}, {with the following specific form:
$$
\mathcal{L} (u;g) \coloneq \int_{\partial \Omega} \left( u(x) - g(x) \right)^2 dx
\quad \text{and} \quad
\mathcal{R} (u) := \int_\Omega \left( \widehat{H}(x, D^+_\delta u(x), D^-_\delta u(x)) \right)^2 dx,
$$
where $\widehat{H}$ is a numerical Hamiltonian, and $D^+_\delta u(x)$ and $D^-_\delta u(x)$ are the forward and backward finite-difference approximations of the gradient $\nabla u(x)$.}
The main contribution of the current paper is a sufficient condition, within a prescribed Lipschitz class, for critical points of the Lax-Friedrichs residual functional to have zero residual. Combined with compatible boundary data and the usual convergence assumptions for monotone schemes, this connects the two properties above.

We point out that other variational approaches exist for solving Hamilton-Jacobi equations.
In \cite{meng2023primal, meng2024primal}, the viscosity solution of a Hamilton-Jacobi equation is identified with the saddle-point of a min-max problem, whose solution is addressed through a primal-dual hybrid gradient algorithm.
In contrast, in developing our proposed formulation, we aim to benefit from the optimisation techniques commonly used in Deep Learning.

{
In practice, the optimisation problem \eqref{general DL problem} is addressed by considering a discrete version of the gradient descent method \eqref{cont grad descent}, applied to a discrete approximation of the functional $\mathcal{J}(u)$ in \eqref{functional def intro}.
In this paper, we propose an algorithm that leverages the stochastic gradient descent algorithm to minimize $\mathcal{J}$ 
in the form
\begin{equation}
    \theta_{n+1}:=\theta_n - \eta_n \nabla_\theta J_n(u(\cdot, \theta_n)),
\end{equation}
where $J_n$ is a Monte-Carlo approximation of the integral terms in $\mathcal{J}$, involving a batch of randomly positioned ``collocation points'' $B_n \subset \overline{\Omega}$. We refer to the strategy of changing $B_n$ at every iteration as \emph{resampling} of the domain. 
This is a crucial step to the success of the proposed algorithm.
}

In the next section, we illustrate our findings with a simple example.
In particular, we consider the Eikonal equation in a one-dimensional interval. First, we compute the optimality condition for a functional based on the PDE residual. Then, we carry out the same computation for a functional based on the residual of a finite-difference approximation of the Hamiltonian.
In Section~\ref{sec: PINNs}, we discuss existing approaches based on Deep Learning for the numerical approximation of PDEs,
such as the well-known framework known as \textit{Physics informed Neural Networks}, or PINNs.
In Section~\ref{sec: main results}, we study the critical points of loss functionals based on numerical Hamiltonians of Lax-Friedrichs type and present our main theoretical contributions.
Sections~\ref{sec: numerics} and \ref{sec: analysis} are devoted to the numerical aspects of our approach. In Section~\ref{sec: numerics}, we describe our algorithm to approximate the viscosity solution employing a Neural Network, and then, we test our method in several examples of Hamilton-Jacobi equations.
{In Section~\ref{sec: analysis} we analyse, through numerical experiments, various aspects of our algorithm, such as data efficiency and accuracy. We also show how the resampling strategy dramatically outperforms other approximation strategies, in which the collocation points are reused throughout the gradient descent iterations.}
In Section~\ref{sec: finite-difference proof}, we present the proofs of the theoretical results presented in Section~\ref{sec: main results}. Finally, we conclude the paper in Section~\ref{sec: summary and perspectives}, summarising our contributions and some perspectives for future research.

The code to reproduce the numerical experiments from Sections~\ref{sec: numerics} and \ref{sec: analysis} is publicly available in the following link:

\url{https://github.com/carlosesteveyague/HamiltonJacobi_LeastSquares_LxF_NNs}

\section{Example}
\label{sec:example}
Consider the model viscous Eikonal equation in the unit interval, with Dirichlet boundary conditions:
\begin{equation}\label{eq:model toy problem}
\begin{cases}
    (\partial_x u)^2 =1+\varepsilon \partial_{xx} u & x\in(0,1), \\
    u(0) = u_0, \quad u(1) = u_1.
\end{cases}
\end{equation}
It is well-known that, for any $\varepsilon\neq 0$ and $u_0,u_1\in \R$, the problem \eqref{eq:model toy problem} has a unique $C^\infty$-solution.

Let us consider the minimisation problem
\begin{equation}\label{J_eps}
\min_u \mathcal{R}_\varepsilon(u) := \dfrac{1}{2} 
\int_0^1 \left( (\partial_x u(x))^2 - \varepsilon \partial_{xx} u(x) -1  \right)^2 dx,
\end{equation}
among functions in $C^\infty ([0,1])$ satisfying the boundary conditions $u(0) = u_0$ and $u(1) = u_1$. 

For any $u\in C^\infty ([0,1])$, let us define the PDE-residual function $w(x) = (\partial_x u(x))^2 - \varepsilon \partial_{xx}u(x)-1$. For any test function $v\in C^\infty([0,1])$ with $v(0)=v(1)=0$, the G\^ateaux derivative of $\mathcal{R}_\varepsilon (\cdot)$ at $u$ in the direction $v$ can be written as
\begin{eqnarray*}
    D_v \mathcal{R}_\varepsilon (u) &=& \int_0^1
    \underbrace{\left( (\partial_x u)^2 - \varepsilon \partial_{xx} u -1  \right)}_{= w} \left( 2\partial_x u \partial_x v - \varepsilon \partial_{xx}v \right) dx \\
    &=& - 2\int_0^1 v \left( \partial_x (w \partial_x u)+ {\frac{\varepsilon}{2}} \partial_{xx} w\right) dx
    - { \varepsilon} \left[ w(1) \partial_x v(1) - w(0) \partial_x v(0)\right].
\end{eqnarray*}
For $\varepsilon\neq0$, suppose that $u^*\in C^\infty ([0,1])$ is a critical point of $\mathcal{R}_\varepsilon(\cdot)$. Then, $w$ is a weak solution of 
$$
\begin{cases}
    \partial_x ( w\partial_x u^* ) + {\dfrac{\varepsilon}{2}} \partial_{xx} w = 0 & \text{in} \ (0,1), \\
    w(0) = w(1) = 0.
\end{cases}
$$
It is clear that, for any $\varepsilon \neq 0$, the unique weak solution to this boundary value problem is $w=0$. Therefore any critical point $u^*\in C^\infty ([0,1])$ of $\mathcal{R}_\varepsilon(\cdot)$, satisfying the boundary conditions
is a classical solution of
$$
(\partial_x u)^2 - \varepsilon \partial_{xx} u -1 =0.
$$
Since this PDE, coupled with boundary conditions $u(0) = u_0$ and $u(1) = u_1$, has a unique classical solution, we deduce that this solution is the unique critical point of $\mathcal{R}_\varepsilon(\cdot)$ in $C^\infty ([0,1])$.

For homogeneous boundary data $u_0=u_1=0$, the solution is concave when $\varepsilon>0$ and convex when $\varepsilon<0$.
The viscosity solution of the inviscid equation is the limit as $\varepsilon\to0^+$.
For $\varepsilon=0$ and these boundary data, there is no classical solution, and the infimum of $\mathcal{R}_0$ over smooth admissible functions is zero but is not attained. Nevertheless, smooth critical points with nonzero residual do exist: $u\equiv0$ is a critical point with $\mathcal{R}_0(0)=1/2$. The residual boundary conditions derived above apply only when $\varepsilon\neq0$.
If the admissible class is enlarged to Lipschitz functions, there are multiple zero-residual minimisers of $\mathcal{R}_0$, as illustrated in Figure~\ref{fig:Eikonal non-uniqueness}.

\begin{figure}
    \centering
    \includegraphics[scale = .28]{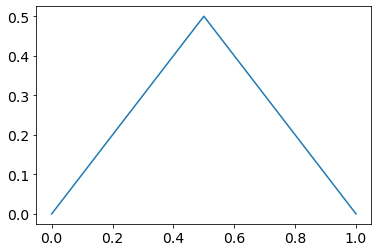}
    \includegraphics[scale = .28]{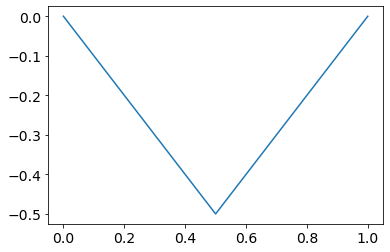}
    \includegraphics[scale = .28]{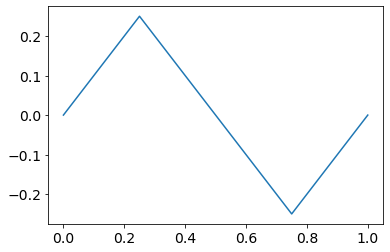}
    \includegraphics[scale = .28]{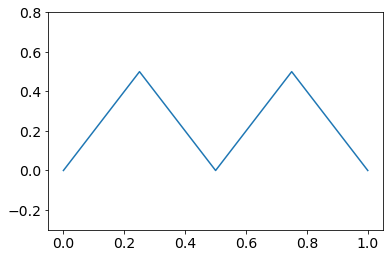}
    \caption{Four weak solutions of the Eikonal equation $(  \partial_x u )^2 -1 =0$ in $\Omega := (0, 1)$ with boundary condition $u(0) = u(1) = 0$, corresponding to four global minimizers to the functional \eqref{PINNs functional intro}.}
    \label{fig:Eikonal non-uniqueness}
\end{figure}

The admissible function space also matters for $\varepsilon>0$. To illustrate this, solve the viscous equation separately on $[0,1/2]$ and $[1/2,1]$, with the original endpoint values and an arbitrary common value $u_{1/2}$ at $x=1/2$.
Gluing these two classical solutions gives a continuous, piecewise smooth function whose classical residual vanishes away from $x=1/2$.
In general, however, its first derivative jumps at $x=1/2$, so its weak second derivative contains a Dirac mass. Such a function is not an admissible minimiser of the smooth or $H^2$ residual functional in \eqref{J_eps}. A residual evaluation based only on classical derivatives away from the interface would miss this singular contribution. This distinction motivates using a residual that is well-defined for continuous neural-network functions.

{ One could think of introducing regularization terms to convexify the functional and single out the correct viscosity solution.
This idea was exploited in \cite{le2024carleman}. There, the convexity of the functional is achieved by considering the viscous Hamilton-Jacobi equation with a Carleman weight function and a Sobolev norm, which allows for a Carleman-type inequality. We do not pursue in that direction.}

Instead of considering functionals based on the PDE residual, we consider minimising the $L^2$--residual of a finite-difference numerical scheme.
{ The needed regularisation will come implicitly from the numerical scheme.}
For any integer $N\geq3$, set $\delta := 1/N$, the uniform grid $\overline{\Omega}_\delta := \delta \Z \cap [0,1]$, and consider the following discrete version of the PDE-residual functional $\mathcal{R}_\varepsilon(\cdot)$ defined in \eqref{J_eps}, with $\varepsilon =0$:
\begin{equation} \label{constraint opt Eikonal LxF}
\min_{u} \widehat{\mathcal{R}}(u) := F(U) = \sum_{i=1}^{N-1}  \delta \left[ \widehat{H}_\alpha \left( \frac{u_{i+1}-u_i}{\delta}, \frac{u_{i}-u_{i-1}}{\delta}\right) \right]^2,~~\text{subject to} \ u_0 = u_N = 0,
\end{equation}
where,  for any $u\in C([0,1])$, the vector $U:= (u_0, \ldots, u_N)\in \R^{N+1}$ is the grid function on $\overline{\Omega}_\delta$ associated to $u$, and
$$
\widehat{H}_\alpha (p^+, p^-) = \left(\dfrac{p^+ + p^-}{2}\right)^2  - \alpha\dfrac{p^+ - p^-}{2} - 1
$$
for some $\alpha>0$, is the Lax-Friedrichs numerical Hamiltonian associated to $H(p) =  p^2 -1$.

Note that the functional $F(\cdot)$ is differentiable in $\R^{N+1}$, and then, the grid function $U\in \R^{N+1}$ associated to any local minimizer $u\in C([0,1])$ must satisfy the interior first-order optimality conditions $\partial_{u_i}F(U)=0$ for $i=1,\ldots,N-1$, since the endpoint values are fixed.

Denoting $p^\pm_i = \pm(u_{i \pm 1} - u_i)/\delta$, we can write the partial derivative of $F(U)$ with respect to $u_i$ as
\begin{eqnarray*}
    \partial_{u_i} F(U) &=& \widehat{H}_\alpha (p_{i-1}^+, p_{i-1}^-) \left( p_{i-1}^+ + p_{i-1}^- - \alpha \right) - \widehat{H}_\alpha (p_{i+1}^+, p_{i+1}^-) \left( p_{i+1}^+ + p_{i+1}^- + \alpha \right)  + 2\alpha \widehat{H}_\alpha (p_i^+, p_i^-), %
\end{eqnarray*}
for any $i = 2, \ldots , N-2$.

Now, let us denote 
$$
w_i = \widehat{H}_\alpha (p_i^+, p_i^-)\quad 
\text{and} \quad
v_i = p_i^+ + p_i^-, \qquad \text{for} \ i=1,\ldots, N-1.
$$
Using this notation, we can write the optimality condition for $u_i$ with $i = 2, \ldots, N-2$ as
\begin{equation*}
-(\alpha - v_{i-1}) w_{i-1}
-(\alpha + v_{i+1}) w_{i+1}  
+ 2\alpha w_i = 0 , \qquad i = 2,\ldots, N-2.
\end{equation*}

Further, since
\begin{eqnarray*}
\nonumber
    \partial_{u_1} F(U) &=&
    2\alpha \widehat{H}_\alpha (p_1^+, p_1^-)
    - \widehat{H}_\alpha (p_2^+, p_2^-) (p_2^+ + p_2^- + \alpha) \\
    &=& -(\alpha + v_2)w_2 + 2\alpha w_1,\\
    \partial_{u_{N-1}} F (U) &=&
    \widehat{H}_\alpha (p_{N-2}^+, p_{N-2}^-) (p_{N-2}^+ + p_{N-2}^- - \alpha) + 2\alpha\widehat{H}_\alpha (p_{N-1}^+, p_{N-1}^-) \\
    &=& -(\alpha - v_{N-2}) w_{N-2} + 2\alpha w_{N-1},
    \label{opti cond Eikonal u_{N-1}}
\end{eqnarray*}
we obtain the first-order optimality condition for the minimisation problem \eqref{constraint opt Eikonal LxF}, which reads as
\begin{equation}
    \label{opti cond Eikonal LxF}
    \begin{cases}
    -v_2 w_2 - \alpha (w_2 - 2w_1) =0 \\
    v_{i-1} w_{i-1} - v_{i+1} w_{i+1} - \alpha (w_{i+1} + w_{i-1} - 2 w_i) = 0 & \text{for} \ i= 2, \ldots , N-2 \\
     v_{N-2} w_{N-2}-\alpha (w_{N-2} - 2w_{N-1})= 0.
    \end{cases}
\end{equation}
Denoting $W= (w_1, \ldots, w_{N-1})$ and $V=(v_1, \ldots, v_{N-1})$, this system of equations can be written as 
\begin{equation}
\label{optimality system example}
-(A_N(V) + \alpha \Delta_N)  W = 0,
\end{equation}
where $A_N (V) \in \R^{(N-1)\times (N-1)}$ is a matrix depending on $V$ and $\Delta_N \in \R^{(N-1)\times (N-1)}$ is the Toeplitz matrix associated to the discrete Laplace operator on a uniform grid with $N-1$ points.
If $\alpha>\max_{1\leq i\leq N-1}|v_i|$, the transpose of $A_N(V)+\alpha\Delta_N$ satisfies a discrete maximum principle with homogeneous Dirichlet boundary values, and is therefore invertible. Consequently, the optimality system \eqref{optimality system example} has only the trivial solution $W=0$. In particular, for an $L$-Lipschitz function $u$, the sufficient condition $\alpha>2L$ is independent of $\delta$; see Section~\ref{sec: finite-difference proof} for the argument in arbitrary dimension.
This implies that, if $u$ is a critical point of $\widehat{\mathcal{R}}(\cdot)$, then the associated grid function $U = (u_0, \ldots, u_N)$ satisfies 
\begin{equation}
\label{finite-diffs eq example}
w_i =  \left(\dfrac{u_{i+1} - u_{i-1}}{2\delta} \right)^2 - \alpha \dfrac{u_{i+1} + u_{i-1} - 2u_i}{2\delta} - 1 = 0, \qquad \text{for all $i=1,\ldots,N-1$.}
\end{equation}
In other words, $U$ satisfies the finite-difference equation associated with the Lax-Friedrichs numerical Hamiltonian $\widehat{H}_\alpha (p^+, p^-)$.
See Section~\ref{sec: main results} for the precise condition on $\alpha$, and Section~\ref{sec: finite-difference proof} for the precise computations in the multi-dimensional setup.
The numerical Hamiltonian $\widehat{H}_\alpha$ is consistent with $H(p)=p^2-1$ for every $\alpha$, and is monotone on stencils with centred slope of magnitude at most $L$ when $\alpha\geq2L$, as in Definition~\ref{def: consistent and monotone scheme}.
Moreover, it is known (see \cite{crandall1984two, BS:1991}) that stable finite-difference solutions for consistent and monotone numerical schemes converge as $\delta\to0^+$, provided the boundary approximation is consistent and the limiting problem satisfies a comparison principle.
In Section~\ref{sec: numerics}, we propose a strategy to train a Neural Network, first with a large value of $\delta$ and then progressively reducing $\delta$ to decrease the numerical diffusion. This continuation strategy is motivated by the numerical experiments; the zero-residual conclusion above requires no lower bound on $\delta$.

Of course, in the multi-dimensional case, a functional such as $\widehat{\mathcal{R}} (u)$ in \eqref{constraint opt Eikonal LxF} would involve a summation over a multidimensional grid.
Computing the full gradient of $\widehat{\mathcal{R}} (u)$ would be computationally infeasible. 
In practice, one can use a version of stochastic gradient descent, in which at every gradient step, the gradient of $\widehat{\mathcal{R}} (u)$ is approximated by the gradient of a randomized selection of the terms in the sum.

\section{Deep Learning and Partial Differential Equations}
\label{sec: PINNs}

Due to its tremendous success in fields such as computer graphics and computer vision, Deep Learning has recently become a popular candidate for tackling various problems where classical methods face limitations. These include the numerical approximation of PDEs in multiple settings.
Under the Deep Learning framework, 
one typically minimizes 
a functional $\mathcal{J}:C(\Omega) \to \R^+$,
\begin{equation}
\label{optim problem DL}
{u}^\ast \in \arg \min_{u\in \mathcal{F}} \mathcal{J} (u), 
\end{equation}
where $\mathcal{F}\subset C(\overline{\Omega})$ is a class of functions defined by a neural network architecture. 
If $\mathcal{J}(u)\equiv  J(u; D)$ is based on some notion of misfit between the minimiser $u$ and a given data set $D$,
 the setup is often called ``supervised learning'' in the machine learning community.
If $\mathcal{J}$ is defined without labelled data, the setup is often called ``unsupervised".

One of the reasons for the success of Deep Learning techniques 
is that the commonly used (deep) neural networks can be constructed to approximate continuous functions in high dimensions conveniently  \cite{barron1993universal, yarotsky2017error, shen2019nonlinear} and very efficiently for certain classes of smooth functions \cite{bach2017breaking, SIEGEL2020, Siegel-Xu, weinan2022representation, weinan2019barron}.

Another reason for the success of Deep Learning techniques is the development of optimisation methods, which are used to address the problem \eqref{optim problem DL}. 
These are generally gradient-based iterative methods, such as stochastic gradient descent or variants. Hence, choosing the functional $\mathcal{J}(u)$ becomes paramount.
For finding the viscosity solution to the boundary value problem  \eqref{BVP intro}, we consider loss functionals 
of the form
$$
\mathcal{J}(u) = \mathcal{L}(u; D) +  \mathcal{R} (u).
$$

In \cite{raissi2019physics}, the popular Deep Learning framework known as Physics Informed Neural Networks (PINNs) was introduced.
In a PINNs approach, $\mathcal{R}$ can be interpreted as a Monte-Carlo approximation of the squared $L^2$--norm of the PDE residual, and $\mathcal{L}$ may involve both boundary data 
and possibly additional ``collocation" points (pointwise values of the solution in $\Omega$).
When the data fitting term $\mathcal{L}$ is present, $\mathcal{R}$ is regarded as a regularisation term to some practitioners.

Following the PINNs strategy, a least-square principle for the Hamilton-Jacobi equation would involve 
\begin{equation}\label{PINNs functional intro}
    \mathcal{R}(u)=\int_\Omega H(x, \nabla u(x))^2 dx.
\end{equation}
Assuming that we do not have, in principle, access to the viscosity solution in the interior of the domain, the functional $\mathcal{L}$ can only involve the boundary data:
\begin{equation}\label{PINNs functional intro boundary}
    \mathcal{L}(u;g) = \int_{\partial\Omega} (u(x) - g(x))^2 dx.
\end{equation}
In practice, both terms $\mathcal{R} (u)$ and $\mathcal{L}(u;g)$ are approximated through Monte Carlo integration.

We see that any Lipschitz continuous function $u$ satisfying $H(x,  \nabla u(x))=0$ for almost every $x\in \Omega$, together with the boundary condition $u(x) = g(x)$ on $\partial \Omega$, is a global minimizer of $\mathcal{J}(u)$.
It is well-known that solutions of this type are not unique. 
This is true even for the simple Eikonal equation on a one-dimensional interval shown in Example 1.
This is, of course, very inconvenient since one cannot guarantee by any means that minimising the PINNs functional $\mathcal{J}(u)$ with \eqref{PINNs functional intro} and \eqref{PINNs functional intro boundary} would produce an approximation of the viscosity solution.
See Figure~\ref{fig:Eikonal non-uniqueness} for an illustration of four global minimizers of the PINNs functional associated with the one-dimensional Eikonal equation in $(0,1)$ with zero-boundary condition.
To overcome the issue of non-uniqueness of minimisers, and promote convergence to the viscosity solution, 
we replace the $L^2$-norm of the PDE-residual with the least-square principle associated with a finite-difference numerical scheme. We prove that the numerical diffusion of the finite-difference scheme enhances regularity on the minimisers. Moreover, by choosing a suitable consistent and monotone numerical scheme, we can prove that the unique minimiser approximates the viscosity solution.

There are also supervised Deep Learning approaches to address Hamilton-Jacobi equations. 
Different Deep Learning approaches for learning, through data, a weak solution of Hamilton-Jacobi equations are proposed in \cite{cui2024supervised} and in \cite{borovykh2022data}.
In \cite{cui2024supervised}, data is generated by numerically evolving the Hamiltonian system (the characteristic equations of the Hamilton-Jacobi equation) in the phase space, and 
a least square principle is used to assign the solution gradient. However, it is unclear if the minimiser is the viscosity solution.
In \cite{borovykh2022data}, the issue of non-uniqueness of minimisers of the PDE-residual arising in Hamilton-Jacobi equations is alleviated by pre-training the NN with 
supervised data. 
In contrast, our approach requires no supervised data  
for convergence to the viscosity solution.

\subsection{Why finite-difference approximations?}\label{sec:why-fdm}

\paragraph{Larger domain of dependence/influence.}
Finite difference discretization leads to a larger domain of dependence/influence per collocation point compared to the pointwise evaluation of the PDE residual. This feature can lead to a stronger coupling of the ``pointwise losses" and more efficient capturing of the problem's underlying causality. 
{ In a finite-difference approximation of the differential operator, each evaluation involves points in a stencil that covers a region of positive measure, related to the discretization step. Instead, a pointwise evaluation of the differential operator using automatic differentiation only involves a single point, which has measure zero.
In other words, a residual-based Deep Learning model using finite differences has better data efficiency compared to a similar approach using automatic differentiation.
This property is essential as typical solutions of the HJ equation are not classical, and the singular sets of the solutions are lower dimensional. The probability of a finite set of collocation points ``sampling" the singular set is proportional to the size of the domain of dependence of the residual term to be minimized.
Therefore, using a pointwise evaluation of the PDE residual, it is impossible to distinguish between solutions that satisfy the PDE in an almost everywhere sense. See Section \ref{sec:algorithms} for further details.}

{
A similar reasoning can be done in terms of the complexity of the NN. Consider a real-valued function, $u$, defined by a ReLU-activated feedforward neural network. This function is a continuous piecewise linear function.
The collocation of $u$ with the given partial differential equation  
directly regulates only one of its ``linear" pieces. 
Each linear piece is defined by a subset of the network's parameters.
It is possible that the collocated linear pieces are not connected due to 
the number of collocation points being far fewer than the number of linear pieces in $u$.  
In contrast, collocating a finite difference equation regulates the relationship of possibly multiple linear pieces, influencing a larger subset of the parameters in $u$.
This means that \emph{a finite difference-based residual functional reduces the chance of overfitting.}
}

In Section \ref{subsec: data efficiency experiments}, our numerical experiments suggest that using a finer discretization of the PDE in the loss functional requires a larger number of collocation points during the optimisation algorithm.
Evaluating the PDE residual at the collocation points would be equivalent to letting the discretization step go to zero. In this case, the number of collocation points needed to compute a reasonable approximation of the solution would be significantly larger.

\paragraph{Simpler loss gradient evaluations.}
While automatic differentiation can be 
a very convenient and efficient tool for evaluating first-order derivatives, the evaluation of second-order derivatives can be costly with a large memory footprint. See e.g. \cite{sharma2023accelerated} and \cite{9945171}, Tab. V.

\paragraph{Convergence to the viscosity solutions.}
By using a convergent monotone numerical scheme, e.g. \cite{crandall1984two}, our proposed algorithm 
is capable of computing the viscosity solution of the given boundary value problem, without explicitly determining the convexity of the solution or where
the characteristics of problems should collide. 

\section{Finite difference residual functionals}
\label{sec: main results}

Let us consider a numerical Hamiltonian
$$
\widehat{H} (x,p^+, p^-): \Omega \times \mathbb{R}^d \times \mathbb{R}^d \longrightarrow \R.
$$
For a fixed $\delta>0$, we use the notation
$$
D^+_\delta u (x) := \left( \dfrac{u(x+\delta e_1) - u(x)}{\delta} , \ldots ,  \dfrac{u(x+\delta e_d) - u(x)}{\delta} \right) \in \R^d
$$
and
$$
D^-_\delta u (x) := \left( \dfrac{u(x) - u(x-\delta e_1)}{\delta} , \ldots ,  \dfrac{u(x) - u(x-\delta e_d)}{\delta} \right) \in \R^d,
$$
respectively for the upwind and downwind finite-difference approximation of the gradient $\nabla u(x)$.
Here, the vectors
$\{e_i\}_{i=1}^d$ represent the canonical basis of $\mathbb{R}^d$.
An important definition is in order:
\begin{definition}\label{def: consistent and monotone scheme} (Consistency and monotonicity)
    \begin{enumerate}
        \item We say that the numerical Hamiltonian $\widehat{H} (x, p^+, p^-)$ is consistent if 
        $\widehat{H}(x,p,p)\equiv H(x,p)$,~~~$p\in \mathbb{R}^d.$
    \item We say that the numerical Hamiltonian $\widehat{H} (x, p^+, p^-)$ is monotone in $\Omega$ if the function
    $$
    \left( u(x), \{u(x+\delta e_i)\}_{i=1}^d, \{u(x-\delta e_i)\}_{i=1}^d \right)
    \longmapsto \widehat{H} (x, D_\delta^+ u(x), D_\delta^- u(x))
    $$
    is non-decreasing with respect to $u(x)$ and non-increasing with respect to $u(x+\delta e_i)$ and $u(x-\delta e_i)$.
    \end{enumerate}
\end{definition}

\subsubsection*{The Lax-Friedrichs scheme}

In this paper, we work with the Lax-Friedrichs scheme, given by
\begin{equation}\label{H_LxF}
\widehat{H}_\alpha (x, p^+, p^-):= H\left(x, \dfrac{p^+ + p^-}{2}  \right)  - \alpha \sum_{i=1}^d \dfrac{p_i^+ - p_i^-}{2}, \qquad \text{for} \ \alpha>0.
\end{equation}
This numerical scheme is consistent with $H(x,p)$, and for any $L>0$, if one takes $\alpha$ satisfying
\begin{equation}
\label{C_H of L def}
\alpha \geq C_{H,\infty}(L) := \sup_{\substack{\|p\|_\infty\leq L \\
x\in \Omega}} \| \nabla_p H (x, p) \|_\infty,
\end{equation}
then $\widehat{H}_\alpha (x, p^+, p^-)$ is monotone on stencils whose centred gradient lies in $[-L,L]^d$. Here $\|\cdot\|_\infty$ is the maximum norm. If $u$ is $L$-Lipschitz in the Euclidean distance, each component of $(D_\delta^+u+D_\delta^-u)/2$ has absolute value at most $L$, so this bound applies. Whenever a condition involving $C_{H,\infty}(L)$ is imposed, this supremum is assumed finite. 

For any $\delta>0, $ let us define the uniform grid
$$\Omega_\delta := \delta \Z^d \cap \Omega,$$ and consider the residual functional
\begin{equation}
    \label{R hat delta}
    \widehat{\mathcal{R}} (u)
:= \delta^d\sum_{x\in \Omega_\delta} \left[\widehat{H}_\alpha (x, D_\delta^+ u(x), D_\delta^- u(x))\right]^2.
\end{equation}

\subsubsection*{Global minimisers}

Any function whose grid values satisfy
\begin{equation}
\label{finite-diff eq}
\widehat{H}_\alpha (x, D_\delta^+ u(x), D_\delta^- u(x)) = 0 \qquad x\in \Omega_\delta
\end{equation}
is a global minimiser of $\widehat{\mathcal{R}}$, with value zero. Conversely, if the discrete system is solvable in the admissible class, every global minimiser has zero residual and satisfies \eqref{finite-diff eq}. All stencil values are assumed to be defined, including any exterior values needed near the boundary.

For consistent and monotone schemes, convergence results such as \cite{crandall1984two, BS:1991} connect solutions of \eqref{finite-diff eq} to viscosity solutions of
\begin{equation}
\label{HJ PDE intro}
H(x,\nabla u)=0,\qquad x\in\Omega.
\end{equation}
More precisely, let $\delta_n\to0^+$ and let $u_n$ be an equicontinuous sequence of interpolants of zero-residual grid solutions converging uniformly to $u^*$ on $\overline{\Omega}$. If the schemes are monotone and consistent on smooth test functions as $\delta_n\to0^+$ (for Lax-Friedrichs schemes, in particular, $\alpha_n\delta_n\to0$), then $u^*$ solves \eqref{HJ PDE intro} in the viscosity sense. Convergence to the solution of the boundary value problem additionally requires a consistent treatment of the boundary data and a comparison principle for the limiting problem.

The second requirement, concerning critical points, is the main contribution of this paper. We provide a sufficient condition on $\alpha$, depending on the Hamiltonian and an a priori Lipschitz bound but independent of $\delta$, ensuring that critical points within that Lipschitz class have zero residual. Their convergence to a viscosity solution is then subject to the assumptions just described.

\subsubsection*{Critical points}
For the finite-dimensional description in this paragraph, all stencil values outside $\Omega_\delta$ are held fixed, and variations vanish at those nodes.
Let $N_\delta$ denote the number of grid nodes in $\Omega_\delta$. For any $u\in C(\overline{\Omega})$, let $U :=u|_{\Omega_\delta}\in\mathbb{R}^{N_\delta}$ denote the vector corresponding to the values of $u$ on $\Omega_\delta$ following a chosen ordering of the grid nodes in $\Omega_\delta$. We define the functional $F:\mathbb{R}^{N_\delta}\mapsto \mathbb{R},$ given by
$$
F(U) := \widehat{\mathcal{R}} (u) \quad  \text{for any}~u\in C(\overline{\Omega}) \ \text{such that}\ U=u|_{\Omega_\delta}.
$$
In view of \eqref{cond H}, the function $F(\cdot)$ is differentiable in $\R^{N_\delta}$, 
and we denote its gradient by $\nabla F (U)\in \R^{N_\delta}$.

For such variations, the differential of $\widehat{\mathcal{R}}(\cdot)$ at $u$ is given by
$$
\phi \in C(\overline{\Omega}) \longmapsto \Phi \cdot \nabla F (U), \quad \text{where} \ \Phi:=\phi|_{\Omega_\delta}\in \R^{N_\delta}.
$$
Define the residual values on the interior grid by
$$
w(x)=\widehat{H}_\alpha(x,D_\delta^+u(x),D_\delta^-u(x)),\qquad x\in\Omega_\delta.
$$
From the form of $\widehat{\mathcal{R}} (u)$ in \eqref{R hat delta}, which is a least-squares like functional, the gradient of $F$ has the following form
\begin{equation}\label{def:A_delta}
    \nabla F (U) = A_{\alpha, \delta} (U) W,
\end{equation}
where $W := w|_{\Omega_\delta}\in \R^{N_\delta}$, and $A_{\alpha,\delta} (U)$ is a linear bounded operator $\R^{N_\delta}\to \R^{N_\delta}$, depending on $\widehat{H}_\alpha$ and $U$ (see Section~\ref{sec: finite-difference proof} for the explicit form).
Then, for any $u\in C (\overline{\Omega})$ critical point of $\widehat{\mathcal{R}} (\cdot)$,  the vector of residuals $W$ satisfies the equation
\begin{equation}
\label{adjoint eq intro}
A_{\alpha,\delta} (U) W = 0.
\end{equation}
This linear system can be seen as a finite-difference equation for the residual function $w(x)$ on the grid $\Omega_\delta$. It can be interpreted as the adjoint equation associated with the finite difference system \eqref{finite-diff eq}.
Our next goal is to prove that by choosing $\alpha$ above a suitable bound on the derivatives of the Hamiltonian, one can ensure that the adjoint equation \eqref{adjoint eq intro} has a unique solution, which is the trivial solution $W = 0$.
This implies that any critical point of $\widehat{\mathcal{R}} (u)$ solves the finite-difference equation \eqref{finite-diff eq}. 

\subsection{A sufficient condition for residuals to be zero}

In the following theorem, for any constant $L>0$, we give a sufficient condition on $\alpha$ independent of $\delta$, ensuring that any critical point $u$ of $\widehat{\mathcal{R}} (u)$ with Lipschitz constant $L$ solves the finite-difference equation \eqref{finite-diff eq}.
We consider the domain $\Omega$ to be a $d$-dimensional cube for simplicity.
The result can be adapted to more general domains, satisfying suitable regularity properties, albeit with some additional technicalities.

\begin{theorem}
    \label{thm: uniqueness finite-diff}
    Let $H$ be a Hamiltonian satisfying \eqref{cond H}, and let $\Omega=(0,1)^d,$  $d\in \N$.
    Consider the uniform Cartesian grid $\Omega_\delta := \delta \Z^d\cap \Omega$, for some $\delta=1/N$ and $N\in\mathbb{N}$,
    and the functional $\widehat{\mathcal{R}} (\cdot)$ defined in \eqref{R hat delta}.
    
    Let $u\in C(\overline{\Omega})$ be a critical point of $\widehat{\mathcal{R}} (\cdot)$ with a Lipschitz constant $L$.
    If
        \begin{equation}
    \label{condition theorem}
      \alpha > C_{H,\infty}(L),
    \end{equation}   
    where $C_{H,\infty} (L)>0$ is given in \eqref{C_H of L def}, then   
    \begin{equation} \label{finite-diffs Lax-Friedrichs intro}
        \widehat{H}_\alpha (x, D_\delta^+ u(x), D_\delta^- u(x))=0,~~~ \forall x\in\Omega_\delta.
    \end{equation}
\end{theorem}

The proof is given in Section~\ref{sec: finite-difference proof}. Essentially, it shows that if $\alpha$ and $L$ satisfy \eqref{condition theorem}, then the linear operator $A_{\alpha,\delta}(U)$ in \eqref{def:A_delta} associated to $\widehat{H}_\alpha$ is invertible for any function $u$ with Lipschitz constant equal or smaller than $L$. This implies that if the first-order optimality condition \eqref{adjoint eq intro} holds, then the residuals $W$ associated with the function $u$ on $\Omega_\delta$ satisfy $W = 0$, and hence, $u$ solves the system of equations \eqref{finite-diffs Lax-Friedrichs intro}.

Condition \eqref{condition theorem} is independent of the mesh size. Thus, for a family of critical points with a common Lipschitz bound $L$, one may fix $\alpha>C_{H,\infty}(L)$ while letting $\delta\to0^+$. For smooth functions, the numerical diffusion in \eqref{H_LxF} has leading coefficient $\alpha\delta/2$, which then tends to zero. The proof uses a discrete maximum principle for the transpose of the linearized residual operator, rather than separate estimates of the drift and diffusion operators.

\paragraph{Note added in this revision.}
The mesh-independent condition on $\alpha$ in Theorem~\ref{thm: uniqueness finite-diff} replaces the mesh-dependent bound in the published version. The revised proof uses the same discrete maximum-principle mechanism as the subsequent work of Bokanowski, Esteve-Yag\"ue, and Tsai~\cite{bokanowskiEsteveTsaiResidualMinimization}. For the Cartesian Lax--Friedrichs scheme, our condition gives the positive ellipticity margin
\[
\underline{\alpha}
=\frac{\alpha-C_{H,\infty}(L)}{2}>0
\]
on stencils whose centered gradients lie in $[-L,L]^d$. It therefore corresponds to a localized form of hypothesis {\bf (H4)}$(i)$ in that work; the Cartesian Dirichlet grid satisfies the remaining connectivity and symmetry requirements. Their framework treats more general monotone schemes and, under global hypotheses, also establishes existence, uniqueness, and residual-based error estimates.

Since our goal is to use an artificial neural network to approximate the solution of a given Hamilton-Jacobi problem, we do not need to 
commit ourselves to one fixed grid as in a standard finite difference method.
Therefore, 
we consider functions on the closed stencil domain
$$
\Omega_\delta^{\mathrm{ext}}:=\overline{\Omega}+[-\delta,\delta]^d
$$
and the possibly extended-valued functional $\mathcal{R}:C(\Omega_\delta^{\mathrm{ext}})\to[0,+\infty]$ given by
\begin{equation}
\label{R alpha delta integral}
\mathcal{R}(u) =  \int_\Omega \left[\widehat{H}_\alpha (x, D_\delta^+ u(x), D_\delta^- u(x))\right]^2 dx.
\end{equation}
A critical point is understood to have finite functional value and vanishing first variations; local Lipschitz continuity of $H$ on the open domain alone does not ensure finiteness of the integral for every continuous function. This convention makes the stencil values near $\partial\Omega$ well-defined. Any Lipschitz bound for this functional is imposed on $\Omega_\delta^{\mathrm{ext}}$. The proof below uses only variations compactly supported in $\Omega$, so exterior values may also be held fixed. The Lipschitz bound is a property of the critical point, not a constraint on the first variations.
Note that
$$
\mathcal{R} (u) = \delta^{-d}\int_{(0,\delta)^d} \widehat{\mathcal{R}} (u; z) dz,
$$
where $\widehat{\mathcal{R}} (u; z)$ is a discrete functional similar to \eqref{R hat delta}, but on the shifted grid $\Omega_\delta (z) := (\delta \Z^d + z) \cap \Omega$ for $z\in (0,\delta)^d$.

A continuous function with $\mathcal{R}(u)=0$ solves the finite-difference equation throughout $\Omega$, and hence on every shifted Cartesian grid. Theorem~\ref{thm: main result} gives a sufficient condition for critical points to have this property.
This fact is then exploited in~Algorithm~\ref{alg: training}, which  
minimizes $\mathcal{R}(u)$ for a decreasing sequence of $\delta$.

We have the following result for the functional $\mathcal{R} (u)$.

\begin{theorem}
    \label{thm: main result}
    Let $H$ be a Hamiltonian satisfying \eqref{cond H}, and let $\Omega=(0,1)^d$, $d\in \N$.
    For any fixed $\alpha>0$ and integer $N\geq2$, set $\delta = 1/(N-1)$ and consider the functional $\mathcal{R}(\cdot)$ %
    defined in \eqref{R alpha delta integral}.
    
    Let $u\in C(\Omega_\delta^{\mathrm{ext}})$ be any critical point of $\mathcal{R} (\cdot)$ with Lipschitz constant at most $L$ on $\Omega_\delta^{\mathrm{ext}}$.
    If \eqref{condition theorem} holds,
    then   
    \begin{equation} \label{finite-diffs Lax-Friedrichs continuous}
        \widehat{H}_\alpha (x, D_\delta^+ u(x), D_\delta^- u(x)) =0,~~~ \forall x\in\Omega.
    \end{equation}
\end{theorem}

The proof is given in Section~\ref{subsec: proof main thm}. 

We make the following remarks concerning the above results: 
\begin{enumerate}
    \item Theorems  \ref{thm: uniqueness finite-diff} and \ref{thm: main result} are local results in the sense that they only ensure that a critical point $u\in C(\overline{\Omega})$ of $\widehat{\mathcal{R}} (\cdot)$ solves the finite-difference equation \eqref{finite-diffs Lax-Friedrichs intro}, provided that the Lipschitz constant of $u$ is smaller or equal than $L$. They do not rule out the possibility of critical points with larger Lipschitz constants failing to satisfy \eqref{finite-diffs Lax-Friedrichs intro}.
    However, one can take $L>0$ as large as needed, at the expense of increasing $\alpha$ so that \eqref{condition theorem} holds, provided $C_{H,\infty}(L)$ is finite.
    \item Larger values of $\alpha$ and $\delta$ introduce more numerical diffusion and can over-smooth the approximation. However, \eqref{condition theorem} requires no lower bound on $\delta$: for fixed $\alpha>C_{H,\infty}(L)$, the numerical diffusion vanishes as $\delta\to0^+$. In Section~\ref{sec:algorithms}, we nevertheless use a progressive refinement schedule as an optimization strategy. Theorems \ref{thm: uniqueness finite-diff} and \ref{thm: main result} concern critical points in function space; they do not by themselves establish convergence of neural-network training.
    \item The theorems characterize critical points but do not prove their existence. Solving the discrete equation independently on each shifted grid does not by itself construct a continuous solution: solvability and compatibility across the boundaries of the shifted-grid cells must also be established.
\end{enumerate}

\subsection{Vanishing residuals by regularisation of the Hamiltonian}

In Theorem \ref{thm: uniqueness finite-diff} (resp. Theorem \ref{thm: main result}), the condition $\alpha>C_{H,\infty}(L)$ ensures that any critical point with Lipschitz constant at most $L$ satisfies
\begin{equation}
\label{finite-diff eq ensuring uniqueness}
\widehat{H}_\alpha (x,D_\delta^+u(x),D_\delta^-u(x))=0,
\qquad x\in\Omega_\delta\quad\text{(resp. $\Omega$).}
\end{equation}
Adding a monotonically increasing term to the numerical Hamiltonian gives another sufficient condition. Consider
$$
\widehat{\mathcal{R}}_\tau(u):=\delta^d\sum_{x\in\Omega_\delta}
\left[\widehat{H}_\alpha(x,D_\delta^+u(x),D_\delta^-u(x))+\tau u(x)\right]^2,
\qquad \tau>0.
$$
Let $W$ be the vector of these regularized residuals. With the normalization of \eqref{def:A_delta}, the interior first-order optimality condition is
$$
\left(A_{\alpha,\delta}(U)+2\tau\delta^d I_{\Omega_\delta}\right)W=0.
$$
Equivalently, in the notation of Section~\ref{subsec: proof of thm discrete func}, it is
$$
\left(A_N(V)+\alpha\Delta_N-2\delta\tau I\right)W=0.
$$
A sufficient condition for invertibility is
\begin{equation}
\label{condition regularized}
\alpha+\frac{\delta\tau}{d}>C_{H,\infty}(L).
\end{equation}
Indeed, the transpose of this matrix has diagonal entries $-2d\alpha-2\delta\tau$, and the sum of the absolute values of the off-diagonal entries in row $\beta$ is at most
$$
\sum_{i=1}^d\left(|\alpha-V_\beta^{(i)}|+|\alpha+V_\beta^{(i)}|\right)
=2\sum_{i=1}^d\max\{\alpha,|V_\beta^{(i)}|\}
\leq2d\max\{\alpha,C_{H,\infty}(L)\}.
$$
Since $\tau>0$, condition \eqref{condition regularized} makes the matrix strictly diagonally dominant, hence invertible. In particular, $\alpha\geq C_{H,\infty}(L)$ suffices for every $\tau>0$. Stationarity therefore implies
$$
\widehat{H}_\alpha(x,D_\delta^+u(x),D_\delta^-u(x))+\tau u(x)=0,
\qquad x\in\Omega_\delta.
$$
In this paper, we predominantly consider the case $\tau=0$.

\subsection{The least square principle involving a numerical Hamiltonian and supervised data}
\label{subsec: supervised data}

In the above section, we discussed functionals
involving a numerical Hamiltonian. We show that under appropriate conditions, any global minimiser approximates a viscosity solution of the Hamilton Jacobi equation
\begin{equation}
\label{HJ eq supervised data}
 H(x, \nabla u) = 0, \qquad \text{in} \ \Omega.   
\end{equation}
The uniqueness of the viscosity solution to the partial differential equation \eqref{HJ eq supervised data} typically involves additional (suitable) boundary conditions. 
Likewise, our ultimate least square formulation has to incorporate
boundary data.

{
For a given $\delta>0$, let us consider the uniform Cartesian grid $\Omega_\delta : = \delta \Z^d\cap\Omega$, where $\delta = 1/N$ for some $N\in \N$.
Similarly, we define the numerical boundary as
\begin{equation}
\label{numerical boundary}
\partial\Omega_\delta := \{x \in \delta \mathbb{Z}^d\setminus \Omega_\delta \, : \ \min_{y\in\partial\Omega} |x-y| < \sqrt{d}~\delta\}.
\end{equation}
In order to set the boundary condition on the numerical boundary, the function $g:\partial \Omega \to \R$ in \eqref{HJ PDE intro} needs to be extended to a $\delta$-neighbourhood of the boundary, defined as
\begin{equation}
\label{boundary neighbourhood}
\Gamma_\delta := \{ x\in \R^d\setminus \Omega \, : \ \min_{y\in\partial\Omega} |x-y| < \sqrt{d}~\delta \}. 
\end{equation}
Since $g$ is a continuous function, the extension of $g$ to $\Gamma_\delta$ can be done in a continuous way.
}

For the parameters $\alpha, \gamma >0$, we define the functional
\begin{equation}
\label{discrete func boundary}
\widehat{\mathcal{J}} (u) :=
\delta^d \sum_{x\in \Omega_\delta} \left[ \widehat{H}_\alpha (x, D_\delta^+ u(x), D_\delta^- u(x)) \right]^2 + \gamma \delta^{d-1} \sum_{x\in \partial \Omega_\delta} \left(  u(x) - g(x) \right)^2,
\end{equation}
where $\widehat{H}_\alpha$ is the Lax-Friedrichs numerical Hamiltonian defined in \eqref{H_LxF} and $g$ is the boundary condition in~\eqref{BVP intro}, continuously extended to $\Gamma_\delta$.
The following result is a consequence of Theorem \ref{thm: uniqueness finite-diff}.

\begin{corollary}
\label{cor: uniqueness boundary data}
    Let $H$ be a Hamiltonian satisfying \eqref{cond H}, $\Omega=(0,1)^d$,  $d\in \N$, and $g\in C(\partial\Omega)$.
    For any $N\in \N$, set $\delta = 1/N$ and the uniform Cartesian grid $\Omega_\delta := \delta \Z^d\cap \Omega$, with the associated numerical boundary {$\partial\Omega_\delta$ defined in \eqref{numerical boundary}, and consider a continuous extension of $g$ to $\Gamma_\delta$.}
    For the parameters $\alpha, \gamma>0$, consider the functional $\widehat{\mathcal{J}} (\cdot)$ defined in \eqref{discrete func boundary}.
    Let $u\in C(\overline{\Omega}\cup\overline{\Gamma_\delta})$ be any critical point of $\widehat{\mathcal{J}} (\cdot)$ with Lipschitz constant at most $L$.
    If \eqref{condition theorem} holds, 
    then   
    \begin{equation} \label{finite-diffs Lax-Friedrichs with boundary}
    \begin{cases}
        \widehat{H}_\alpha (x, D_\delta^+ u(x), D_\delta^- u(x)) =0, & x\in\Omega_\delta, \\
        u(x) = g(x), & x\in \partial\Omega_\delta.
    \end{cases}
    \end{equation}
\end{corollary}

Under \eqref{condition theorem}, strict monotonicity and the discrete comparison principle imply that the system \eqref{finite-diffs Lax-Friedrichs with boundary} has at most one solution in the class of grid functions with Lipschitz constant at most $L$. Thus, the corollary gives uniqueness of the restriction of a critical point to $\Omega_\delta\cup\partial\Omega_\delta$, for the chosen extension of $g$. It does not distinguish continuous functions agreeing at all sampled nodes.

The proof of this corollary can be found in Section~\ref{subsec: proof of thm discrete func}. 
Suppose, in addition to Assumption~\ref{assump: Lipschitz}, that the boundary value problem satisfies a comparison principle. If $u_n$ is a uniformly bounded, equicontinuous sequence of critical points of $\widehat{\mathcal{J}}$ with a common Lipschitz bound $L$, $\delta_n\to0^+$, and a fixed $\alpha>C_{H,\infty}(L)$, then Corollary~\ref{cor: uniqueness boundary data} gives zero residual and exact numerical boundary data at every stage. With a consistent boundary extension, the monotone-scheme convergence argument implies uniform convergence on $\overline{\Omega}$ to the unique viscosity solution. Assumption~\ref{assump: Lipschitz} alone does not provide comparison or uniqueness.
Condition \eqref{condition theorem} remains valid as $\delta$ decreases, provided the Lipschitz bound $L$ and $\alpha$ remain fixed. In Section~\ref{sec: numerics}, we use a large initial value of $\delta$ and then refine it as an optimization strategy.

For Monte Carlo sampling, one may instead consider the weighted functional below, with $u$ defined on a closed neighborhood containing all stencil points and $\Gamma_\delta$:
\begin{equation*}
\mathcal{J} (u) := 
\int_\Omega 
\left[ \widehat{H}_\alpha (x, D_\delta^+ u(x), D_\delta^- u(x)) \right]^2 d\rho(x)
+ \gamma \int_{\Gamma_\delta} \left( u(x) - g(x) \right)^2 d\mu(x),  
\end{equation*}
where $\rho$ and $\mu$ are probability measures on $\Omega$ and $\Gamma_\delta$ respectively, {and $\Gamma_\delta$ is defined in \eqref{boundary neighbourhood}.} The critical-point conclusions above are not asserted here for arbitrary sampling measures.

In this paper, we propose to minimize such least square principle via a stochastic gradient descent algorithm. This means that the integrals will be approximated by the Monte Carlo method.
The involvement of the Monte-Carlo method requires deciding on a probability distribution $\rho$ over $\Omega$, and, as we will show in Section \ref{subsec: data distribution experiments}, the choice of this distribution can depend on the error metric of choice.

{Ultimately, we shall construct a sequence $u_k$ that minimizes $\mathcal{J}$
for a sequence of $\delta_k$ that decreases to $0$. In the numerical experiments, we impose the boundary penalty directly on $\partial\Omega$ for convenience. This differs from the numerical-boundary formulation of Corollary~\ref{cor: uniqueness boundary data}; continuity of individual networks alone does not establish equivalence of the two formulations. We use
\begin{equation}\label{eq:J-with-boundary-data}
\mathcal{J} (u) := 
\int_\Omega 
\left[ \widehat{H}_\alpha (x, D_\delta^+ u(x), D_\delta^- u(x)) \right]^2 d\rho(x)
+ \gamma \int_{\partial\Omega} \left( u(x) - g(x) \right)^2 d\mu(x),
\end{equation}
where $\mu$ is a probability measure on $\partial\Omega$.
}

\paragraph{Incorporating additional labelled data}

Finally, in many applications, labelled data is available in the interior of $\Omega$. {Let us denote this interior set by $\Gamma\subset \Omega$, and the associated data by the function $h\in C(\Gamma)$.} We can of course introduce an additional data-fitting term to form a least square principle as follows:
\begin{equation}\label{def:J-BVP}
    \mathcal{J}(u)
:= \int_\Omega 
\left[ \widehat{H}_\alpha (x, D_\delta^+ u(x), D_\delta^- u(x)) \right]^2 d\rho(x)
+ \gamma_1 \int_{\partial\Omega} \left[ u(x) - g(x) \right]^2 d\mu(x)
+ \gamma_2\int_{\Gamma} \left[ u(x) - h (x) \right]^2 d\nu(x),
\end{equation}
where $\rho$ is a probability measure on $\Omega$, $\mu$ and $\nu$ are probability measures on $\partial\Omega$ and $\Gamma$.

The addition of labelled data, a.k.a. supervised data, changes the optimization landscape and can influence the critical points of the parameterized loss. Section~\ref{subsec: adding supervised data} studies the effect of such data empirically.

\subsection{Time dependent problems}
\label{subsec: time-dependent pbms}

The finite-difference least-squares formulation can also be used for time-dependent problems:
\begin{equation}
\label{time-evol HJ}
\begin{cases}
\partial_t u + H(x, \nabla u)=0 & \Omega\times(0,T), \\
u(x,t)=g(x) &  \partial\Omega \times (0,T), \\
u(x,0) = u_0(x) & \Omega,
\end{cases}
\end{equation}
with $T\in (0, \infty]$, initial data
$u_0:\Omega \to \R$ and boundary data $g: \partial\Omega\to \R$.

In this case, a suitable choice for the least squares approach is to use the following finite-difference scheme
\begin{equation}
\label{finite-diffs time dep}
u(x, t+\delta_t) = u(x, t) - \delta_t \widehat{H}_\alpha (x, D_{\delta_x}^+ u(x, t), D_{\delta_x}^- u(x, t)),
\end{equation}
for some $\delta_t, \delta_x, \alpha >0$,
which combines an explicit Euler scheme for the time derivative and a Lax-Friedrichs numerical Hamiltonian. On stencils whose centred spatial gradients lie in $[-L,L]^d$, this update is monotone if
$$
\alpha\geq C_{H,\infty}(L),\qquad \frac{d\alpha\delta_t}{\delta_x}\leq1.
$$

The least square finite-difference functional associated to the above numerical scheme is
\begin{equation}
\label{R(u) time dep}
\mathcal{R} (u) :=
\int_0^T \int_\Omega \left[ u(x, t+\delta_t) - u(x, t) + \delta_t \widehat{H}_\alpha (x, D_{\delta_x}^+ u(x, t), D_{\delta_x}^- u(x, t)) \right]^2 dx\,dt,
\end{equation}
where $u$ is defined on the spatial stencil neighborhood and, for finite $T$, up to time $T+\delta_t$. This functional can be coupled with the initial and boundary data as
\begin{equation}\label{def:J_IBVP}
    \mathcal{J}(u) := \mathcal{R}(u) + 
\gamma_b \int_0^T \int_{\partial\Omega} \left( u(x,t) - g(x) \right)^2 dS(x)\,dt + \gamma_0 \int_\Omega \left( u(x,0) - u_0(x) \right)^2 dx.
\end{equation}
The corresponding first-order optimality system can be analysed on a space-time grid. The stationary results above do not by themselves establish convergence of training for this time-dependent functional.
In Section~\ref{subsec: numerics time-dependent}, we show some numerical experiments, where we apply our least squares approach to time evolution problems of the form \eqref{time-evol HJ}.

\section{Numerics}
\label{sec: numerics}

As discussed in the introduction, the purpose of formulating the numerical method for \eqref{BVP intro} as a minimisation problem is that one can approximate the solution using a Neural Network trained through stochastic gradient descent (SGD).
To promote convergence to the viscosity solution, we use a Lax-Friedrichs residual. When the scheme is monotone and the convergence assumptions discussed in Section~\ref{sec: main results} hold, its zero-residual solutions approximate a viscosity solution.
Theorems \ref{thm: uniqueness finite-diff} and \ref{thm: main result} ensure that critical points in function space with Lipschitz constant at most $L$ have zero residual when $\alpha>C_{H,\infty}(L)$, independently of $\delta$. This does not exclude nonzero-residual stationary points after restriction to a neural-network parameterization. Larger values of $\alpha\delta$ introduce more smoothing and may facilitate optimization, at the expense of approximation accuracy. We therefore use a continuation strategy: train first with larger numerical diffusion, and then reduce the parameters to refine the approximation.

\subsection{The proposed algorithms}\label{sec:algorithms}

Let us consider any NN function $\Phi: \R^d\times \Theta \to \R$, where $\Theta\subset \R^P$ is the parameter space and $P$ is the number of parameters. We define the parameterized family of functions
\begin{equation}\label{family of functions}
\mathcal{F} := \{x\mapsto
\Phi(x; \theta) \ : \quad \theta \in \Theta
\} \subset C(\overline{\Omega}).
\end{equation}
Our numerical solution of \eqref{BVP intro} will be a function of the form $u(x) = \Phi (x; \theta^\ast)$, where $\theta^\ast\in \Theta$ is obtained by minimising the function
\begin{equation}
\theta \longmapsto \mathcal{J} (\Phi (\cdot; \theta)),
\end{equation}
and $\mathcal{J}$ is defined in \eqref{def:J-BVP} (or in \eqref{def:J_IBVP} for time-dependent Hamilton-Jacobi equations).
The minimisation algorithm to obtain $\theta^\ast$ is a version of SGD in which, at every iteration, the functional $\mathcal{J}$ in \eqref{def:J-BVP} is approximated by a Monte Carlo approximation of the integrals.

Let $\mathcal{D}_0$, $\mathcal{D}_{b_1}$, and $\mathcal{D}_{b_2}$ be the chosen probability distributions over $\Omega$, $\partial\Omega$, and $\Gamma$ respectively. 
Our algorithms are built on the typical stochastic gradient descent algorithm: 
\begin{equation}\label{functional numerics}
\theta^{k+1} = \theta^k - 
\eta_k \left(\nabla_\theta \widetilde{\mathcal{L}} (\Phi (\cdot; \theta^k);\gamma_1,\gamma_2) + \nabla_\theta \widetilde{\mathcal{R}} (\Phi (\cdot; \theta^k))\right),
\end{equation}
with
\begin{eqnarray*}
\widetilde{\mathcal{L}} (\Phi (\cdot; \theta);\gamma_1,\gamma_2) &:=& 
\dfrac{\gamma_1}{N_{b_1}} \sum_{x\in X_{b_1}} \left( \Phi (x;\theta) - g(x) \right)^2 + \dfrac{\gamma_2}{N_{b_2}} \sum_{x\in X_{b_2}} \left( \Phi (x;\theta) - h(x) \right)^2, \\
 \widetilde{\mathcal{R}} (\Phi (\cdot; \theta)) &:=&
\dfrac{1}{N_0} \sum_{x\in X} \left[ \widehat{H}_\alpha (x, D_\delta^+ u(x), D_\delta^- u(x)) \right]^2.  
\end{eqnarray*}
Here $u=\Phi(\cdot;\theta)$. When no interior supervised data are used, we set $\gamma_2=0$ and $N_{b_2}=0$ and omit the second term in $\widetilde{\mathcal{L}}$, rather than divide by $N_{b_2}$. The set $X$ contains $N_0$ i.i.d. samples from $\mathcal{D}_0$, $X_{b_1}$ contains $N_{b_1}$ i.i.d. samples from $\mathcal{D}_{b_1}$, and $X_{b_2}$ contains $N_{b_2}$ i.i.d. samples from $\mathcal{D}_{b_2}$. The sets $X,X_{b_1},$ and $X_{b_2}$ are \emph{resampled} in each iteration.
In the Deep Learning community, the i.i.d. samples $X\subset \Omega$, $X_{b_1}\subset \partial\Omega$ and $X_{b_2}\subset \Gamma$ are often referred to as \textit{mini-batches}, and $N_0$, $N_{b_1}$ and $N_{b_2}$ denote the batch-size for each mini-batch.

A possible choice for the probability distributions $\mathcal{D}$, $\mathcal{D}_{b_1}$ and $\mathcal{D}_{b_2}$ are the uniform distributions over $\Omega$, $\partial\Omega$, and $\Gamma$ respectively. However, as we observe in the numerical experiments in section \ref{subsec: data distribution experiments}, other distributions can be more efficient, especially in high-dimensional domains. 

\begin{algorithm}\label{algorithm:alg_1}
\caption{SGD for Lax-Friedrichs finite-difference loss}\label{alg: SGD}
\begin{algorithmic}
\Require $\gamma_1>0$, $\gamma_2\geq0$, $\{\eta_k\}_{k=0}^{K-1}\in(0,\infty)^K$, $N_0,N_{b_1}\in\N$, $N_{b_2}\in\N\cup\{0\}$, $\mathcal{D}_0\in \mathcal{P}(\Omega)$, $\mathcal{D}_{b_1}\in \mathcal{P}(\partial\Omega)$, and $\mathcal{D}_{b_2}\in \mathcal{P}(\Gamma)$.
\State \textbf{Input:} $\theta_0\in \Theta$, $\alpha\geq0$, and $\delta>0$.
\Procedure{SGDLxF}{$\theta_0, \alpha, \delta$}

\State $\theta^0\gets\theta_0$, $k\gets0$
\While{$k < K$}
\State $X\gets$ $N_0$ i.i.d. samples from $\mathcal{D}_0$
\State $X_{b_1}\gets $ $N_{b_1}$ i.i.d. samples from $\mathcal{D}_{b_1}$
\State $X_{b_2}\gets$ $N_{b_2}$ i.i.d. samples from $\mathcal{D}_{b_2}$ (or $\emptyset$ if $N_{b_2}=0$)

\State $\theta^{k+1} := 
\theta^k - 
\eta_{k} \left( \nabla_{\theta} \widetilde{\mathcal{L}} (\Phi (\cdot; \theta^k);\gamma_1,\gamma_2) + \nabla_\theta \widetilde{\mathcal{R}}(\Phi (\cdot; \theta^k))\right)$
\State $k\gets k+1$
\EndWhile

\Return $\theta^K$
\EndProcedure
\end{algorithmic}
\end{algorithm}

A useful initialization can facilitate the optimization of the parameterized loss.
For the finite ReLU networks used below, restricting the parameters to a compact set provides a uniform Lipschitz bound $L$.
We can then choose $\alpha_0>C_{H,\infty}(L)$ to meet \eqref{condition theorem}, independently of the initial stencil size $\delta_0>0$.

In view of Theorem \ref{thm: main result}, we can ensure that any critical point of 
$$
\mathcal{R}(u) = \mathcal{R}_{\alpha_0, \delta_0} (u) =
\int_\Omega \left[ \widehat{H}_{\alpha_0} (x, D_{\delta_0}^+ u(x), D_{\delta_0}^- u(x)) \right]^2 dx
$$
with Lipschitz constant $L>0$ satisfies 
\begin{equation}
\label{finite-diff eq numerics}
\widehat{H}_\alpha (x, D_\delta^+ u(x), D_\delta^- u(x))=0 \qquad  \text{for all} \quad x\in \Omega.
\end{equation}
This conclusion concerns critical points of the full functional; convergence of the parameterized optimization requires additional analysis.
Once the NN is trained, we can reduce the hyper-parameters $\alpha, \delta$ by choosing
$$
0< \alpha := \alpha_1 \lesssim \alpha_0
\quad \text{and} \quad 
0< \delta := \delta_1 \lesssim \delta_0
$$
and continue with the gradient descent iterations applied to the new functional $\mathcal{R}_{\alpha_1, \delta_1} (u)$.
The previous approximation provides an initialization for the next optimization stage. This procedure can be iterated with a reduction schedule for $\alpha$ and $\delta$. The zero-residual guarantee continues to apply at each stage for which $\alpha_m>C_{H,\infty}(L_m)$, where $L_m$ bounds the Lipschitz constant of the candidate function. Reducing $\delta$ at fixed admissible $\alpha$ preserves this condition, whereas reducing $\alpha$ may violate it. Decreasing $\alpha\delta$ reduces the leading numerical diffusion and motivates interpreting the schedule as an approximation of the vanishing viscosity method. Some experiments also use a final stage with $\alpha=0$, which lies outside the sufficient condition. We assess the training procedure numerically below.

\begin{algorithm}
\caption{Optimisation for a Lax-Friedrichs finite-difference loss with a progressive refinement schedule}\label{alg: training}
\begin{algorithmic}
\Require $M\in \N$, $L>0$, $\mathcal{F}$ as in \eqref{family of functions} and $\{(\alpha_m,\delta_m)\}_{m=0}^{M-1}\in([0,\infty)\times(0,\infty))^M$.
\Ensure $(L,\alpha_0)$ satisfy \eqref{condition theorem} and $\alpha_{m+1}\leq\alpha_m$ and $\delta_{m+1}\leq\delta_m$ for all $0\leq m<M-1$.
\State Initialise $\theta_0\in\Theta$ so that $x\mapsto\Phi(x;\theta_0)$ has Lipschitz constant at most $L$. 
\State $m \gets 0$
\While{$m < M$}
\State $\theta_{m+1} :=\operatorname{SGDLxF}(\theta_m, \alpha_m, \delta_m)$ as defined in Algorithm \ref{alg: SGD}
\State $m\gets m+1$
\EndWhile

\Return $\theta_M$
\end{algorithmic}
\end{algorithm}

\subsubsection*{Relating $\delta$, the total number of collocations $N$ and the dimension $d$}
In this paper, we argued that using a finite difference scheme allows us to regulate the network for the viscosity solution. The effectiveness of the regularization directly depends on
the probability of the finite difference stencil sampling over the region where the characteristics should collide. This region corresponds to the location where the viscosity solution is not differentiable and is typically lower dimensional.
On the one hand, for a fixed number of collocation points randomly placed in the domain, larger $\delta$ in $\hat{H}$ increases the probability of the finite difference stencil regulating over that region. On the other hand, finite difference solutions with larger values of $\delta$ have worse approximation errors due to numerical diffusion.

To relate the stencil size to the spatial resolution of the collocation points, let $D_N\subset\Omega$ be the sampled point set and define its fill distance by
$$
h_N:=\sup_{y\in\Omega}\min_{x\in D_N}|y-x|.
$$
By definition, $\delta\geq h_N$ ensures that the closed balls of radius $\delta$ centred at the collocation points cover $\Omega$. This motivates comparing $\delta$ with the sampling resolution, but does not imply that the finite set of axial stencil points covers the domain, resolves every singularity, or guarantees convergence of training. Probabilistic estimates for interior Voronoi cell diameters under suitable sampling and density assumptions are given in \cite{liu2024nearest}. Here the relationship between stencil size and sampling resolution is used as a heuristic, whose effect is examined numerically in Section~\ref{subsec: data efficiency experiments}.

\subsubsection*{The benefits of resampling}
Resampling of collocation points during the training process is essential to control the generalization error. It is also a simple way to avoid overfitting.

As we mentioned earlier, collocating a ReLU-activated network with a finite-difference equation directly regulates the network in the neighbourhood of each collocation point.
Thus, the resampling of collocation points effectively increases the total number of collocation points used in regulating the network while having much fewer gradient evaluations in each training step. The resulting approximation will have a more uniformly distributed error. 

Finally, as we showed above, having a larger total number of collocation points allows us to use smaller $\delta$ in the finite differences to improve the accuracy of the approximations.

\subsection{Computational experiments}

We test the efficacy of our algorithm on various canonical HJ equations posed in different dimensions. We start with the Eikonal equation posed on simple domains (Section~\ref{sec:Eikonal-experiments}). Then, we consider a time-evolution HJ equation of the form \eqref{time-evol HJ} (Section~\ref{subsec: numerics time-dependent}). Finally, we consider HJ equations arising in control problems and differential games associated with curvature-constrained dynamics.
In particular, we consider two problems associated with Reeds-Shepp's car model: the shortest path to a given target (Section~\ref{subsec: car OCP}) and the pursuit-evasion game (Section~\ref{subsec: pursuit evasion game}).
{ In Section \ref{sec: analysis}, we analyse, through numerical tests, various aspects of our proposed algorithm, such as data efficiency and the choice of the distribution for the collocation points in the PDE domain.
}

In all the experiments, we use a fully connected multi-layer perceptron (MLP) as NN architecture.
For a chosen number of hidden layers $\ell\in \N$ (NN depth) and a vector $[d_1, \ldots , d_\ell]\in \N^\ell$ representing the number of neurons in each hidden layer,
the network function can be written as
\begin{equation}
\label{MLP def}
\Phi (x ; \theta) = W_\ell y_\ell + b_\ell, \quad
\text{where} \quad
\begin{cases}
    y_i = \sigma\left( W_{i-1} y_{i-1} + b_{i-1} \right), 
    & \text{for} \ i = 1, \ldots , \ell \\
    y_0 = x.
\end{cases}
\end{equation}
where $\theta := \{ (W_i, b_i) \}_{i=0}^\ell$ is the parameter of the NN.
The matrices $W_0\in \R^{d_1\times d}$, $\{W_i\}_{i=1}^{\ell-1}\in \prod_{i=1}^{\ell-1} \R^{d_{i+1}\times d_{i}}$ and $W_\ell\in \R^{1\times d_\ell}$ are typically called weights, and the vectors $\{b_i\}_{i=0}^{\ell-1} \in \prod_{i=0}^{\ell-1} \R^{d_{i+1}}$ and the scalar $b_\ell\in \R$ are called biases.
The function $\sigma(\cdot)\in C(\R; \R^+)$ defined as $\sigma (s) := \max \{ 0,s  \}$ is the activation function, which acts component-wise on vectors.

The NN architecture will be indicated in the experiment description for each case.
For instance, in the case of a fully connected MLP with three hidden layers, the architecture is denoted by $[d_1, d_2, d_3]$, where $d_1, d_2$ and $d_3$ represent the number of neurons in each hidden layer.
{ Investigating the optimal choice of the NN architecture is beyond the scope of this paper. Therefore, in all the experiments, we chose relatively simple architectures in which the number of neurons in the hidden layers is kept constant.
}
For the problems in Sections~\ref{subsec: car OCP} and \ref{subsec: pursuit evasion game}, we consider a slightly different architecture because the problem's domain requires the solution to be $2\pi$-periodic with respect to some of the variables.
This is achieved by considering a linear combination of trigonometric functions, for which the coefficients are given by MLPs as defined above.

We consider various examples for which the viscosity solution is available in closed form.
Namely, we consider the Eikonal equation in simple domains such as a cube, a ball and an annulus, in different dimensions.
In this case, the viscosity solution corresponds to the distance function to the boundary and can be easily computed for the three aforementioned domains. This function is used as a ground truth to evaluate the accuracy of the solution provided by the trained NN.

We use two metrics to evaluate the accuracy: the Mean Square Error (MSE) and the $L^\infty$-error.
These are defined respectively as
\begin{equation}
\label{MSE and Linf def exact}
MSE (\Phi(\cdot; \theta)) \coloneq \dfrac{1}{|\Omega|} \int_\Omega | \Phi (x;\theta) - u(x) |^2 dx
\qquad \text{and} \qquad
E_\infty (\Phi(\cdot; \theta)) \coloneq \max_{x\in \overline{\Omega}} | \Phi (x;\theta) - u(x) |,
\end{equation}
where $u\in C(\overline{\Omega})$ is the viscosity solution used as ground truth.

Obviously, we cannot compute exactly the integral and the maximum in \eqref{MSE and Linf def exact}. We approximated it by the Monte Carlo method. For some $N\in \N$, let us define
\begin{align}
\label{MSE def}
\widetilde{MSE} (\Phi(\cdot; \theta)) := \dfrac{1}{N} \sum_{i=1}^N | \Phi (x_i;\theta) - u(x_i) |^2 \\
\label{Linf def}
\widetilde{E}_\infty (\Phi(\cdot; \theta)) := \max_{i = 0, \ldots , N} | \Phi (x_i;\theta) - u(x_i) |,
\end{align}
where $\{x_1, \ldots, x_N\}\subset \Omega$ is an i.i.d. uniform sampling in $\Omega$, and $x_0 = (0, 0, \ldots)\in \R^d$ is the origin.
Note that to approximate the $L^\infty$-error we add the origin to the sampling. Of course, this does not apply when $\Omega$ is an annulus.
When $\Omega$ is a cube or a ball, we observe that the highest error of the numerical solution is achieved near the origin, which is the furthest point from the boundary. Therefore, especially when the dimension is high, adding this point to the sampling improves the approximation of the $L^\infty$-error. In all the experiments below, we took the sample size for the Monte Carlo approximation as $N= 10^6$.

In the following sections, the run times for the experiments are recorded from running our PyTorch implementation on a desktop Dell with Intel Core i7-14700 CPU at 2.10GHz$\times 20$ with $15.8$ GiB of RAM.
The code is available at the following link:

\url{https://github.com/carlosesteveyague/HamiltonJacobi_LeastSquares_LxF_NNs}

\subsubsection{Eikonal equation}\label{sec:Eikonal-experiments}

To illustrate the algorithm presented in Section \ref{sec:algorithms}, we carry out some numerical experiments for the Eikonal equation in various $d$-dimensional domains:
\begin{equation}
\label{Eikonal numerical experiments}
\begin{cases}
    |\nabla u (x)|^2 = 1, &  x\in \Omega,  \\
    u(x) = 0, & x\in \partial\Omega.
\end{cases}
\end{equation}

Let us start with a simple example by considering the domain $\Omega:= (-3,3)^d$. The solution, in this case, corresponds to the distance function to the boundary of the $d$-dimensional cube and can be written as
\begin{equation}
\label{Eikonal ground truth}
u(x) = 3 - \|x\|_\infty, \qquad \forall x\in (-3,3)^d.
\end{equation}
We shall use the explicit solution as a ground truth to evaluate the accuracy of the numerical approximations. %

\paragraph{The two-dimensional case ($d=2$).}
We use Algorithm~\ref{alg: training}, with $M=5$, to train a fully connected Neural Network with one hidden layer with $20$ neurons. 
The parameters $\alpha$ and $\delta$ in each round/iteration are assigned according to the following schedules:
{
\begin{equation}
    \label{params alpha delta numerics}
    (\alpha_0, \alpha_1, \alpha_2, \alpha_3, \alpha_4) =  ( 2.5, 2, 1.5, 1, 0.5 )
    \quad  \text{and} \quad
    (\delta_0, \delta_1, \delta_2, \delta_3, \delta_4) = ( 0.75, 0.5, 0.3, 0.1, 0.05).
\end{equation}
}
Observe that, after each round, both parameters $\alpha$ and $\delta$ are decreased.
This is designed to progressively reduce the effect of the numerical viscosity to obtain higher fidelity in the approximate solution in the final round.
In each round, we applied Algorithm \ref{alg: SGD} with batch size $N_0 = 60$ for the collocation points sampled uniformly in $\Omega$, and $N_{b_1} = 20$ for the boundary points, sampled uniformly in $\partial\Omega$. {We did not consider any supervised data in the interior of the domain, i.e. $N_{b_2} = 0$. Concerning the weighting parameter for the boundary loss, we used $\gamma_1 = 1$.} 
The number of SGD iterations in each of the five rounds is
$K = 1000$.

{We ran 10 numerical experiments with i.i.d. initializations and computed the average and the standard deviation of the Mean Square error (MSE)  \eqref{MSE def} and the $L^\infty$-error \eqref{Linf def} after each round in the training schedule.
The results are reported in Figure \ref{fig:Eikonal 2-d error evol} and Table \ref{tab: experiment 2-dimensional cube}. The solid lines in Figure \ref{fig:Eikonal 2-d error evol} represent the average error over the 10 experiments, and the shadowed areas represent the interval within one standard deviation.
The run time for each experiment was 6 seconds on average.
In Figure \ref{fig:Eikonal 2-d NN evolution}, we can see, for one of the ten experiments, the numerical approximation provided by the NN at initialisation and after each of the five rounds in Algorithm \ref{alg: training}.

We see in Table \ref{tab: experiment 2-dimensional cube} that the standard deviation of the MSE and the $L^\infty$-error decreases for smaller values of $\alpha$ and $\delta$. However, if we look at the ratio between the standard deviation and the average error, we see that it increases throughout the training rounds. Indeed, after the first round this ratio is approximately $0.04$ and $0.03$ for the MSE and $L^\infty$-error respectively, whereas after the fifth round, these ratios are approximately $0.5$ and $0.35$.
This can also be observed in the graphical representation of the numerical experiment in Figure \ref{fig:Eikonal 2-d error evol}, which shows greater relative variability of the error in later training rounds.
We attribute this phenomenon to the fact that for larger $\delta$, each collocation point influences a larger portion of the domain, which leads to an enhanced performance of the training algorithm in terms of data efficiency (see Section \ref{subsec: data efficiency experiments} for a more detailed analysis). 
The flexibility in using different values of $\delta$ to control
the domain of influence for a finite set of collocation points may be another argument supporting the use of a finite-difference-based loss function over the PDE residual, which would correspond to the limit $\delta\to 0^+$.
}

\begin{figure}
    \centering
    \includegraphics[scale=.5]{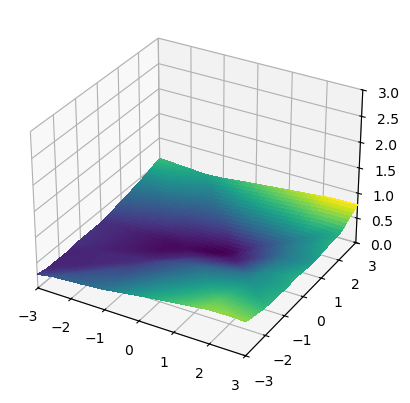}
    \includegraphics[scale=.5]{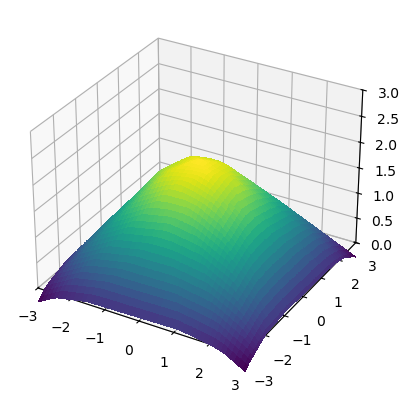}
    \includegraphics[scale=.5]{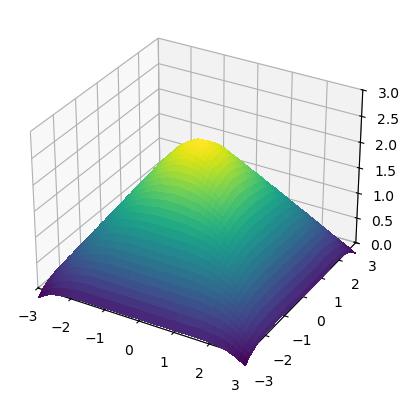}
    \includegraphics[scale=.5]{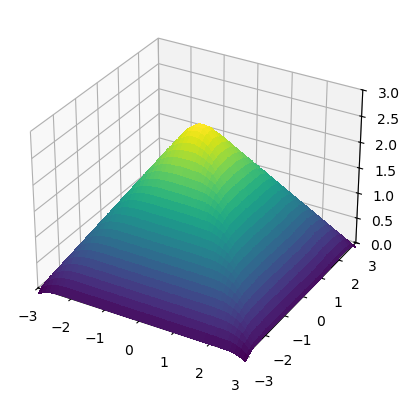}
    \includegraphics[scale=.5]{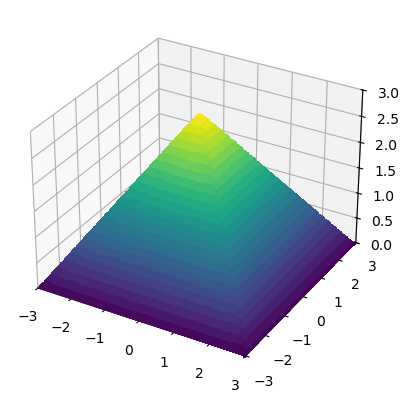}
    \includegraphics[scale=.5]{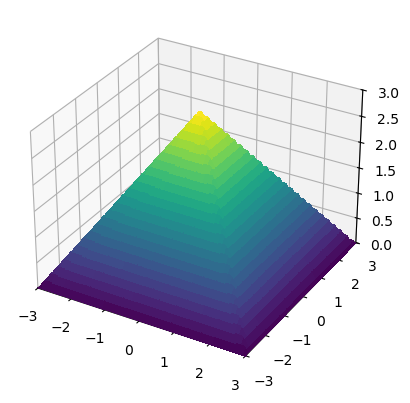}
    \caption{The numerical approximation of the Eikonal equation in $\Omega = (-3,3)^2$ constructed by Algorithm \ref{alg: SGD}. The top-left plot corresponds to the NN at initialisation. Then, from left to right and from top to bottom, each plot corresponds to the approximation after every iteration. The values of $\alpha$ and $\delta$ are taken as in \eqref{params alpha delta numerics}.}
    \label{fig:Eikonal 2-d NN evolution}
\end{figure}

\begin{figure}
    \begin{center}
        \includegraphics[scale=.5]{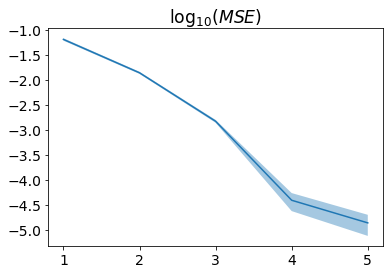}
        \includegraphics[scale=.5]{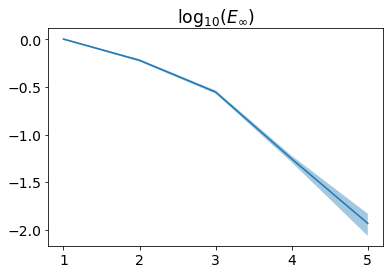}
    \end{center}
    \caption{\textbf{(Eikonal equation in a square in $\mathbb{R}^2$)} Evolution of the Mean Square Error and the $L^\infty$-Error after each round of Algorithm \ref{alg: training}, for the Eikonal equation in $(-3,3)^2$. The parameters $\alpha$ and $\delta$ in each of the five training rounds are taken as in \eqref{params alpha delta numerics}. The experiment is repeated 10 independent times. The solid line represents the average error, and the shadowed area represents the interval within one standard deviation.}
    \label{fig:Eikonal 2-d error evol}
\end{figure}

\begin{table}
    \caption{\textbf{(Eikonal equation in a square in $\mathbb{R}^2$)} Summary of the numerical experiment for the Eikonal equation in $\Omega = (-3,3)^2$. In each experiment, we applied 5 rounds of Algorithm \ref{alg: training} with decreasing values of $\alpha$ and $\delta$ as in \eqref{params alpha delta numerics}. The experiment is repeated 10 times using different random seeds. The error values in the table correspond to the mean and the standard deviation after each training round. See Figure \ref{fig:Eikonal 2-d error evol} for a graphical representation of the numerical results.}
    \label{tab: experiment 2-dimensional cube}
    \medskip\centering
   \begin{tabular}{lrrrrr}
\toprule
Round & 1 & 2 & 3 & 4 & 5 \\
\midrule
$\delta$ & 0.75 & 0.5 & 0.3 & 0.1 & 0.05 \\
$\alpha$ & 2.5 & 2 & 1.5 & 1 & 0.5 \\
\midrule
MSE (avg) & 6.45e-02 & 1.42e-02 & 1.55e-03 & 4.11e-05 & 1.45e-05 \\
MSE (std) & 2.68e-03 & 8.88e-04 & 1.21e-04 & 2.09e-05 & 7.02e-06 \\
\midrule
$L^\infty$-error (avg) & 1.01 & 6.12e-01 & 2.87e-01 & 6.07e-02 & 1.30e-02 \\
$L^\infty$-error (std) & 2.97e-02 & 1.96e-02 & 1.19e-02 & 6.93e-03 & 4.49e-03 \\
\bottomrule
\end{tabular}
\end{table}

\paragraph{Higher dimensional cases.}
Next, we solve \eqref{Eikonal numerical experiments}
in $(-3,3)^d, $  for $d=5$ and $d=8$.
In the experiments, we ran  Algorithm \ref{alg: training}  with $M=4$ and parameters $\alpha$ and $\delta$ following the schedule 
{
\begin{equation}
    \label{params alpha delta numerics highD}
    (\alpha_0, \alpha_1, \alpha_2, \alpha_3) =  ( 2.5, 2, 1, 0.5 )
    \quad  \text{and} \quad
    (\delta_0, \delta_1, \delta_2, \delta_3) = ( 0.75, 0.3, 0.1, 0.05).
\end{equation}
} As we increase the dimension, we need to increase the complexity of the NN and the number of collocation points. 
Rather than increasing the batch sizes $N_0$ and $N_{b_1}$ in \eqref{functional numerics}, we increase the number of iterations per round.
We compared the performance of 
simple 
fully connected MLP with 2 or 3 hidden layers. 

See Table \ref{tab: experiments d-dimensional cube} for a summary of the numerical experiments (each row corresponds to a different experiment).
{ Each experiment is repeated 10 times with different independent random seeds. The error values reported in the table represent the average error and the standard deviation over the 10 independent runs of each experiment.
In all the cases, we used a uniform distribution to sample the points in $\Omega$ and $\partial\Omega$ during the iterations of Algorithm \ref{alg: SGD}.
The weighting parameter for the boundary loss is taken, again, as $\gamma_1 = 1$.}
In Figure \ref{fig:error-evol cube}, we see the evolution of the MSE and the $L^\infty$-error throughout the application of Algorithm \ref{alg: training}, for experiments 4,5 and 6, which correspond to the Eikonal equation in an $8$-dimensional cube.
{ The solid lines represent the average error after each training round, and the shadowed area represents the interval within one standard deviation.}

As mentioned in section \ref{sec:algorithms}, the reason behind decreasing the values of $\alpha$ and $\delta$ throughout the iterations of Algorithm \ref{alg: training}, is to improve the accuracy of the NN.
However, this does not always produce a more accurate approximation. See in Figure \ref{fig:error-evol cube} that the { variance of the approximation error increases over the iterations of Algorithm \ref{alg: training}. 
We think that this is due to $\delta$ being too small relative to the number of collocation points and the network's expressive power.
See a more detailed discussion in subsection \ref{subsec: data efficiency experiments}.
This issue can be alleviated by increasing the number of collocation points. Since the collocation points are resampled at every SGD iteration, one can increase the total number of collocation points by increasing the number of SGD iterations.}

\begin{table}
    \caption{\textbf{(Eikonal equation in a hypercube in $\mathbb{R}^d$)} Summary of the numerical experiments for the Eikonal equation in $\Omega = (-3,3)^d$ with $d=5,8$. In each experiment, we applied 4 iterations of Algorithm \ref{alg: training} with decreasing values of $\alpha$ and $\delta$ as in \eqref{params alpha delta numerics highD}. Each experiment is repeated 10 times using different random seeds. The error values in the table correspond to the mean and the standard deviation. See Figure \ref{fig:error-evol cube} for the evolution of the MSE and the $L^\infty$-error after each iteration. In all the experiments, the batch size in the implementation of SGD in \eqref{functional numerics} is taken as $N_0 = 200$ and $N_{b_1} = 80$.}
    \label{tab: experiments d-dimensional cube}
    \medskip\centering
    \begin{tabular}{lrllllr}
\toprule
 Exp. & Dimension & Architecture & Iterations ($\times 10^3$) & MSE & $L^\infty$-error & Runtime (s) \\
\midrule
1 & 5 & [30, 30] & $\{ 2, 2, 2, 2\}$ & 0.013 ± 0.002 & 0.606 ± 0.079 & 62 \\
2 & 5 & [30, 30, 30] & $\{2, 2, 2, 2\}$ & 0.01 ± 0.007 & 0.527 ± 0.119 & 85 \\
3 & 5 & [30, 30, 30] & $\{2, 2.5, 3, 5\}$ & 0.006 ± 0.004 & 0.435 ± 0.137 & 133 \\
\midrule
4 & 8 & [40, 40] & $\{3, 3, 3, 3\}$ & 0.047 ± 0.006 & 1.129 ± 0.18 & 154 \\
5 & 8 & [40, 40, 40] & $\{3, 3, 3, 3\}$ & 0.031 ± 0.011 & 0.87 ± 0.131 & 209 \\
6 & 8 & [40, 40, 40] & $\{3, 4, 5, 7\}$ & 0.012 ± 0.006 & 0.751 ± 0.245 & 331 \\
\bottomrule
\end{tabular}
\end{table}

\begin{figure}
    \centering
    \includegraphics[scale=0.5]{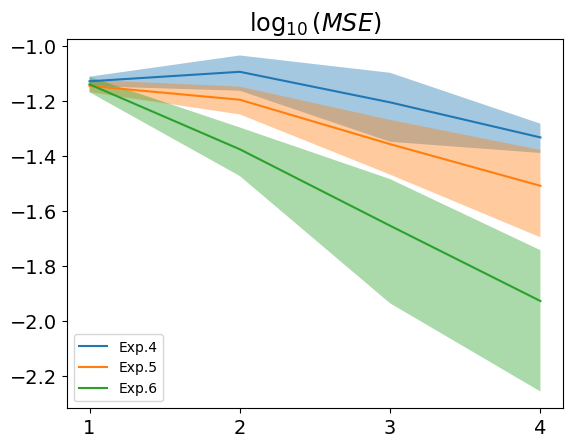}
    \includegraphics[scale=0.5]{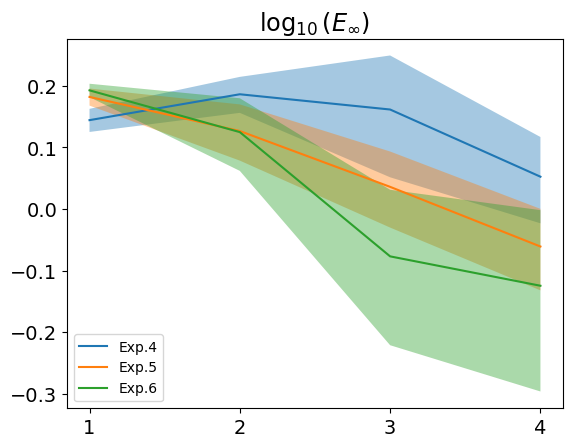}
    \caption{\textbf{(Eikonal equation in a hypercube in $\mathbb{R}^8$)} Evolution of the MSE (left) and the $L^\infty$-error (right) throughout the iterations of Algorithm \ref{alg: training}, for the Experiments 4, 5 and 6 reported in Table \ref{tab: experiments d-dimensional cube}. 
    These correspond to the Eikonal equation in an 8-dimensional cube. Each experiment is repeated 10 times with different random seeds. The solid lines represent the average error over the 10 experiments, and the shadowed area represents the interval within one standard deviation. 
    }
    \label{fig:error-evol cube}
\end{figure}

In Table \ref{tab: experiments d-dimensional annulus}, we can see the results of some numerical experiments for the solution of the Eikonal equation in an annulus of dimension $d=5$ and $10$, with inner radius $2$ and outer radius $6$, i.e.
$$
\Omega := \{
x\in \R^d \ : \quad 2 < |x| < 6
\}.
$$
We applied $4$ rounds of Algorithm \ref{alg: training} with decreasing values of $\alpha$ and $\delta$, following the training schedule in \eqref{params alpha delta numerics highD}.
See Figure \ref{fig:error-evol annulus} for the evolution of the MSE and the $L^\infty$-error of the solution after each iteration.
See Figure \ref{fig:annulus 10d} for a representation of the numerical solution in a central cross-section of the domain.

\begin{table}
    \caption{\textbf{(Eikonal equation in an annulus in $\mathbb{R}^{d}$)} Summary of the numerical experiments for the Eikonal equation in $\Omega = \{ x\in \R^d \ : \ 2 < |x| < 6  \}$ with $d=5,10$, i.e. the $d$-dimensional annulus with inner and outer radii $2$ and $6$ respectively. Each experiment is repeated 10 times with different random seeds. The error values in the table correspond to the mean and the standard deviation. In each experiment, we applied 4 iterations of Algorithm \ref{alg: training} with decreasing values of $\alpha$ and $\delta$ as in \eqref{params alpha delta numerics highD}. See Figure \ref{fig:error-evol annulus} for the evolution of the MSE and the $L^\infty$-error after each iteration. In all the experiments, the batch size in the implementation of SGD in \eqref{functional numerics} is taken as $N_0 = 200$ and $N_{b_1} = 80$.}
    \label{tab: experiments d-dimensional annulus}
    \medskip\centering
    \begin{tabular}{lrllllr}
    \toprule
    Exp. & Dim. & Architecture & Iterations ($\times 10^3$) & MSE & $L^\infty$-error & Runtime (s) \\
    \midrule
    1 & 5 & [40, 40] & $\{1.5, 1.5, 1.5, 1.5\}$ & 0.019 ± 0.007 & 0.611 ± 0.092 & 62 \\
    2 & 5 & [40, 40, 40] & $\{1.5, 1.5, 1.5, 1.5\}$ & 0.009 ± 0.003 & 0.434 ± 0.078 & 90 \\
    3 & 5 & [40, 40, 40] & $\{1.5, 2, 2.5, 3\}$ & 0.012 ± 0.013 & 0.396 ± 0.054 & 135 \\
    \midrule
    4 & 10 & [60, 60] & $\{2.5, 2.5, 2.5, 2.5\}$ & 0.026 ± 0.003 & 0.797 ± 0.043 & 80 \\
    5 & 10 & [60, 60, 60] & $\{2.5, 2.5, 2.5, 2.5\}$ & 0.015 ± 0.004 & 0.608 ± 0.059 & 103 \\
    6 & 10 & [60, 60, 60] & $\{2.5, 3, 3.5, 5\}$ & 0.014 ± 0.002 & 0.552 ± 0.04 & 137 \\
    \bottomrule
    \end{tabular}
\end{table}

\begin{figure}
    \centering
    \includegraphics[scale=0.5]{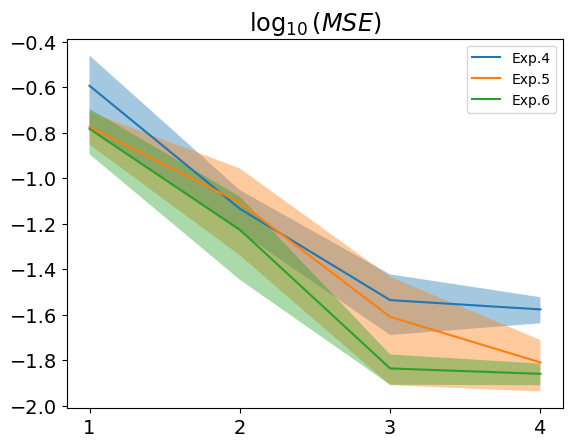}
    \includegraphics[scale=0.5]{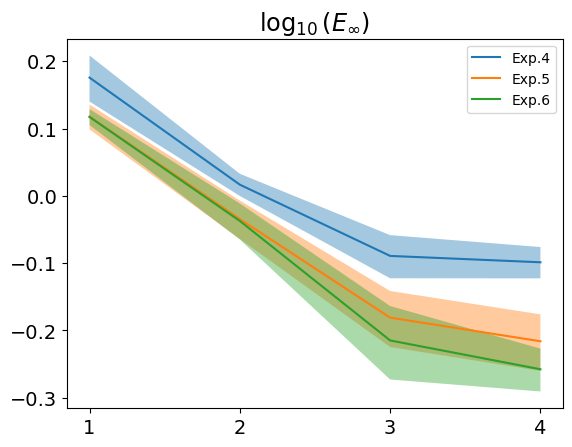}
    \caption{\textbf{(Eikonal equation in an annulus in $\mathbb{R}^{10}$)} Evolution of the MSE (left) and the $L^\infty$-error (right) throughout the iterations of Algorithm \ref{alg: training}, for the experiments 4,5 and 6 reported in Table \ref{tab: experiments d-dimensional annulus}. These correspond to the Eikonal equation in a 10-dimensional annulus. 
    Each experiment is repeated 10 times with different random seeds. The solid lines represent the average error over the 10 experiments, and the shadowed area represents the interval within one standard deviation.
    }
    \label{fig:error-evol annulus}
\end{figure}

\begin{figure}
    \centering
    \includegraphics[scale = .5]{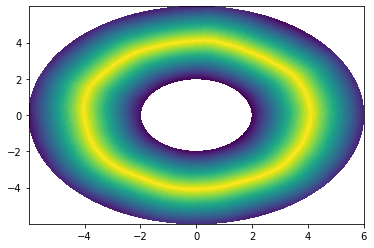}
    \caption{Cross-section of the numerical solution provided by the trained NN in one of the ten independent runs of Experiment 6 in Table \ref{tab: experiments d-dimensional annulus}. The NN approximates the viscosity solution of the Eikonal equation in a 10-dimensional annulus.}
    \label{fig:annulus 10d}
\end{figure}

\subsubsection{Time-dependent Hamilton-Jacobi equation}
\label{subsec: numerics time-dependent}

As mentioned in Section~\ref{subsec: time-dependent pbms}, our method can be applied to time-evolution Hamilton-Jacobi equations of the form \eqref{time-evol HJ}.
Let us consider the initial value problem
\begin{equation}
\label{IVP Riccati}
\begin{cases}
    \vspace{5pt}
    \partial_t u + \dfrac{|\nabla u|^2}{2} + \dfrac{|x|^2}{2} = 0, & (t,x)\in (0,T)\times \R^d, \\
    u_0 (x) = \dfrac{1}{2} \left( \langle x, \, Ax \rangle - 1\right), & x\in \R^d,
\end{cases}
\end{equation}
where $A = \operatorname{diag} (4/25, \, 1, \, 1, \, \ldots)$ and $T=1/2$.
This example is taken from Example 2 in \cite{chow2019algorithm}.

{
By the Riccati Theory, see e.g. \cite{fleming2012deterministic}, the solution to \eqref{IVP Riccati} is given by
\begin{equation}
\label{solution Riccati}
u(t,x) = \dfrac{1}{2} \left( \langle x, \, E(t) x\rangle - 1 \right),
\end{equation}
where $E(\cdot): [0,T)\to \R^{d\times d}$ is the quadratic form, solution to the Riccati differential equation
\begin{equation}
\label{Riccati diff eq}
\begin{cases}
    \dot{E} (t) = - E(t)^2 - I,  & t\in (0,T) \\
    E(0) = A,
\end{cases}
\end{equation}
where $I$ denotes the identity matrix in $\R^d$.
The closed-form solution \eqref{solution Riccati} for the PDE problem \eqref{IVP Riccati} will be used as ground truth to evaluate the accuracy of our method.
}

We consider the bounded domain $\Omega := (-3,3)^d$, and minimize the functional
$$
    \mathcal{J}(u) := \mathcal{R}(u) + \gamma_0 \int_\Omega \left( u(x,0) - u_0(x) \right)^2 dx,
$$
where $\mathcal{R}(u)$ is defined as in \eqref{R(u) time dep}, and $\gamma_0 = 1$.
Note that we do not use supervised data for the entire boundary of $(0, 1/2) \times \Omega$. We only prescribe the initial condition on $\{0\} \times (-3,3)^d$.
The experiment compares the trained network with the whole-space Cauchy solution \eqref{solution Riccati}, restricted to this cube. The restricted problem is
\begin{equation}
    \label{IVP experiment}
    \begin{cases}
    \partial_t u + H(x,\nabla u)=0 & (x,t)\in(-3,3)^d\times(0,1/2),\\
    u(x,0)=u_0(x) & x\in(-3,3)^d.
    \end{cases}
\end{equation}
Initial data on the cube alone do not determine the reference solution throughout the space-time domain. Indeed, the characteristics satisfy
$$
x(t)=(\cos t\,I+\sin t\,A)x(0),
$$
and the first-coordinate multiplier at $t=1/2$ is $\cos(1/2)+(4/25)\sin(1/2)\approx0.9543<1$, so characteristics can enter through the lateral boundary. The results below are therefore an empirical comparison with the whole-space reference solution, without a uniqueness or critical-point guarantee for \eqref{IVP experiment}.

{
In Table \ref{tab:experiment time-dependent}, we report the results of various numerical experiments for the problem with $d=2$ and $d=5$. In all cases, we consider the training schedule
\begin{align}
(\alpha_0, \alpha_1, \alpha_2, \alpha_3) &= (2.5, 2, 1.5, 1) \nonumber \\
(\delta_{x,0}, \delta_{x,1}, \delta_{x,2}, \delta_{x,3}) &= (0.5, 0.3, 0.2, 0.1) \label{schedule time-dependent} \\
(\delta_{t,0}, \delta_{t,1}, \delta_{t,2}, \delta_{t,3}) &= (0.05, 0.03, 0.02, 0.01). \nonumber
\end{align}
The reported schedule is not covered by a general monotonicity guarantee: for example, its first stage in dimension five has $d\alpha\delta_t/\delta_x=1.25>1$. Each experiment reported in Table \ref{tab:experiment time-dependent} has been run 10 times independently, and the errors represent the mean and the standard deviation over the 10 experiments.
Figure \ref{fig: time dep} shows the evolution of the $0$-level set of the approximated solution for the two-dimensional case, after each training round (from left to right). These errors have been computed by taking the solution \eqref{solution Riccati} as the ground truth. The solution in \eqref{solution Riccati} has been computed by solving the Riccati equation \eqref{Riccati diff eq} by means of RK45 method implemented in the Python package Scipy.
}

\begin{table}   
    \caption{\textbf{(Time-evolution Hamilton-Jacobi equation)} Summary of the experiments for the initial value problem \eqref{IVP experiment}, with $d=2$ and $d=5$. Each experiment is repeated 10 times with different random seeds. The MSE and the $L^\infty$-error is computed with respect to the solution given by \eqref{solution Riccati}, obtained by numerically solving the Riccati equation \eqref{Riccati diff eq}.}
    \label{tab:experiment time-dependent}\medskip
\centering
    
\begin{tabular}{lrllllr}
\toprule
Exp. & Dim. & Architecture & Iterations ($\times 10^3$) & MSE & $L^\infty$-error & Runtime (s) \\
\midrule
1 & 2 & [50, 50] & \{2, 2, 2, 2\} & 0.008 ± 0.002 & 0.672 ± 0.102 & 60 \\
2 & 2 & [50, 50, 50] & \{2, 2, 2, 2\} & 0.005 ± 0.002 & 0.596 ± 0.088 & 62 \\
3 & 2 & [50, 50, 50] & \{2, 3, 5, 8\} & 0.002 ± 0.001 & 0.463 ± 0.093 & 137 \\
\midrule
4 & 5 & [80, 80] & \{8, 8, 8, 8\} & 0.021 ± 0.006 & 1.579 ± 0.257 & 294 \\
5 & 5 & [80, 80, 80] & \{8, 8, 8, 8\} & 0.012 ± 0.003 & 1.422 ± 0.323 & 402 \\
6 & 5 & [80, 80, 80] & \{8, 10, 15, 20\} & 0.009 ± 0.002 & 1.13 ± 0.261 & 643 \\
\bottomrule
\end{tabular}
\end{table}

\begin{figure}
    \centering
    \includegraphics[scale = .4]{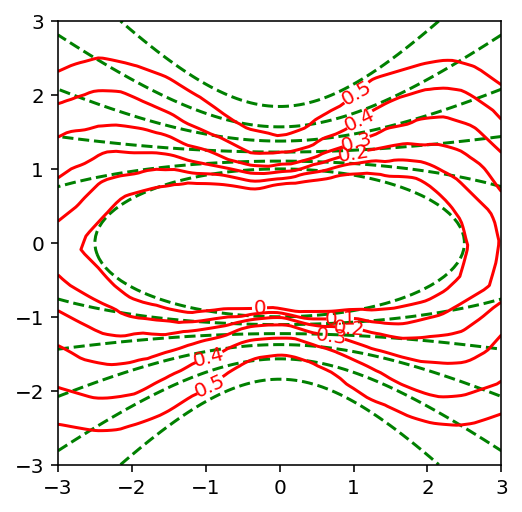}
    \includegraphics[scale = .4]{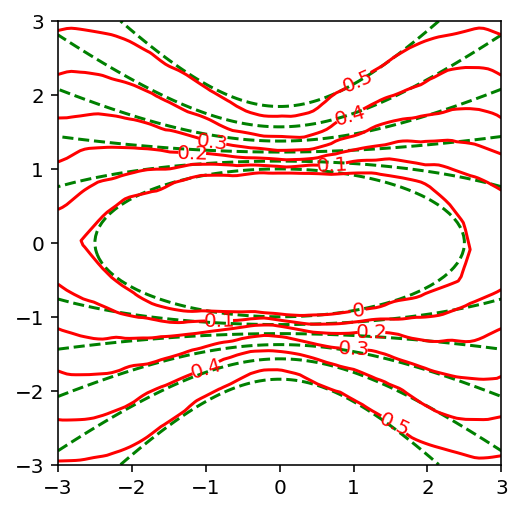}
    \includegraphics[scale = .4]{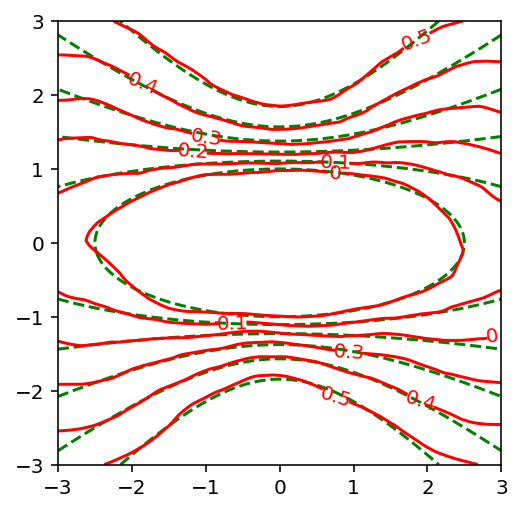}
    \includegraphics[scale = .4]{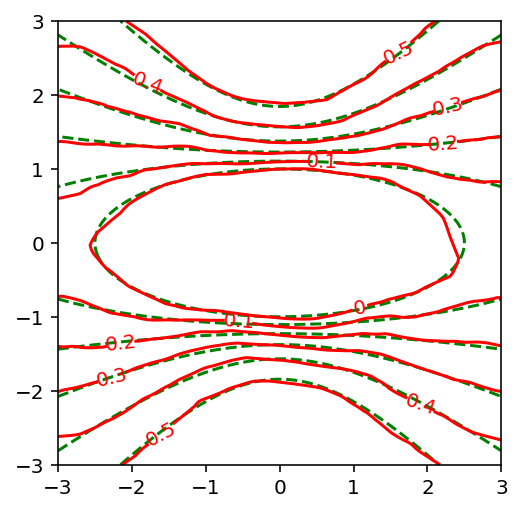}
    \caption{
     \textbf{(Time-evolution Hamilton-Jacobi equation)} For the problem \eqref{IVP Riccati}, with $d=2$, the plots represent the $0$-level set of the numerical solution at times $t\in\{0,0.1,0.2,0.3,0.4,0.5\}$, provided by an NN corresponding to Exp. 3 in Table \ref{tab:experiment time-dependent}. 
    From left to right, the plots represent the NN solution (in red) after training with $\alpha, \delta_x$ and $\delta_t$ as in the schedule \eqref{schedule time-dependent}. The green dashed line represents the solution obtained by numerically solving the Riccati Differential equation \eqref{Riccati diff eq}.
    In the plots, the zero-level sets of the approximate solutions, $\{x\in \R^d \, : \, \Phi(x, t;\theta)=0 \}$, are depicted for $t = 0, 0.1, 0.2, 0.3, 0.4$ and $0.5$.
    The label for the curves reveals the corresponding values of $t$.
    }
    \label{fig: time dep}
\end{figure}

\subsubsection{Shortest path to a target for a Reeds-Shepp's car}
\label{subsec: car OCP}

Here, we consider an optimal control problem for a dynamical system with curvature constraints.
In particular, we consider Reeds and Shepp's car model \cite{reeds1990optimal}, which can move backward and forward with a limited turning radius.
The pose of the car is represented by $(x,y,\omega)\in \R^2 \times \mathbb{T}_{[0,2\pi)}$, where $(x,y)$ represents the spatial position of the centre of mass of the car and $\omega$ represents its orientation angle.
Here, $\mathbb{T}_{[0,2\pi)}$ denotes the one-dimensional flat torus of length $2\pi$.
The optimal control problem involves steering the car to the origin $(x,y) = 0$ from the given initial position $(x_0, y_0, \omega_0)$ in the shortest possible time.

Denoting by $t\mapsto (x(t),y(t), \omega(t))\in \R^2\times \mathbb{T}_{[0,2\pi)}$ the trajectory of the car, the dynamics are given by the following ODE:
\begin{equation}
\label{car dynamics}
\begin{cases}
    \dot{x} (t) = \sigma a(t) \cos \omega (t) \\
    \dot{y} (t) = \sigma a(t) \sin \omega(t) \\
    \dot{\omega} (t)  = \dfrac{b(t)}{\rho} \\
    x(0) = x_0, \ y(0) = y_0,  \ \omega(0) = \omega_0,
\end{cases}
\end{equation}
where $\sigma >0$ is the maximum speed of the car, $\rho>0$ is the inverse of the angular velocity of the car (the turning radius is therefore $\sigma \rho$) and $(x_0,y_0, \omega_0)\in \R^2\times \mathbb{T}_{[0,2\pi)}$ is the initial pose of the car.
The time dependent functions $a(\cdot), b(\cdot): [0,\infty) \to [-1, 1]$ are the controls of the car.

Given the dynamics of the car, the Hamiltonian associated with the shortest path problem is given by
$$
H(x,y,\omega, \nabla u) = \sigma | \partial_x u \cos \omega + \partial_y u \sin \omega | + \dfrac{1}{\rho} |\partial_\omega u| - 1,
$$
This Hamiltonian, and the game Hamiltonian below, contain absolute values and are not differentiable at all momenta. These numerical examples therefore extend beyond assumption~\eqref{cond H}; the differentiable critical-point theorems are not invoked at those momenta. Here $\nabla u=(\partial_x u,\partial_y u,\partial_\omega u)$ represents the gradient of $u: \R^2\times \mathbb{T}_{[0,2\pi)}\to \R$ with respect to the pose of the car.
The target is the ball of radius $r = 0.2$ centred at the origin.
Since our framework applies to bounded domains, we set $R = 5$, and define the domain
$$
\Omega := \mathbb{A}_{r,R}\times \mathbb{T}_{[0,2\pi)},
$$
where $\mathbb{A}_{r,R}$ denotes the two-dimensional annulus with inner radius $r$ and outer radius $R$
\begin{equation}
\label{annulus def}
\mathbb{A}_{r,R} := \{ (x,y)\in \R^2\, : \ r < \sqrt{x^2 + y^2} < R \}.
\end{equation}
As for the boundary conditions, we set $0$ on the inner boundary of the annulus and an imposed terminal penalty $R$ on the outer boundary. Hence, the truncated boundary value problem reads as
\begin{equation}
\label{HJB car}
\begin{cases}
     H (x,y,\omega, \nabla u) = 0 & (x,y,\omega)\in \mathbb{A}_{r,R} \times \mathbb{T}_{[0,2\pi)} \\
     u (x, y, \omega) = 0 & (x,y,\omega) \in \{ \sqrt{x^2 + y^2} = r \} \times \mathbb{T}_{[0,2\pi)} \\
     u (x, y, \omega) = R & (x,y,\omega) \in \{ \sqrt{x^2 + y^2} = R \} \times \mathbb{T}_{[0,2\pi)}.
\end{cases}
\end{equation}

We used our approach to approximate the viscosity solution of the above boundary value problem using a NN.
Since the solution $u(x,y,\omega)$ is $2\pi$-periodic with respect to $\omega$, we considered a parameterised family of functions of the form
$$
\Phi(x,y,\omega;\theta)=c_0(x,y;\theta)+\sum_{n=1}^{9}
\bigl[c_n(x,y;\theta)\cos(n\omega)+s_n(x,y;\theta)\sin(n\omega)\bigr],
$$
where the 19 coefficients are produced jointly by a fully connected ReLU network with two hidden layers of 60 neurons and a linear output layer.
We trained the NN for three rounds of Algorithm \ref{alg: training}, using the following values for the parameters $\alpha$ and $\delta$:
$$
(\alpha_0,\alpha_1,\alpha_2)=(2.5,\,2.5,\,2.5)
\quad \text{and} \quad
(\delta_0,\delta_1,\delta_2)=(0.75,0.6,0.3).
$$
In each application of Algorithm \ref{alg: SGD}, we run $3000$ iterations of stochastic gradient descent, using $N_{b_1} = 100$ boundary points for the boundary data, and in each of the three rounds, we took $N= [800, 1000, 1200]$ as number of collocation points to approximate the finite-difference loss.
The time it took to train the NN was $90$ seconds on a CPU.

{
Once trained, the NN can be used, as an approximation of the value function $u(x,y,\omega)$, to obtain a feedback control for the car.
At points where the value function is differentiable, the formal optimal feedback control is obtained as
$$
\left( a^\ast (x,y,\omega), b^\ast (x,y,\omega) \right) \in
\arg \min_{(a,b)\in [-1,1]^2} \left( 
\sigma a \partial_x u \cos \omega + \sigma a \partial_y u \sin \omega + \dfrac{1}{\rho} b \partial_\omega u + 1
\right),
$$
which yields
\begin{equation*}
    a^\ast (x,y,\omega) := - \operatorname{sgn} \left( \partial_x u \cos \omega + \partial_y u \sin \omega \right), 
    \quad \text{and} \quad
    b^\ast (x,y,\omega) := -\operatorname{sgn} \left( \partial_\omega u \right).
\end{equation*}
In practice, we use the NN $\Phi(x,y,\omega;\theta)$ as an approximation of the value function $u(x,y,\omega)$, and the partial derivatives are replaced by finite-difference approximations as
\begin{align*}
& \partial_x \Phi (x,y,\omega; \theta) \approx \dfrac{\Phi (x + \delta,y,\omega; \theta) - \Phi (x,y,\omega; \theta)}{\delta},
\quad
\partial_y \Phi (x,y,\omega; \theta) \approx \dfrac{\Phi (x,y + \delta,\omega; \theta) - \Phi (x,y,\omega; \theta)}{\delta} 
\\
& \partial_\omega \Phi (x,y,\omega; \theta) \approx \dfrac{\Phi (x,y,\omega + \delta; \theta) - \Phi (x,y,\omega; \theta)}{\delta}.
\end{align*}
where $\delta>0$ is the parameter used in the last step of the training schedule.}
Figure \ref{fig: cars OCP} shows some trajectories obtained by the trained Neural Network.

\begin{figure}
    \centering
    \includegraphics[scale = 0.35]{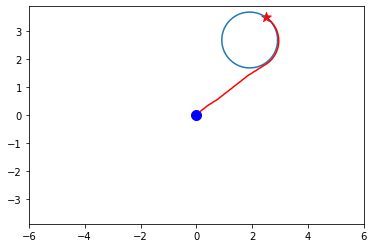}
    \includegraphics[scale = 0.35]{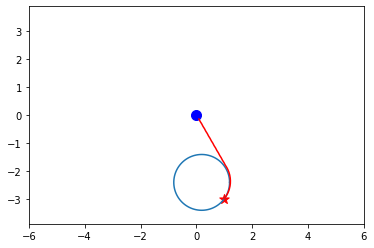}
    \includegraphics[scale = 0.35]{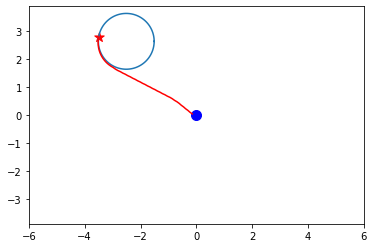}
    \includegraphics[scale = 0.35]{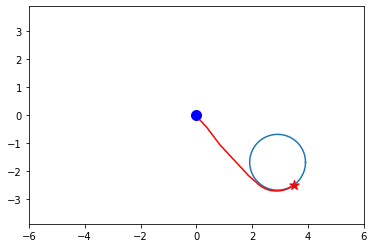}
    \includegraphics[scale = 0.35]{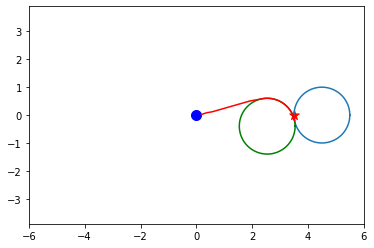}
    \includegraphics[scale = 0.35]{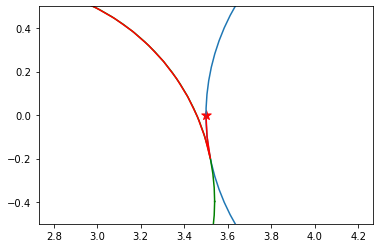}
    \caption{Five sample trajectories for the Reeds-Shepp's car, obtained using Neural Network approximating the viscosity solution of \eqref{HJB car}. The underlying optimal control problem is the shortest path problem to the origin (represented by the blue dot), subject to the curvature-constrained dynamics in  \eqref{car dynamics}. The red star represents the initial position of the car, and the blue and green circles represent the turning circumference from the initial position and orientation of the car. The bottom { right} figure is the fifth trajectory, zoomed-in around the initial position. The car manoeuvres to find a shorter path to the target.}
    \label{fig: cars OCP}
\end{figure}

\subsubsection{Pursuit-evasion game with Reeds-Shepp's cars}
\label{subsec: pursuit evasion game}

As an example of applying our method to a Hamilton-Jacobi equation arising from differential game theory, we consider the pursuit-evasion game for two players that move according to Reeds and Shepp's car model \eqref{car dynamics}.
Given the initial poses for the Evader and the Pursuer, represented by
$$
(x_e, y_e, \omega_e) \in \R^2 \times \mathbb{T}_{[0,2\pi)} \qquad \text{and} \qquad (x_p, y_p, \omega_p) \in \R^2 \times \mathbb{T}_{[0,2\pi)},
$$
we consider the game in which the Pursuer's goal is to intercept the Evader in the shortest possible time, and the Evader's goal is to delay the interception for as long as possible.

Since the position of the game only depends on the relative position between the players, we introduce the variables $X = x_e - x_p$ and $Y = y_e - y_p$. The vector $(X,Y)$ represents the vector from the Pursuer to the Evader.
The position of the game is therefore represented by $(X,Y, \omega_e, \omega_p) \in \R^2 \times \mathbb{T}_{[0,2\pi)}^2$.
Given the dynamics of the two cars in \eqref{car dynamics}, 
we can write down the equation for the game as
\begin{equation}
\label{car game dynamics}
\begin{cases}
    \dot{X} (t) = \sigma_e a_e(t) \cos \omega_e (t)
    - \sigma_p a_p(t) \cos \omega_p (t)\\
    \dot{Y} (t) = \sigma_e a_e(t) \sin \omega_e(t)
    - \sigma_p a_p(t) \sin \omega_p(t)\\
    \dot{\omega}_e (t)  = \dfrac{b_e(t)}{\rho_e}, \qquad
    \dot{\omega}_p (t)  = \dfrac{b_p(t)}{\rho_p} \\
    X(0) = x_e - x_p, \ Y(0) = y_e - y_p,  \ \omega_e(0) = \omega_e, \ \omega_p (0) = \omega_p,
\end{cases}
\end{equation}
where $\sigma_e$ and $\sigma_p$ represent the maximum velocities for the Evader and the Pursuer, respectively, and $\rho_e$ and $\rho_p$ represent the inverse of their angular velocity (the turning radius for the evader and the pursuer is $\sigma_e \rho_e$ and $\sigma_p \rho_p$ respectively). The functions $a_e (\cdot), a_p(\cdot): [0, \infty) \to [-1,1]$ are the controls for the velocities of the players, and $b_e(\cdot), b_p(\cdot) : [0, \infty) \to [-1,1]$ are the controls for the direction.

In this case, we define the domain of the game as
$$
\Omega := \mathbb{A}_{r,R} \times \mathbb{T}_{[0, 2\pi)}^2,
$$
with $r = 0.2$, $R=4$, where $\mathbb{A}_{r,R}$ is the two-dimensional annulus defined in \eqref{annulus def}, and $\mathbb{T}_{[0, 2\pi)}^2$ is the two-dimensional flat torus with side-length $2\pi$. Recall that the position of the game is represented by the vector joining the players $(X,Y) = (x_e - x_p,\, y_e - y_p)\in \R^2$ and the orientation angle of the players $(\omega_e, \omega_p)\in \mathbb{T}_{[0,2\pi)}^2$.
The target for the Pursuer is the inner boundary of the annulus $\mathbb{A}_{r,R}$, where we set zero boundary condition. The target for the Evader is the outer boundary, where we impose a terminal payoff of $10$.

Since the Pursuer minimises the output of the game and the Evader maximises it, we can derive the Hamiltonian associated with the game, which reads as
$$
H(X, Y, \omega_e, \omega_p, \nabla u) :=
\sigma_p \left| \partial_X u \cos \omega_p + \partial_Y u \sin \omega_p \right|
+ \dfrac{1}{\rho_p} \left|\partial_{\omega_p} u \right|
- \sigma_e \left| \partial_X u \cos \omega_e + \partial_Y u \sin \omega_e  \right|
-\dfrac{1}{\rho_e} \left|\partial_{\omega_e} u\right|-1,
$$
where, for a function $u: \Omega \to \R$ defined on the game domain,  $\nabla u = (\partial_X u, \partial_Y u, \partial_{\omega_e} u, \partial_{\omega_p} u)$ denotes the gradient of $u$ with respect to the position of the game.
The Hamilton-Jacobi-Isaacs equation associated with the game reads as follows:
\begin{equation}
\label{HJI car}
\begin{cases}
     H (X, Y, \omega_e, \omega_p, \nabla u) = 0 & (X,Y,\omega_e, \omega_p)\in \mathbb{A}_{r,R} \times \mathbb{T}_{[0,2\pi)}^2 \\
     u(X,Y,\omega_e,\omega_p)=0 & (X,Y,\omega_e,\omega_p)\in\{\sqrt{X^2+Y^2}=r\}\times\mathbb{T}_{[0,2\pi)}^2 \\
     u(X,Y,\omega_e,\omega_p)=10 & (X,Y,\omega_e,\omega_p)\in\{\sqrt{X^2+Y^2}=R\}\times\mathbb{T}_{[0,2\pi)}^2.
\end{cases}
\end{equation}
The viscosity solution of this boundary value problem characterises the value of the truncated game. At differentiability points, the formal optimal feedback strategies are
\begin{equation}
\label{feedback P}
a_p (X,Y,\omega_e, \omega_p) = \operatorname{sgn} \left(\partial_X u \cos \omega_p + \partial_Y u \sin \omega_p \right),
\qquad
b_p (X,Y,\omega_e, \omega_p) = -\operatorname{sgn} \left( \partial_{\omega_p} u \right)
\end{equation}
for the Pursuer, and 
\begin{equation}
\label{feedback E}
a_e (X,Y,\omega_e, \omega_p) =  \operatorname{sgn} \left(\partial_X u \cos \omega_e + \partial_Y u \sin \omega_e \right),
\qquad
b_e (X,Y,\omega_e, \omega_p) = \operatorname{sgn} \left( \partial_{\omega_e} u \right)
\end{equation}
for the Evader.

For different values of the players' velocities $(\sigma_e, \rho_e)$ and $(\sigma_p, \rho_p)$, we used our approach to approximate the viscosity solution of \eqref{HJI car} using a NN. 
Similarly to the previous example, since the solution $u(X,Y,\omega_e, \omega_p)$ is $2\pi$-periodic with respect to $\omega_e$ and $\omega_p$, we considered a parameterised family of functions of the form
\[
\Phi(X,Y,\omega_e,\omega_p;\theta)=\sum_{j,k=0}^{6}
c_{jk}(X,Y;\theta)b_j(\omega_e)b_k(\omega_p),
\]
with the angular basis
\[
b(\omega)=(1,\cos\omega,\cos2\omega,\cos3\omega,
\sin\omega,\sin2\omega,\sin3\omega).
\]
The 49 coefficients $c_{jk}$ are produced jointly by a fully connected ReLU network with two hidden layers of 80 neurons and a linear output layer. The tensor-product basis includes mixed dependence on both orientation angles.
We trained the NN for three rounds of Algorithm \ref{alg: training}, using the following values for the parameters $\alpha$ and $\delta$:
$$
(\alpha_0, \alpha_1, \alpha_2) = (2.5, 2, 1.5)
\quad
\text{and}
\quad
(\delta_0, \delta_1, \delta_2) = (0.7, 0.5, 0.3).
$$
In each application of Algorithm \ref{alg: SGD}, we run $3000$ iterations of stochastic gradient descent, using $N_{b_1} = 100$ points for the boundary data, and $N=800$ collocation points to approximate the least-squares finite-difference loss.
The running time to train each NN was 120 seconds on a CPU.

In Figure \ref{fig: game trajectories}, each sample trajectory of the game is computed using the feedback control according to  
\eqref{feedback P} and \eqref{feedback E}, with 
the trained NN approximating the value function $u$. As in the previous example, the partial derivatives are approximated by finite-differences.
In the first example, the Pursuer moves faster than the Evader ($\sigma_p > \sigma_e$), but the Evader has a higher angular velocity ($1/\rho_e > 1/\rho_p$).
The plot at the right represents the evolution of the distance between the players throughout the game. We see an oscillatory behaviour due to the turns that the Evader performs for not being intercepted.
{In some cases we see that, near the maxima of the curve, the distance between the players increases abruptly before it starts decreasing. This is due to the fact that the Pursuer has the possibility to back up in order to change the direction faster (similarly to the last example in Figure \ref{fig: cars OCP})}.
In the second example, we see a trajectory in which both players have the same angular velocity, but the pursuer moves faster.
Hence, the turning radius of the Evader is smaller ($\sigma_e \rho_e < \sigma_p \rho_p$). In this case, the game enters in a loop.
Finally, in the third example, we see a trajectory in which the pursuer moves faster, and both players have the same turning radius ($\sigma_e \rho_e = \sigma_p \rho_p$). In this case, the Pursuer has a clear advantage over the Evader, and interception occurs quickly.
{ See \url{https://github.com/carlosesteveyague/HamiltonJacobi_LeastSquares_LxF_NNs} for other animated representations of trajectories related to this Pursuit-Evasion game.}

\begin{figure}
\begin{tabular}{c c c}
    \begin{minipage}{.1\textwidth}
        \begin{align*}
        \sigma_e = 0.8 \\ 
        \rho_e =  1 \\
        \sigma_p = 1 \\  
        \rho_p = 1.2
        \end{align*}
    \end{minipage}
    & 
    \begin{minipage}{.37\textwidth}
    \includegraphics[width = \linewidth]{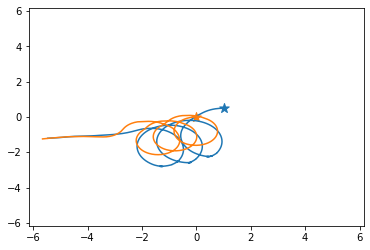}
    \end{minipage} & 
    \begin{minipage}{.37\textwidth}
    \includegraphics[width = \linewidth]{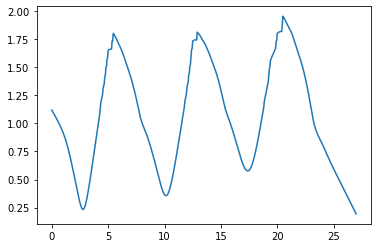}
    \end{minipage}
    \\
    \begin{minipage}{.1\textwidth}
        \begin{align*}
        \sigma_e = 0.8 \\ 
        \rho_e =  1 \\
        \sigma_p = 1 \\  
        \rho_p = 1
        \end{align*}
    \end{minipage}
    & 
    \begin{minipage}{.37\textwidth}
    \includegraphics[width = \linewidth]{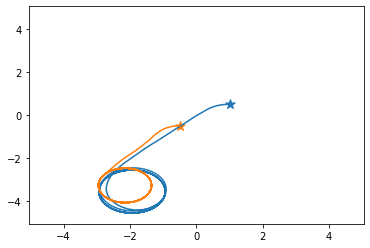}
    \end{minipage} & 
    \begin{minipage}{.37\textwidth}
    \includegraphics[width = \linewidth]{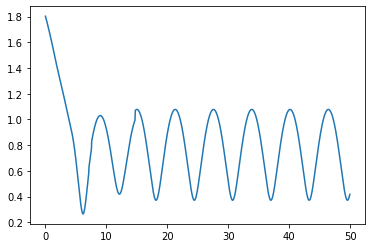}
    \end{minipage}
    \\
    \begin{minipage}{.1\textwidth}
        \begin{align*}
        \sigma_e = 0.8 \\ 
        \rho_e =  1 \\
        \sigma_p = 1 \\  
        \rho_p = 0.8
        \end{align*}
    \end{minipage}
    & 
    \begin{minipage}{.37\textwidth}
    \includegraphics[width = \linewidth]{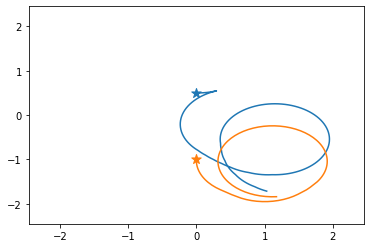}
    \end{minipage} & 
    \begin{minipage}{.37\textwidth}
    \includegraphics[width = \linewidth]{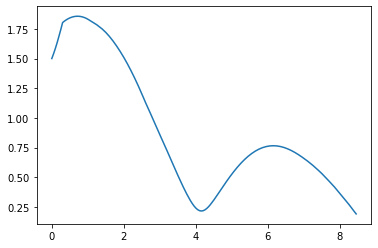}
    \end{minipage}
    \\
    Players' velocities 
    &
    Game trajectory
    &
    Time versus distance between players
\end{tabular}
\caption{Three trajectories of the Pursuit Evasion game with Reeds-Shepp's cars discussed in Section~\ref{subsec: pursuit evasion game}. We consider different velocity parameters for the players (left) and different initial positions shown by the stars. 
}
\label{fig: game trajectories}
\end{figure}

{

\section{Qualitative analysis of the algorithm}
\label{sec: analysis}

In this section, we analyse various aspects of our proposed algorithm using judiciously designed numerical experiments.
In particular, we want to show how costly it is to train a NN to approximate the global minimizer of a functional of the form
\begin{equation}
\label{loss function analysis section}
\mathcal{J}(u) :=
\int_\Omega \left[ 
\widehat{H}_\alpha (x, D_\delta^+ u(x), D_\delta^- u(x))\right]^2 d\rho(x)
+
\gamma_1 \int_{\partial\Omega} \left( u(x) - g(x)\right)^2 d\mu(x)
+
\gamma_2 \int_{\Gamma} \left( u(x) - u^\ast (x) \right)^2 d\nu (x),
\end{equation}
where $\widehat{H}_\alpha (x, p^+, p^-)$ is the Lax-Friedrichs numerical Hamiltonian defined in \eqref{H_LxF}, $\Gamma\subset \Omega$ is a subset of the domain where the viscosity solution $u^\ast$ is assumed to be known, and $\rho,\mu$ and $\nu$ are probability distributions on $\Omega$, $\partial\Omega$ and $\Gamma$ respectively. In all the experiments in this section, we consider the weighting parameters $\gamma_1 = \gamma_2 =1$.

\begin{enumerate}
    \item In Section \ref{subsec: data efficiency experiments}, the experiments will show that, to prevent deterioration of generalization errors, 
    the number of collocation points should increase as $\delta$ decreases.
    \item Also in Section  \ref{subsec: data efficiency experiments}, we will see that resampling the collocation points at each SGD iteration drastically increases the performance of the trained NNs.
    \item In Section \ref{subsec: data distribution experiments}, we study the choice of the distribution for the collocation points $\rho$ in \eqref{loss function analysis section}. We will see that, especially in high-dimensional scenarios, a suitable choice of $\rho$ may greatly impact the accuracy of the obtained numerical solution.
    \item In Section \ref{subsec: adding supervised data}, we examine how supervised data in the interior of $\Omega$ affect the observed training success rate for smaller stencil sizes, as discussed in Section~\ref{subsec: supervised data}.
\end{enumerate}
For simplicity, we will always consider the Eikonal equation in simple domains.

\subsection{Data efficiency and sampling strategy}
\label{subsec: data efficiency experiments}

As argued in Section \ref{sec:why-fdm} and Section \ref{sec:algorithms}, minimising the functional $\mathcal{J}(u)$ with smaller values of $\delta$ may demand more collocation points in the domain $\Omega$ due to a smaller range of influence per point.
While we cannot prove it rigorously, we illustrate this phenomenon through simple numerical experiments. In our tests, we consider three different scenarios in terms of training data (collocation points + boundary points), and compare the generalisation power of the NN after being trained with decreasing values of $\delta$ and a fixed parameter $\alpha$.

For each value of $\delta$, the NN is trained until the training error is close to zero. When the training data is insufficient, the finite-difference residual
$\left[  \widehat{H}_\alpha (x, D_\delta^+ u(x) , D_\delta^- u(x))\right]^2$
is not necessarily small at points away from the training data.

We consider the Eikonal equation \eqref{Eikonal numerical experiments} in a $d$-dimensional ball of radius $3$, i.e. $\Omega = B(0, 3)\in \R^d$, first with $d=2$ and then with $d=5$.
The viscosity solution to this boundary value problem is given by
\begin{equation}
\label{ground truth Eikonal disc}
u(x) = 3 - \| x\| , \qquad
\text{for} \ x \in \Omega,
\end{equation}
which will be used as the ground truth solution to evaluate the accuracy of the numerical solution.

To study the data efficiency during the training process, we trained three-layer fully connected MLPs (with $20$ neurons in each hidden layer for $d=2$, and with $30$ neurons for $d=5$) by applying Algorithm~\ref{alg: training} to minimise the loss functional $\mathcal{J} (\cdot)$ in \eqref{loss function analysis section}, with $\Gamma = \emptyset$ (i.e. we consider supervised data only on the boundary).
The loss functional $\mathcal{J} (\cdot)$ is approximated, using Monte Carlo approximation, as 
$$
\mathcal{J} \left( \Phi (\cdot; \theta) \right) \approx \tilde{\mathcal{J}} \left( \Phi (\cdot; \theta) \right) := \dfrac{1}{N_0} \sum_{x\in X} \left[\widehat{H}_\alpha ( D_\delta^+ \Phi(x;\theta), D_\delta^- \Phi(x;\theta) )\right]^2 + \dfrac{1}{N_b} \sum_{x\in X_b} \left(\Phi (x ; \theta) \right)^2
$$
where the collocation points $X = \{   x_j\}_{j = 1}^{N_0}$ are uniformly sampled from $\Omega$ and the boundary points $X_{b} = \{ x_j \}_{j=1}^{N_b}$ are uniformly sampled from $\partial \Omega$.
As for the training schedule, we considered $5$ training steps ($M=5$ in Algorithm \ref{alg: training}) in which the parameters $\alpha$ and $\delta$ in the numerical Hamiltonian are set as follows:
\begin{equation}
\label{training schedule data efficiency}
\alpha_i = 2 \quad \forall i = 0, \ldots, 4
\quad \text{and} \quad
(\delta_0 , \delta_1, \delta_2, \delta_3, \delta_4) = ( 0.7, 0.5, 0.2, 0.1, 0.05).
\end{equation}

In the first two scenarios, the collocation points $X$ and $X_b$ are sampled once at the beginning (we shall refer to them as training data) and then re-used throughout the training of the NN.
In particular, for $d=2$, we considered $N_0 = 80$ and $N_b = 40$ in the first case, and $N_0 = 160$ and $N_b = 80$ in the second case.
For $d=5$, we considered $N_0 = 200$ and $N_b = 40$ in the first case, and $N_0 = 400$ and $N_b = 80$ in the second case.
In both cases, we trained the NN for $500$ epochs using SGD with a batch size of $20$ for $d=2$ and $50$ for $d=5$. The training data is randomly shuffled and partitioned into batches in each epoch. Then, a gradient step is performed using the data in each batch. One epoch refers to when all the training data has been used once.

In the third case, the training data $X$ and $X_b$ are resampled at every SGD iteration, as described in  Algorithm \ref{alg: SGD}. In particular, for $d=2$, we carried $J = 2500$ SGD iterations\footnote{Note that $500$ epochs of SGD with a dataset of size $120$ splitted in batches of $20$, as in the first case, corresponds to $3000$ iterations.} with $N_0 = 20$ and $N_b = 8$.
For $d=5$, we carried $J = 2500$ SGD iterations with $N_0 = 300$ and $N_b = 50$.
To put the three cases into the same setting, the latter case can be seen as a single training epoch with a dataset of size $(20+8)\times 2500 = 70000$ for $d=2$, and $(300+50)\times 2500 = 875000$ for $d=5$.

We test the trained NNs on a uniform grid in $\Omega$ with $\delta=0.01$ for $d=2$ and $\delta=0.3$ for $d=5$. We shall refer to the grid nodes as the 
test data, and denote it by $X_t\subset \Omega$. 
It contains around $N_t\approx 2.8\times 10^5$ points for $d=2$ and around $N_t\approx 5.3\times 10^5$ for $d=5$. 

Figure \ref{fig: FD loss heatmap} shows the heat maps of the finite-difference residual $x\mapsto \left[ \widehat{H}_\alpha (D_\delta^+ u(x), D_\delta^- u(x)) \right]^2$ after training the NN with different values of $\delta$ for the two dimensional case. The top and middle rows of figures correspond to the two first scenarios described above, where the collocation points (represented by red dots) are sampled at the beginning and then re-used throughout the training. The areas not covered by the training data appear brighter since the finite-difference residual is larger. This effect is mitigated for larger values of $\delta$, showing improved data efficiency in this case. In the bottom row, we show the heat maps from using the resampling strategy described above. The training data is not represented in this case due to visualisation purposes. We see that the finite-difference residual is smaller, also for small values of $\delta$.

\begin{figure}
    \centering
    \begin{subfigure}{.185\textwidth}
  \centering
  \includegraphics[width=\linewidth]{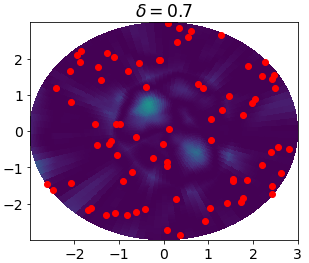}
  \includegraphics[width=\linewidth]{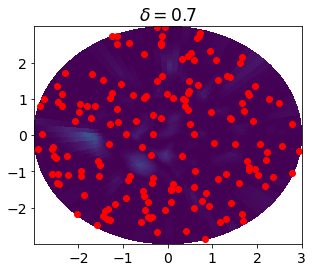}
  \includegraphics[width=\linewidth]{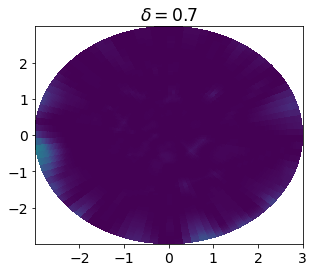}
  \caption{$\delta = 0.7$}
\end{subfigure}
\begin{subfigure}{.185\textwidth}
  \centering
  \includegraphics[width=\linewidth]{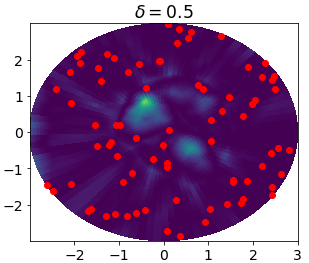}
  \includegraphics[width=\linewidth]{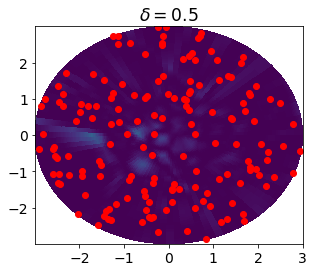}
  \includegraphics[width=\linewidth]{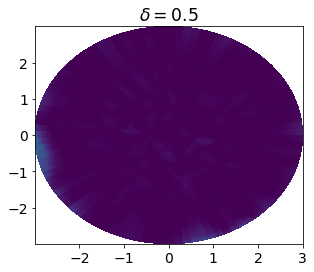}
  \caption{$\delta = 0.5$}
\end{subfigure}
\begin{subfigure}{.185\textwidth}
  \centering
  \includegraphics[width=\linewidth]{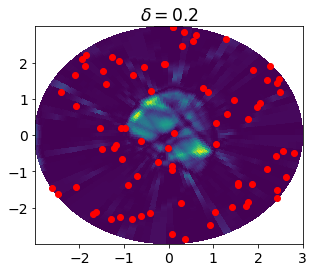}
  \includegraphics[width=\linewidth]{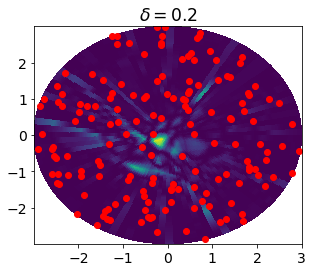}
  \includegraphics[width=\linewidth]{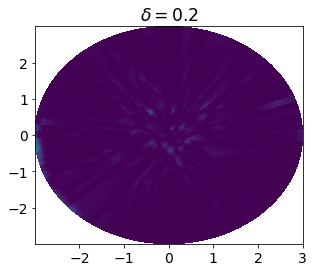}
  \caption{$\delta = 0.2$}
\end{subfigure}
\begin{subfigure}{.185\textwidth}
  \centering
  \includegraphics[width=\linewidth]{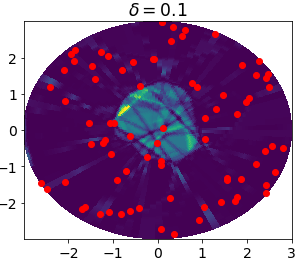}
  \includegraphics[width=\linewidth]{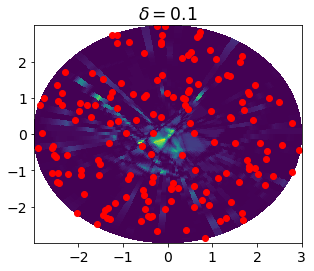}
  \includegraphics[width=\linewidth]{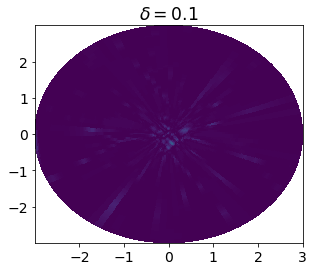}
  \caption{$\delta = 0.1$}
\end{subfigure}
\begin{subfigure}{.22\textwidth}
  \centering
  \includegraphics[width=\linewidth]{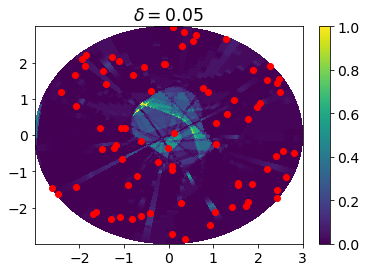}
  \includegraphics[width=\linewidth]{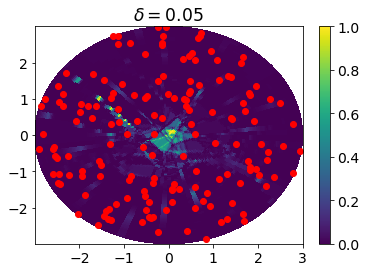}
  \includegraphics[width=\linewidth]{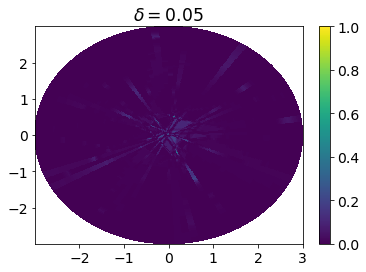}
  \caption{$\delta = 0.05$}
\end{subfigure}
    \caption{Heat map of the finite-difference residual $x\mapsto \left[ \widehat{H}_\alpha (D_\delta^+ u(x), D_\delta^- u(x)) \right]^2$.
    We show one example for each training scenario described in section \ref{subsec: data efficiency experiments}: $N_0 = 80$ collocation points (top), $N_0=160$ collocation points (middle) and using the resampling strategy (bottom). The collocation points for the first two cases are represented by red dots. The training data for the resampling strategy is not represented as it covers most of the domain.}
    \label{fig: FD loss heatmap}
\end{figure}

Figures \ref{fig: FD loss comparison} and \ref{fig: data efficiency error comparison} show a summary of ten independent experiments using each of the three training data scenarios described above, for the case $d=2$. Figure \ref{fig: FD loss comparison} shows the evolution of the finite-difference residual, evaluated on the test data.
The left plot represents the average over the test data, and the right plot shows the maximum value:
\begin{equation}
\label{FD residual mean and Linf}
\widehat{\mathcal{R}} (u) = 
\dfrac{1}{N_t} \sum_{x\in X_t} \left[ \widehat{H} (D_\delta^+ u(x), D_\delta^- (x)) \right]^2
\qquad \text{and} \qquad
\widehat{\mathcal{R}}_\infty (u) = 
\max_{x\in X_t} \left[ \widehat{H} (D_\delta^+ u(x), D_\delta^- (x)) \right]^2.
\end{equation}
In Figure \ref{fig: data efficiency error comparison} we see the evolution of the mean square error \eqref{MSE def} and the $L^\infty$-error \eqref{Linf def}, with respect to the ground truth solution \eqref{ground truth Eikonal disc}, as the parameter $\delta$ is decreased following the training schedule in \eqref{training schedule data efficiency}. The solid line shows the averaged value over the 10 experiments, and the shadowed area shows the interval within the standard deviation.
Figures \ref{fig: FD loss comparison 5D} and \ref{fig: data efficiency error comparison 5D} show the same summary of ten independent experiments using each of the three training data scenarios described above for the case $d=5$.

In Figures \ref{fig: data efficiency error comparison} and \ref{fig: data efficiency error comparison 5D}, we observe that the error with respect to the ground truth solution decreases as $\delta$ is reduced.
On the contrary, the finite-difference loss evaluated on test data does not decrease. It actually increases if we look at the maximum over the test data (right plots in Figures \ref{fig: FD loss comparison} and \ref{fig: FD loss comparison 5D}). 
Although these two statements might seem to be in contradiction, we stress that, when $\delta$ is smaller, the improved accuracy in terms of error with respect to the ground truth solution is due to less numerical diffusion in the discretised Hamiltonian.
In conclusion, reducing $\delta$ has two opposite effects on the error with respect to the ground truth solution: on one hand, the reduced numerical diffusion enhances the accuracy, and on the other hand, the smaller range of influence per collocation point makes training less data efficient.
We also observe that the resampling strategy gives better results, and this improvement is more noticeable for smaller values of $\delta$. In Figure \ref{fig: data efficiency error comparison 5D}, we see that when the dimension is higher, the improvement of the resampling strategy is more substantial.
We also note that, when using resampling, the standard deviation over the ten experiments is smaller, indicating that this strategy makes the method more robust.

\begin{figure}
    \centering
    
\includegraphics[width=.4\linewidth]{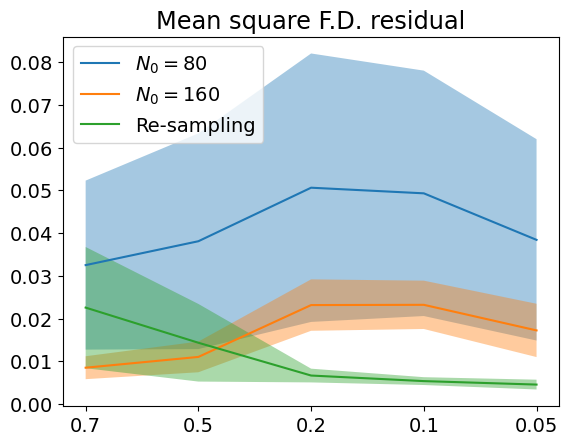}
\includegraphics[width=.4\linewidth]{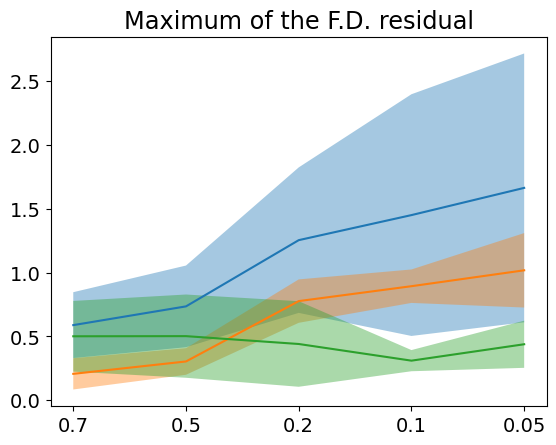}

    \caption{\textbf{(Residual versus $\delta$ for the eikonal equation in a disc)} The plots show the evolution of the finite-difference residual, evaluated on the test data $X_t$, as the parameter $\delta$ is reduced according to \eqref{training schedule data efficiency}, for the problem \eqref{Eikonal numerical experiments} with $\Omega = B(0,3)$ and $d=2$. We show the average over the domain (left) $\widehat{\mathcal{R}} (u)$ and its maximum (right) $\widehat{\mathcal{R}}_\infty (u)$, as defined in \eqref{FD residual mean and Linf}. The solid line represents the mean value computed over ten independent experiments, and the shadowed area represents the interval within one standard deviation.}
    \label{fig: FD loss comparison}
\end{figure}

\begin{figure}
    \centering

\includegraphics[width=.4\linewidth]{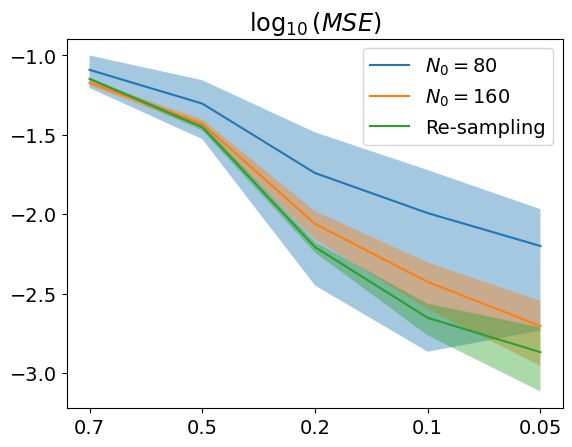}
\includegraphics[width=.4\linewidth]{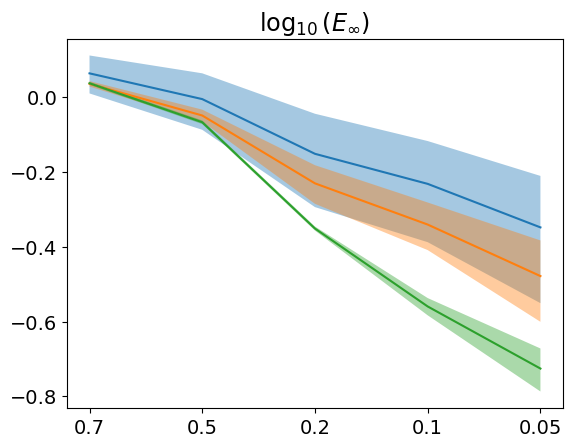}

    \caption{\textbf{(Error versus $\delta$ for the eikonal equation in a disc)} Evolution of the error with respect to the ground truth solution \eqref{ground truth Eikonal disc} with $d=2$, evaluated on the test data $X_t$, as the parameter $\delta$ is reduced according to \eqref{training schedule data efficiency}. We show the mean square error (left) and the $L^\infty$-error (right), as defined in \eqref{MSE def} and \eqref{Linf def} respectively. The solid line represents the mean value computed over ten independent experiments, and the shadowed area represents the interval within one standard deviation.}
    \label{fig: data efficiency error comparison}
\end{figure}

\begin{figure}
    \centering
    
\includegraphics[width=.4\linewidth]{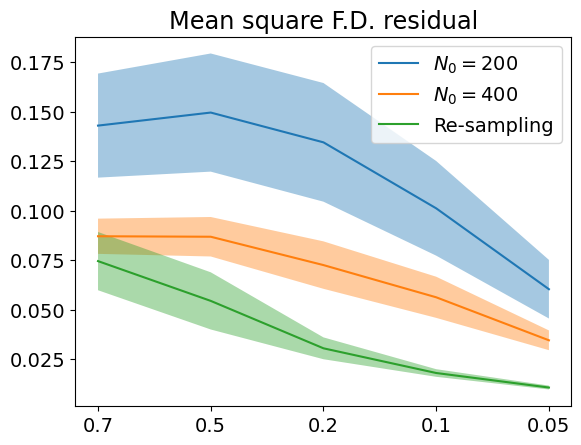}
\includegraphics[width=.4\linewidth]{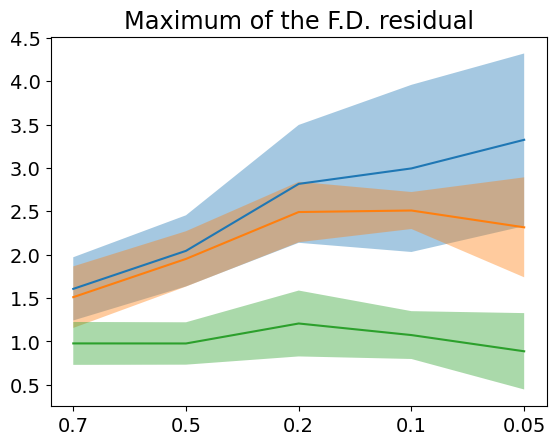}

    \caption{\textbf{(Residual versus $\delta$ for the eikonal equation in a ball in $\R^5$)} Evolution of the finite-difference residual, evaluated on the test data $X_t$, as the parameter $\delta$ is decreased according to \eqref{training schedule data efficiency}, for the problem \eqref{Eikonal numerical experiments} with $\Omega = B(0,3)$ and $d=5$. We show the average over the domain (left) $\widehat{\mathcal{R}} (u)$ and its maximum (right) $\widehat{\mathcal{R}}_\infty (u)$, as defined in \eqref{FD residual mean and Linf}. The solid line represents the mean value computed over ten independent experiments, and the shadowed area represents the interval within one standard deviation.}
    \label{fig: FD loss comparison 5D}
\end{figure}

\begin{figure}
    \centering

\includegraphics[width=.4\linewidth]{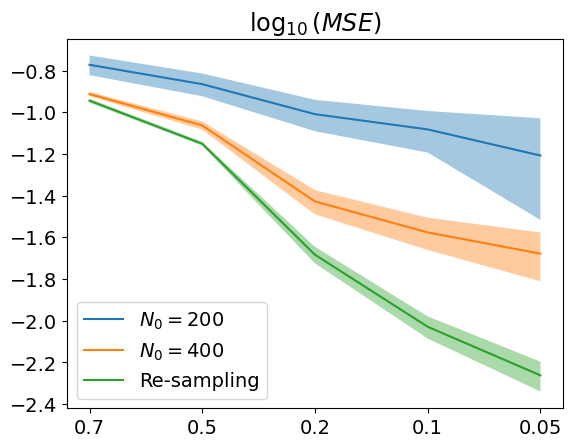}
\includegraphics[width=.4\linewidth]{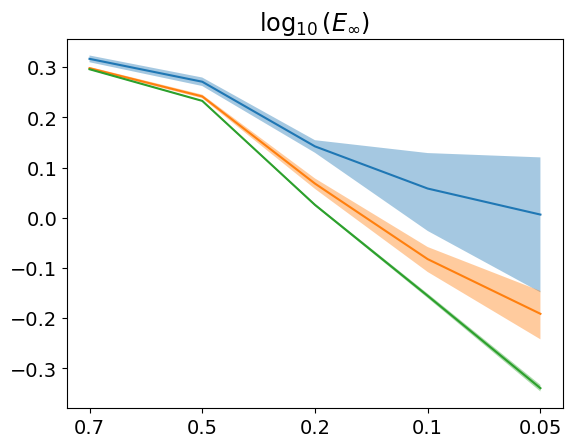}

    \caption{\textbf{(Error versus $\delta$ for the eikonal equation in a ball in $\R^5$)} Evolution of the error with respect to the ground truth solution \eqref{ground truth Eikonal disc} with $d=5$, evaluated on the test data $X_t$, as the parameter $\delta$ is decreased according to \eqref{training schedule data efficiency}. We show the mean square error (left) and the $L^\infty$-error (right), as defined in \eqref{MSE def} and \eqref{Linf def} respectively. The solid line represents the mean value computed over ten independent experiments, and the shadowed area represents the interval within one standard deviation.}
    \label{fig: data efficiency error comparison 5D}
\end{figure}

\subsection{Data distribution in high-dimensional scenarios}
\label{subsec: data distribution experiments}

When the dimension is high, the sampling distribution used during the iterations of Algorithm \ref{alg: SGD} dramatically impacts the numerical solution.
Let the domain $\Omega$ be a $d$-dimensional ball.
Sample points of the uniform distribution over $\Omega$ with respect to the Lebesgue measure, 
will be closer to the boundary with higher probability. This is simply because the outer rings of the ball occupy a larger volume (of the order $\sim r^{d-1}$).
This ``concentration" causes the accuracy near the centre of the ball to be relatively poor compared to the accuracy of the solution near the boundary.
This issue can be alleviated by using a different distribution for the collocation points.
In the case when $\Omega$ is a $d$-dimensional ball of radius $R$, one possibility is to use a radially uniform sampling, i.e. 
\begin{equation}
\label{radially unif sample}
X = rv,
\quad \text{where} \ r\sim \operatorname{Unif}(0,R) \ \text{and} \
v\sim \operatorname{Unif} (\mathbb{S}^{d-1}),
\end{equation}
where $\mathbb{S}^{d-1}$ denotes the $(d-1)-$dimensional unit sphere.

In Table \ref{tab: experiments d-dimensional ball}, we see the summary of the numerical experiments for the Eikonal equation in a $10$-dimensional ball of radius $6$.
Again, for each experiment we carried out $4$ rounds of Algorithm \ref{alg: training}, with parameters $\alpha$ and $\delta$ following the schedule
\begin{equation}
    \label{params alpha delta numerics D-ball}
    (\alpha_0, \alpha_1, \alpha_2, \alpha_3) =  ( 2.5, 2, 1, 0 )
    \quad  \text{and} \quad
    (\delta_0, \delta_1, \delta_2, \delta_3) = ( 0.7, 0.3, 0.1, 0.01).
\end{equation}
Each experiment is repeated 10 times with different random seeds. The error values shown in the table represent the average and one standard deviation.
The evolution of the MSE and the $L^\infty$-error throughout the iterations of Algorithm \ref{alg: training} is shown in Figure \ref{fig:error-evol ball}.
We observe that the accuracy of the numerical solution in terms of MSE is similar when using a uniform distribution or a radially uniform distribution to sample the collocation points in $\Omega$. See Table \ref{tab: experiments d-dimensional ball} and Figure \ref{fig:error-evol ball} (left).
However, in terms of $L^\infty$-error, the accuracy is remarkably better when using the radially uniform sampling defined in \eqref{radially unif sample}.
See Table \ref{tab: experiments d-dimensional ball} and Figure \ref{fig:error-evol ball} (right).
This is because, when using uniform sampling, most of the points are close to the boundary, and the accuracy near the centre of the ball is poorer. See Figure \ref{fig:Eikonal 10-d sampling} for a representation of the numerical solution using different sampling distributions for the Eikonal equation in the $10$-dimensional ball $\Omega = B_{10} (0,6)$.
In the plot on the left, we used a uniform sampling for the collocation points during the optimisation iterations. In contrast, in the plot at the right, we used the radially uniform sampling described in \eqref{radially unif sample}. 
In both cases, the plots represent the NN evaluated on a uniform grid of the two-dimensional cross-section $B_2(0,6)\times \{ 0 \}^8$.

\begin{table}
    \caption{Summary of the numerical experiments for the Eikonal equation in $\Omega = B_{10} (0, 6)$, i.e. the $10$-dimensional ball of radius 6. In each experiment, we applied 4 iterations of Algorithm \ref{alg: training} with decreasing values of $\alpha$ and $\delta$, following the schedule \eqref{params alpha delta numerics D-ball}. See Figure \ref{fig:error-evol ball} for the evolution of the MSE and the $L^\infty$-error after each iteration. In all the experiments, the batch size in the implementation of SGD in \eqref{functional numerics} is taken as $N_0 = 200$ and $N_{b_1} = 80$.  We see that the use of a radially uniform distribution \eqref{radially unif sample} for the collocation points (experiments 4,5 and 6) improves the accuracy in terms of $L^\infty$-error.}
    \label{tab: experiments d-dimensional ball}
    \medskip\centering
    \begin{tabular}{lrlllllr}
    \toprule
     Exp. & Dim. & Distr. & Architecture & Iterations ($\times 10^3$) & MSE & $L^\infty$-error & Runtime (s) \\
    \midrule
    1 & 10 & Unif. & [40, 40] & $\{1.5, 1.5, 1.5, 1.5\}$ & 0.02 ± 0.003 & 2.355 ± 0.202 & 55 \\
    2 & 10 & Unif. & [40, 40, 40] & $\{1.5, 1.5, 1.5, 1.5\}$ & 0.015 ± 0.004 & 2.773 ± 0.077 & 74 \\
    3 & 10 & Unif. & [40, 40, 40] & $\{1.5, 2, 2.5, 3\}$ & 0.017 ± 0.008 & 2.646 ± 0.148 & 115 \\
    \midrule
    4 & 10 & Rad. & [40, 40] & $\{1.5, 1.5, 1.5, 1.5\}$ & 0.026 ± 0.004 & 0.768 ± 0.053 & 55 \\
    5 & 10 & Rad. & [40, 40, 40] & $\{1.5, 1.5, 1.5, 1.5\}$ & 0.031 ± 0.011 & 0.795 ± 0.064 & 77 \\
    6 & 10 & Rad. & [40, 40, 40] & $\{1.5, 2, 2.5, 3\}$ & 0.017 ± 0.003 & 0.619 ± 0.034 & 121 \\
    \bottomrule
    \end{tabular}
\end{table}

\begin{figure}
    \centering
    \includegraphics[scale=0.5]{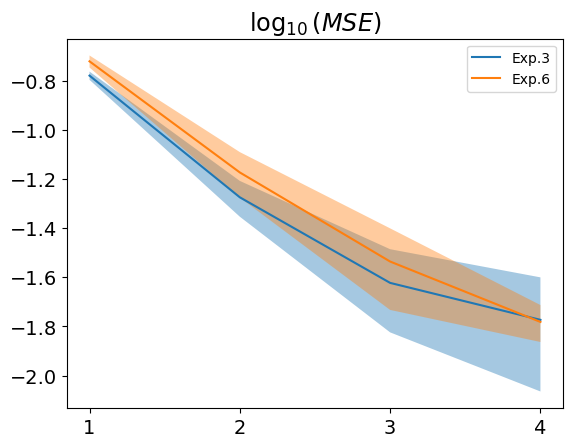}
    \includegraphics[scale=0.5]{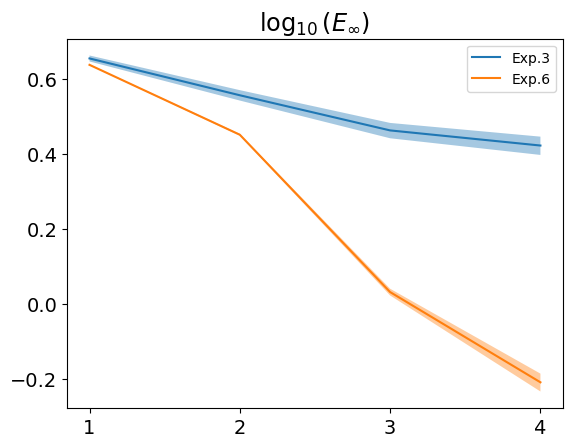}
    \caption{\textbf{(Error versus iteration index for the eikonal equation in a ball in $\R^{10}$)}Evolution of the MSE (left) and the $L^\infty$-error (right) throughout the iterations of Algorithm \ref{alg: training}, for the experiments 3 and 6 reported in Table \ref{tab: experiments d-dimensional ball}. These correspond to the Eikonal equation in a 10-dimensional ball.}
    \label{fig:error-evol ball}
\end{figure}

\begin{figure}
    \centering
    \includegraphics[scale=.5]{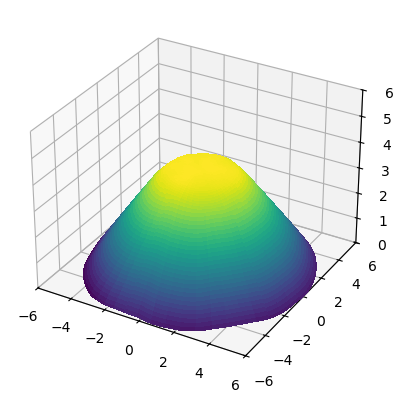}
    \includegraphics[scale=.5]{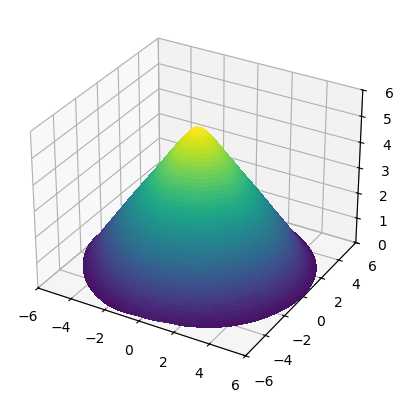}
    \caption{\textbf{Eikonal equation in a ball in $\R^{10}$)} The numerical solution for the Eikonal equation in $\Omega = B_{10}(0,6)\subset \mathbb{R}^{10}$, evaluated on a uniform grid of the two-dimensional cross-section $B_2(0,6)\times \{0\}^8$. In the figure at the left, a uniform distribution over $\Omega$ has been used in Algorithm \ref{alg: SGD}, whereas in the figure at the right, we used the radially uniform sampling defined in \eqref{radially unif sample}.}
    \label{fig:Eikonal 10-d sampling}
\end{figure}

\subsection{Supervised data in domain interior}
\label{subsec: adding supervised data}

Theorem \ref{thm: uniqueness finite-diff} gives a mesh-independent sufficient condition for critical points of the full residual functional to solve the difference equation. It does not guarantee convergence of Algorithm~\ref{alg: training}, which optimizes over neural-network parameters using sampled residuals. To investigate the effect of the stencil size on this optimization, we estimate the success rate for the Eikonal equation on different domains (see Section~\ref{sec:Eikonal-experiments} for the problem and settings). We fix $\alpha=2$ and consider
\[
(\delta_0, \delta_1, \delta_2) = (0.7, 0.2, 0.01).
\]
The value of $\delta$ is kept constant in each experiment.

We use the sign of the trained solution as an empirical success diagnostic: a run is counted as successful if it has the expected sign after the prescribed training epochs, or acquires that sign after an additional training session. This diagnostic does not by itself certify convergence to the viscosity solution or to a global minimiser. For the squared Eikonal Hamiltonian $H(p)=|p|^2-1$, $\alpha=2$ ensures monotonicity only on stencils with $\|p\|_\infty\leq1$.
The reported success rate is the fraction of runs satisfying this sign criterion.
Table \ref{tab:conv_cube_ball_with_without_gamma} shows a decrease in the success rate as $\delta$ decreases. This empirical training behavior is compatible with a mesh-independent condition for critical points of the full functional.

\begin{table}[h]
    \centering
    \caption{Sign-based empirical success rate for different domains and values of $\delta$}
    \medskip
    \begin{tabular}{l|c|c|l|l|l}
    \toprule
        Domain & $\alpha$ & $\delta$ & Success rate ($\Gamma = \emptyset$) & Success rate (localized $\Gamma$) & Success rate (uniform $\Gamma$)\\
        \cmidrule{1-6}
        Square & $2$ & $0.7$  & $1.00$ & $1.00$ & $1.00$\\
               &     & $0.2$  & $0.83 \pm 0.08$  & $1.00 - 0.03$ & $1.00 - 0.03$\\
               &     & $0.01$ & $0.48 \pm 0.09$  & $0.96 \pm 0.03$ & $1.00 - 0.03$\\
        \cmidrule{1-6}
        Ball   & $2$ & $0.7$  & $1.00$ & $1.00$ & $1.00$\\
               &     & $0.2$  & $0.91 \pm 0.05$  & $1.00 - 0.03$ & $1.00 - 0.03$\\
               &     & $0.01$ & $0.53 \pm 0.09$  & $0.96 \pm 0.03$ & $1.00 - 0.03$\\
        \cmidrule{1-6}
        Annulus & $2$ & $0.7$  & $1.00$ & $1.00$ & $1.00$\\
                &     & $0.2$  & $0.43 \pm 0.09$  & $0.96 \pm 0.03$ & $1.00 - 0.03$\\
                &     & $0.01$ & $0.02 \pm 0.02$  & $0.31 \pm 0.09$ & $0.75 \pm 0.08$\\
        \bottomrule
    \end{tabular}
    \label{tab:conv_cube_ball_with_without_gamma}
\end{table}

Now we investigate the effect of additional supervised data on an interior set $\Gamma\subset\Omega$. These data modify the parameterized loss and may improve its optimization. We compare two choices of ten data points: sampling in a small interior region, and sampling throughout $\Omega$. No fixed separation from $\partial\Omega$ is imposed on the latter sample.

We first run experiments where $\Gamma$ contains $10$ randomly selected points sampled from a small subdomain of $\Omega$.
The success rates are reported in the fifth column of Table~\ref{tab:conv_cube_ball_with_without_gamma}. 
We see that the added knowledge on $\Gamma$ enable a remarkable improvement in the success rate. Nevertheless, it is worth underlining that, as before, the shape of the domain seems to affect the results so that the positive influence of $\Gamma$ is mitigated, for instance, in the annulus case. 

Next, we run experiments where $\Gamma$ contains $10$ points sampled uniformly from 
the interior of $\Omega$. The success rates are reported in the last column of Table~\ref{tab:conv_cube_ball_with_without_gamma}.  We can appreciate an improvement in the success rate even on the annulus, confirming the dependency of the choice of $\Gamma$ on the shape of the domain $\Omega$. 

{ These experiments suggest that supervised data points in the domain interior can improve the overall data efficiency and generalization power of the trained NNs. In practice, one may use other pathwise Hamilton-Jacobi solvers, such as those based on Hopf-Lax formula \cite{darbon2016algorithms} (when applicable), to generate supervised data points. }

}

\section{Proofs of the main results}
\label{sec: finite-difference proof}

\subsection{Notation and preliminaries}
In this section, we only consider the $d$-dimensional cube $\Omega := (0,1)^d$.
For any $N\in \N$, fix $\delta = 1/N$, and consider the uniform Cartesian grids of the domains $\overline{\Omega}$ and $\Omega$, given by 
\begin{align}
\label{grid Omega bar}
&\overline{\Omega}_\delta := \delta\Z^d \cap \overline{\Omega}
= \left\{
x_\beta := \delta \beta \, : \ 
\beta\in \{ 0, 1, \ldots, N \}^d
\right\} \\
&\Omega_\delta := \delta\Z^d \cap \Omega
= \left\{
x_\beta := \delta \beta \, : \ 
\beta\in \{ 1, \ldots, N-1 \}^d
\right\}
\label{grid Omega interior}
\end{align}
From now on, we will use the multi-index notation $\beta := (\beta_1, \ldots, \beta_d) \in \{ 0, 1, \ldots, N \}^d$.
Furthermore, we denote by $\{ e_i\}_{i=1}^d$ the canonical basis of $\R^d$, so that $x_{\beta + e_i} = x_\beta + \delta e_i$ and $x_{\beta - e_i} = x_\beta - \delta e_i$.
We also introduce the index domains
\begin{equation}
\label{idx domain def}
\overline{\mathfrak{B}} := \{ 
0, \ldots , N
\}^d
\quad \text{and} \quad
\mathfrak{B} := \{ 
1, \ldots , N-1
\}^d.
\end{equation}
Then, we can write $\overline{\Omega}_\delta = \{ x_\beta \, : \ \beta \in \overline{\mathfrak{B}} \}$ and $\Omega_\delta = \{ x_\beta \, : \ \beta\in \mathfrak{B} \}$.

For a given a function $u\in C(\overline{\Omega})$, let us denote by $U$ the grid function associated to $u$ on the grid $\overline{\Omega}_\delta$, i.e.
$$
U := \left\{
U_\beta = u (x_\beta) \, : \ 
\beta \in \overline{\mathfrak{B}}
\right\}.
$$

Using the above notation, we can write the functional $\widehat{\mathcal{R}} (u)$ as
\begin{eqnarray}
\nonumber
\widehat{\mathcal{R}} (u) &=&
\delta^d \sum_{x\in \Omega_\delta}
\left[ \widehat{H}_\alpha (x, D_\delta^+ u(x) , D_\delta^- u(x)) \right]^2 \\
\nonumber
&=&
\delta^d \sum_{x\in \Omega_\delta}
\left[ H \left(x, \dfrac{D_\delta^+ u(x) + D_\delta^- u(x)}{2} \right) - \alpha \sum_{i=1}^d \left( \dfrac{D_\delta^+ u(x) - D_\delta^- u(x)}{2}\right)_i \right]^2  \\
\label{F(U) def}
&=& \delta^d \sum_{\beta \in \mathfrak{B}}
\left[ H \left(x_\beta, D_\delta U_\beta \right) - \alpha \sum_{i=1}^d D_{\delta, i}^2 U_\beta \right]^2  
=: F(U),
\end{eqnarray}
where, for any $\beta\in \mathfrak{B}$, $D_\delta U_\beta$ denotes the centred finite-difference approximation of the gradient, i.e.
\begin{equation}
\label{centred diff gradient grid}
D_\delta U_\beta = \left(  
\dfrac{U_{\beta + e_1} - U_{\beta-e_1}}{2\delta}, \ldots ,
\dfrac{U_{\beta + e_d} - U_{\beta-e_d}}{2\delta}
\right) \in \R^d,
\end{equation}
and, for any $i\in \{1, \ldots ,d\}, $ $D_{\delta, i}^2 U_\beta$ is given by
\begin{equation}
\label{finite diff second deriv}
D_{\delta, i}^2 U_\beta = \dfrac{U_{\beta + e_i} + U_{\beta - e_i} - 2 U_\beta}{2\delta}.
\end{equation}
{ Note that the notation $D_{\delta, i}^2$ does not correspond to the application of $D_{\delta, i}$ twice.}

In view of the assumption \eqref{cond H}, the functional $U \mapsto F(U)$ defined in \eqref{F(U) def} is differentiable in the space of grid functions on $\overline{\Omega}_\delta$. 
The main step in the proof of Theorem \ref{thm: uniqueness finite-diff} is to verify that if $\nabla F(U) = 0$, then
$$
H \left(x_\beta, D_\delta U_\beta \right) - \alpha \sum_{i=1}^d D_{\delta, i}^2 U_\beta  = 0,
\qquad \forall \beta \in \mathfrak{B}.
$$

In the next lemma, we compute the partial derivatives of $F(U)$ with respect to $U_\beta$, for all $\beta\in \mathfrak{B}$.
\begin{lemma}
    \label{lem: gradient of F(U)}
    Let $H$ satisfy \eqref{cond H}, $\alpha>0$ and, for any $N\in \N$ and $\delta = 1/N$, consider the grid $\overline{\Omega}_\delta$ as in \eqref{grid Omega bar}, and the index domains $\overline{\mathfrak{B}}$ and $\mathfrak{B}$ as in \eqref{idx domain def}. Let $F(U)$ be defined as in \eqref{F(U) def} for any grid function $U$ on $\overline{\Omega}_\delta$.
    Then we have
    $$
    \partial_{U_\beta} F(U) =
    - \delta^{d-1} \left( \sum_{i=1}^d \left(V_{\beta + e_i}^{(i)} W_{\beta+e_i} - V_{\beta - e_i}^{(i)} W_{\beta - e_i} \right)
    + \alpha  \sum_{i=1}^d \left(W_{\beta + e_i} + W_{\beta - e_i} - 2 W_\beta\right) \right) , \qquad \forall \beta\in \mathfrak{B},
    $$
    where
    \begin{equation}
    \label{W def}
    W_\beta := \begin{cases}
        H \left(x_\beta, D_\delta U_\beta \right) - \alpha \sum_{i=1}^d D_{\delta, i}^2 U_\beta,  & \forall \beta \in \mathfrak{B}, \\
        0 & \forall \beta\in \overline{\mathfrak{B}}\setminus \mathfrak{B},
    \end{cases}
    \end{equation}
    and for all $i\in \{1, \ldots, d\}$,
    \begin{equation}
    \label{V i def}
    V_\beta^{(i)} :=
    \begin{cases}
        \partial_{p_i} H \left( x_\beta, D_\delta U_\beta \right), & \forall \beta \in \mathfrak{B} \\
        0 & \forall \beta\in \overline{\mathfrak{B}}\setminus \mathfrak{B}.
    \end{cases}
    \end{equation}
\end{lemma}

\begin{proof}
Note that, in view of the definition of $F(U)$ in \eqref{F(U) def} and the definition of $W = \{ W_\beta\}_{\beta\in \overline{\mathfrak{B}}}$ in \eqref{W def}, we have
$$
F(U) = \delta^d\sum_{\beta \in \overline{\mathfrak{B}}} W_\beta^2.
$$
For any $\beta\in \mathfrak{B}$, the term $U_\beta$ appears only in the terms associated to the index $\beta$ and the indexes $\beta\pm e_i$ for all $i\in \{1, \ldots, d\}$.
    Moreover, we note that, in view of \eqref{centred diff gradient grid}, the term $H(x_\beta, D_\delta U_\beta)$ is independent of $U_\beta$, and for any $i$, the terms 
    $H(x_{\beta \pm e_i}, D_\delta U_{\beta \pm e_i})$ depend on $U_\beta$ only through the $i$-th component of the vector $D_\delta U_{\beta \pm e_i}$.

    Hence, we can write $\partial_{U_\beta} F(U)$ as
    \begin{eqnarray}
    \label{partial F computation}
        \partial_{U_\beta} F(U) &=&
        -2\delta^d \alpha W_\beta \sum_{i=1}^d \partial_{U_\beta} \left( D_{\delta,i}^2 U_\beta \right) \\
        && +
        2\delta^d \sum_{i=1}^d W_{\beta + e_i}
        \Big(
        \underbrace{\partial_{p_i} H \left( x_{\beta+e_i}, D_\delta U_{\beta+e_i} \right)}_{V_{\beta + e_i}^{(i)}} \partial_{U_\beta} \left( D_\delta U_{\beta + e_i}  \right)_i
        - \alpha \partial_{U_\beta} \left(  D_{\delta,i}^2 U_{\beta + e_i}\right)
        \Big) \nonumber \\
        && +
        2\delta^d \sum_{i=1}^d W_{\beta-e_i}
        \Big(
        \underbrace{\partial_{p_i} H \left( x_{\beta - e_i} , D_\delta U_{\beta - e_i}  \right)}_{V_{\beta - e_i}^{(i)}} \partial_{U_\beta} \left( D_\delta U_{\beta - e_i} \right)_i
        - \alpha \partial_{U_\beta} \left( D_{\delta, i}^2 U_{\beta-e_i} \right)
        \Big). \nonumber
    \end{eqnarray}

    Using \eqref{finite diff second deriv}, we can compute
$$
\partial_{U_\beta} \left( D_{\delta,i}^2 U_\beta \right) = \partial_{U_\beta} \left( \dfrac{U_{\beta + e_i} + U_{\beta-e_i} - 2U_\beta}{2\delta} \right) = - \dfrac{1}{\delta}.
$$
Next, for any $i\in \{1, \ldots, d\}$ such that $\beta + e_i\in \mathfrak{B}$ (note that the terms such that $\beta + e_i \not\in \mathfrak{B}$ do not matter since $W_{\beta + e_i}$ would be zero), we can use \eqref{centred diff gradient grid} and \eqref{finite diff second deriv} to compute
\begin{align*}
    & \partial_{U_\beta} \left( D_\delta U_{\beta + e_i} \right)_i = \partial_{U_\beta} \left(\dfrac{U_{\beta + 2 e_i} - U_\beta}{2\delta}\right) = - \dfrac{1}{2\delta}, \\
    & \partial_{U_\beta} \left( D_{\delta, i}^2 U_{\beta + e_i} \right) = \partial_{U_\beta} \left( \dfrac{U_{\beta + 2e_i} + U_{\beta} - 2 U_{\beta + e_i}}{2\delta} \right) = \dfrac{1}{2\delta}.
\end{align*}
Similarly, for any $i\in \{1, \ldots, d\}$ such that $\beta - e_i\in \mathfrak{B}$, we can compute
\begin{align*}
    & \partial_{U_\beta} \left( D_\delta U_{\beta - e_i} \right)_i = \partial_{U_\beta} \left( \dfrac{U_\beta - U_{\beta-2e_i}}{2\delta} \right) = \dfrac{1}{2\delta}, \\
    & \partial_{U_\beta} \left( D_{\delta,i}^2 U_{\beta - e_i} \right) = \partial_{U_\beta} \left( \dfrac{U_\beta + U_{\beta - 2e_i} - 2 U_{\beta-e_i}}{2\delta} \right) = \dfrac{1}{2\delta}.
\end{align*}

By plugging all these values in \eqref{partial F computation}, we obtain
\begin{eqnarray*}
    \partial_{U_\beta} F(U) &=&  \delta^{d-1} \sum_{i=1}^d 2\alpha W_\beta + \delta^{d-1} \sum_{i=1}^d \left(  -V_{\beta+e_i}^{(i)} - \alpha \right) W_{\beta + e_i} + \delta^{d-1} \sum_{i=1}^d \left( V_{\beta - e_i}^{(i)} - \alpha \right) W_{\beta - e_i},
\end{eqnarray*}
and the conclusion follows by regrouping the terms in the above sum.
\end{proof}

\subsection{Proof of Theorem \ref{thm: uniqueness finite-diff} and Corollary \ref{cor: uniqueness boundary data}}
\label{subsec: proof of thm discrete func}

We are now in a position to carry out the proof of Theorem \ref{thm: uniqueness finite-diff},
which consists of two steps: 
\begin{enumerate}
    \item First we prove that if $u\in C(\overline{\Omega})$ is a critical point of $\widehat{\mathcal{R}}(\cdot)$, then the associated grid function $U$ on $\overline{\Omega}_\delta$ satisfies $\nabla F(U) = 0$, where $F(U)$ is given by \eqref{F(U) def}.
    \item Then, we use Lemma \ref{lem: gradient of F(U)} to prove that if $u$ is $L$-Lipschitz and $\alpha$ satisfies \eqref{condition theorem}, then $\nabla F(U) = 0$ implies that $\widehat{H}_\alpha (x, D_\delta^+ u(x),  D_\delta^- u(x)) = 0$ for all $x\in \Omega_\delta$.
\end{enumerate}

\begin{proof}
    \textit{\underline{Step 1:}} Let $u\in C(\overline{\Omega})$. If $u$ is a critical point of $\widehat{\mathcal{R}} (\cdot)$, then it must hold that
    \begin{equation}
    \label{grad equal 0 lim}
    \lim_{\varepsilon \to 0} \dfrac{\widehat{\mathcal{R}} (u + \varepsilon \phi) - \widehat{\mathcal{R}} (u)}{\varepsilon} = 0,
    \qquad
    \text{for any function $\phi\in C(\overline{\Omega})$.}
    \end{equation}
   Now, for any grid function $\Phi$ on $\overline{\Omega}_\delta$, let $\phi \in C(\overline{\Omega})$ be a function such that $\phi(x_\beta) = \Phi_\beta$ for all $\beta\in \overline{\mathfrak{B}}$. In other words, consider a continuous function $\phi$ such that the associated grid function is $\Phi$.
   In view of \eqref{grad equal 0 lim}, and using the function $F(\cdot)$ as defined in \eqref{F(U) def}, we deduce that
   $$
   0 = \lim_{\varepsilon \to 0} \dfrac{\widehat{\mathcal{R}} (u + \varepsilon \phi) - \widehat{\mathcal{R}} (u)}{\varepsilon} = 
   \lim_{\varepsilon \to 0} \dfrac{F (U + \varepsilon \Phi) - F (U)}{\varepsilon}
    = \nabla F(U) \cdot \Phi.
   $$
   Since this holds for any grid function $\Phi$, we deduce that $\nabla F(U) = 0$.

   \textit{\underline{Step 2:}} By Lemma~\ref{lem: gradient of F(U)}, the interior components of $\nabla F(U)=0$ give
   $$
   (A_N(V)+\alpha\Delta_N)W=0,
   $$
   where $W$ and $V$ are defined in \eqref{W def} and \eqref{V i def}. On vectors indexed by $\mathfrak{B}$, with zero extension to $\overline{\mathfrak{B}}\setminus\mathfrak{B}$, define
   \begin{align*}
   (A_N(V)Z)_\beta&:=\sum_{i=1}^d
   \left(V_{\beta+e_i}^{(i)}Z_{\beta+e_i}-V_{\beta-e_i}^{(i)}Z_{\beta-e_i}\right),\\
   (\Delta_N Z)_\beta&:=\sum_{i=1}^d
   \left(Z_{\beta+e_i}+Z_{\beta-e_i}-2Z_\beta\right).
   \end{align*}
   Since $u$ is $L$-Lipschitz, \eqref{centred diff gradient grid} yields
   $$
   \|D_\delta U_\beta\|_\infty\leq L,
   \qquad |V_\beta^{(i)}|\leq C_{H,\infty}(L)<\alpha.
   $$
   Set $B:=A_N(V)+\alpha\Delta_N$. Its transpose acts on a zero-extended vector $\Phi$ as
   \begin{align*}
   (B^T\Phi)_\beta
   =\sum_{i=1}^d\bigl[
   &(\alpha-V_\beta^{(i)})(\Phi_{\beta+e_i}-\Phi_\beta)\\
   {}+&(\alpha+V_\beta^{(i)})(\Phi_{\beta-e_i}-\Phi_\beta)
   \bigr].
   \end{align*}
   Suppose $B^T\Phi=0$. If $\Phi$ has a positive maximum at an interior node, every difference in the last display is nonpositive there, and all coefficients are strictly positive. Hence every neighboring value equals the same maximum. Iterating along a grid path to the boundary contradicts the zero boundary values. Applying the same argument to $-\Phi$ gives $\Phi=0$.
   Thus $B^T$, and therefore $B$, is invertible. Consequently $W=0$, which proves \eqref{finite-diffs Lax-Friedrichs intro}. If $N=1$, the interior grid is empty and the conclusion is immediate.
\end{proof}

Next, we give the proof of Corollary \ref{cor: uniqueness boundary data}.
The proof is almost identical to that of Theorem \ref{thm: uniqueness finite-diff}, with the only difference being that one has to consider the functional's partial derivatives with respect to the boundary points.

\begin{proof}[Proof of Corollary \ref{cor: uniqueness boundary data}]
    First, for a critical point $u$ as in the statement of the corollary, write the functional $\widehat{\mathcal{J}} (u)$ as
    $$
    \widehat{\mathcal{J}} (u) = F(U) + \gamma G (U),
    $$
    where $U$ includes the values at all nodes in $\Omega_\delta\cup\partial\Omega_\delta$, and $F(U)$ is defined by \eqref{F(U) def} (and is independent of any extra boundary nodes). Set
    $$
    \mathfrak{B}_\partial:=\{\beta\in\Z^d:\delta\beta\in\partial\Omega_\delta\}.
    $$
    The boundary term is
    $$
    G (U) := \delta^{d-1} \sum_{\beta\in \mathfrak{B}_\partial} \left( U_\beta - g (x_\beta) \right)^2.
    $$
    We use the multi-index notation introduced in \eqref{idx domain def}.

    Using precisely the same argument as in step 1 in the proof of Theorem \ref{thm: uniqueness finite-diff}, we have that if $u$ is a critical point of $\widehat{\mathcal{J}} (u)$, then it must hold that $\nabla F(U) + \gamma \nabla G(U) = 0$.

    Since $G(U)$ does not depend on $U_\beta$ for any $\beta\in \mathfrak{B}$, we have that $u$ being a critical point of $\widehat{\mathcal{J}} (u)$ implies that
    $$
    \partial_{U_\beta} F(U) =  - \delta^{d-1} \left( \sum_{i=1}^d \left(V_{\beta + e_i}^{(i)} W_{\beta+e_i} - V_{\beta - e_i}^{(i)} W_{\beta - e_i} \right)
    + \alpha  \sum_{i=1}^d \left(W_{\beta + e_i} + W_{\beta - e_i} - 2 W_\beta\right) \right) = 0 , \qquad \forall \beta\in \mathfrak{B},
    $$
    where $W_\beta$ and $V_\beta^{(i)}$ are given as in \eqref{W def} and \eqref{V i def}. Now, using the step 2 in the proof of Theorem \ref{thm: uniqueness finite-diff}, we obtain that, if \eqref{condition theorem} holds, then $W_\beta = 0$ for all $\beta \in \mathfrak{B}$.

    Since $W_\beta = 0$ for all $\beta \in \mathfrak{B}$, we have that
    $$
    \partial_{U_\beta} F(U) + \gamma \partial_{U_\beta} G(U) = \gamma \partial_{U_\beta} G(U) = 2 \gamma \delta^{d-1} \left(U_\beta - g(x_\beta)  \right), \qquad \forall \beta\in \mathfrak{B}_\partial.
    $$
    Hence, $\partial_{U_\beta} F(U) + \gamma \partial_{U_\beta} G(U) = 0$ implies that $U_\beta = g(x_\beta)$ for all $\beta \in \mathfrak{B}_\partial$. This, combined with $W_\beta = 0$ for all $\beta \in \mathfrak{B}$ concludes the proof.
\end{proof}

\subsection{Proof of Theorem \ref{thm: main result}}
\label{subsec: proof main thm}

The main idea of the proof is to re-write the functional $\mathcal{R}(\cdot)$ in such a way that we can apply the conclusion of Theorem \ref{thm: uniqueness finite-diff} for the discrete functional $\widehat{\mathcal{R}} (\cdot)$.

\begin{proof}
Let $N\geq2$, $\delta=1/(N-1)$, $Q:=(-\delta,0)^d$, and $\mathfrak{B}:=\{1,\ldots,N-1\}^d$. Up to cell boundaries, which have measure zero, the sets $Q+\delta\beta$, $\beta\in\mathfrak{B}$, partition $\Omega$. Therefore
\begin{equation}
\label{R continuous computation}
\mathcal{R}(u)=\int_Q\sum_{\beta\in\mathfrak{B}}W_\beta(z)^2\,dz,
\qquad
W_\beta(z):=\widehat{H}_\alpha\bigl(z+\delta\beta,D_\delta^+u(z+\delta\beta),D_\delta^-u(z+\delta\beta)\bigr).
\end{equation}
For $z\in Q$, the interior stencil nodes are $z+\delta\beta$, $\beta\in\mathfrak{B}$; their neighboring nodes lie in $\Omega_\delta^{\mathrm{ext}}$. Define
$$
V_\beta^{(i)}(z):=\partial_{p_i}H\bigl(z+\delta\beta,D_\delta u(z+\delta\beta)\bigr),
\qquad D_\delta u:=\frac{D_\delta^+u+D_\delta^-u}{2},
$$
and set $W_\beta(z)=V_\beta^{(i)}(z)=0$ for indices outside $\mathfrak{B}$ when forming the adjoint residual equation. The Lipschitz assumption gives $|V_\beta^{(i)}(z)|\leq C_{H,\infty}(L)<\alpha$.

Fix $\beta\in\mathfrak{B}$ and $\eta\in C_c^\infty(Q)$. Choose the variation $v$ to equal $\eta(x-\delta\beta)$ on $Q+\delta\beta$ and zero elsewhere. Then $v$ is continuous on $\Omega_\delta^{\mathrm{ext}}$ and has compact support in $\Omega$. Differentiating \eqref{R continuous computation} and using the computation of Lemma~\ref{lem: gradient of F(U)}, without the prefactor $\delta^d$, stationarity gives
$$
0=-\frac1\delta\int_Q
\bigl[(A_N(V(z))+\alpha\Delta_N)W(z)\bigr]_\beta\eta(z)\,dz.
$$
Differentiation under the integral is justified locally on the compact support of $\eta$ by \eqref{cond H}; only finitely many translated stencil nodes are involved. By the fundamental lemma of the calculus of variations,
$$
(A_N(V(z))+\alpha\Delta_N)W(z)=0
\qquad\text{for almost every }z\in Q.
$$
The transpose maximum-principle argument in Step~2 of the proof of Theorem~\ref{thm: uniqueness finite-diff} applies to every such finite shifted grid. It gives $W(z)=0$ for almost every $z\in Q$. Hence the residual vanishes almost everywhere in $\Omega$. Since $H$ and $u$ are continuous, the residual is continuous in $\Omega$, and therefore vanishes everywhere there, as claimed.
\end{proof}

\section{Summary and perspectives}
\label{sec: summary and perspectives}
In this paper, we propose to solve the boundary value problems for a class of Hamilton-Jacobi equations using Deep Learning. 
Artificial neural networks are trained by typical gradient descent-based algorithms, optimising for the least square problem defined by the residual of a consistent and monotone finite difference scheme for the equation.  
We analyze the critical points of the least square problem in the space of continuous functions and give a mesh-independent sufficient condition, within a prescribed Lipschitz class, for their residuals to vanish. The resulting finite-difference solutions approximate the viscosity solution under the usual consistency, monotonicity, and convergence assumptions. 
The sufficient condition involves the Lax-Friedrichs parameter and a bound on the derivatives of the Hamiltonian. We also study the effects of additional supervised data empirically.
We demonstrated numerically that the proposed approach effectively solves Hamilton-Jacobi equations in higher dimensions under a favourable setup.

An important point is that artificial neural networks are used instead of grid functions. Thus, one may take advantage of neural networks' approximation power for the smooth parts of a solution and the scaling property of Monte-Carlo sampling involved in the training procedure.
Thus, under favourable conditions, it is reasonable to expect that the proposed approach would suffer less the so-called ``curse of dimensionality" than the classical algorithms.

But during the training iterations, one must ensure that 
the (high dimensional) computational domain is sufficiently ``covered" --- enough to effectively follow the characteristics from the boundaries to everywhere in the domain and ``monitor" how the monotone scheme resolves the collision of the characteristics. This is crucial for computing accurate approximation of the viscosity solutions.

It is widely acknowledged that the choice of collocation points 
can determine the success of a PINN approach.  
This is also true for the proposed approach, given a limited training budget. 
In higher dimensions, due to the ``concentration of measure" errors, the errors ``deeper" in the interior of $\Omega$ are significantly ``discounted."
However, reducing the errors in the interior may be essential to constructing approximate solutions with the proper global structure (such as convexity).
Therefore,  uniform sampling (relative to the Lebesque measure) in training or assigning additional ``supervised" data may be sub-optimal. 
We plan to develop adaptive algorithms for the proposed approach in the near future.

We point out that it is possible to use high-resolution schemes. 
If a gradient descent-based optimization scheme relying automatic differentiation is employed, the numerical scheme selected must be a differentiable function with respect to the variables defined by the scheme's stencil.
This formerly precludes the use of higher order essentially non-oscillatory (ENO) schemes \cite{Osher-Shu:1993}. However, the weighted essentially oscillatory (WENO) schemes \cite{Jiang-Peng-WENO:2000} can be considered.
A potential advantage may come from the wider stencils of such a scheme that have a large domain of influence in the optimization model. 

Finally, we remark that the idea of using a convergent numerical method to form least square principles is not limited to only finite difference methods. It is possible to develop such an algorithm using the finite element methodology. In that case, we envision that proper mesh refinement strategies may be involved during training to enhance error control.

\section*{Acknoledgement}
Esteve-Yagüe is partly supported by the EPSRC grant EP/V029428/1, and by the Ram\'on y Cajal 2022 grant RYC2022-035966-I.
Tsai is partially supported by NSF grant DMS-2110895 and Army Research Office Grant W911NF2320240.

\bibliographystyle{abbrv}
\bibliography{refs}

\end{document}